\documentclass[journal]{IEEEtran}

\RequirePackage{amsthm, amsmath, amsfonts, amssymb, mathtools}
\usepackage{bm}
\usepackage{booktabs}
\usepackage{graphicx}
\usepackage{mathrsfs}
\usepackage[mathcal]{euscript}
\usepackage{bbold}
\usepackage{enumitem}
\usepackage{makecell}
\usepackage{url}
\usepackage{xcolor}
\definecolor{darkblue}{rgb}{0.0, 0.0, 0.55}
\definecolor{chicago-maroon}{RGB}{128,0,0}
\definecolor{northwestern-purple}{RGB}{82,0,99}
\usepackage[numbers,sort&compress]{natbib}

\usepackage[bookmarks, plainpages = false, colorlinks = true, citecolor = darkblue, linkcolor = chicago-maroon, anchorcolor = red, urlcolor = chicago-maroon]{hyperref}
\RequirePackage{hypernat}
\usepackage[capitalise,nameinlink]{cleveref}
\usepackage{array}
\usepackage{subfigure}
\usepackage{stfloats}
\usepackage{verbatim}
\providecommand{\texorpdfstring}[2]{#1}
\usepackage{orcidlink}
\theoremstyle{plain}
\newtheorem{theorem}{Theorem}[section]
\newtheorem{lemma}{Lemma}[section]
\newtheorem{proposition}{Proposition}[section]
\newtheorem{corollary}{Corollary}[section]

\newtheorem{remark}{Remark}[section]
\theoremstyle{definition}
\newtheorem{definition}{Definition}[section]
\newtheorem{example}{Example}[section]

\newtheorem{assumption}{Assumption}[section]

\newcommand\independent{\protect\mathpalette{\protect\independenT}{\perp}}
\def\independenT#1#2{\mathrel{\rlap{$#1#2$}\mkern2mu{#1#2}}}
\def\nuis{\mathrm{nuis}}
\def\Hodges{\mathrm{Hodges}}
\def\Frechet{Fr\'{e}chet}

\def\Hajek{H\'{a}jek}
\def\Cramer{Cram\'{e}r}
\def\E{\mathbb{E}}
\def\P{\mathbb{P}}

\def\Q{\mathbb{Q}}
\def\p{\mathsf{p}}

\def\calM{\mathcal{M}}
\def\calP{\mathcal{P}}
\def\bbS{\mathbb{S}}

\def\bbM{\mathbb{M}}
\def\bbI{\mathbb{I}}

\def\bbV{\mathbb{V}}
\def\bbG{\mathbb{G}}
\def\bbR{\mathbb{R}}
\def\bbX{\mathbb{X}}

\def\calA{\mathcal{A}}
\def\hat{\widehat}
\def\tilde{\widetilde}
\def\diff{\mathrm{d}}

\def\Exp{\mathsf{Exp}}

\def\bbL{\mathbb{L}}
\def\T{\mathrm{T}}
\def\calR{\mathcal{R}}
\def\Cov{\mathrm{cov}}
\def\Var{\mathrm{var}}

\def\r{\mathsf{r}}
\def\calN{\mathcal{N}}
\def\Bias{\mathsf{bias}}

\def\L{\mathsf{L}}
\def\bbB{\mathbb{B}}

\def\rmi{\mathrm{i}}
\def\frakc{\mathfrak{c}}
\def\eps{\epsilon}

\def\id{\mathrm{id}}
\def\linspan{\mathrm{span}}

\def\IF{\mathsf{IF}}
\def\eff{\mathrm{eff}}
\def\s{\mathsf{s}}
\def\S{\mathsf{S}}
\def\bX{\mathbf{X}}
\def\calC{\mathcal{C}}

\def\calT{\mathcal{T}}
\def\calK{\mathcal{K}}

\def\calX{\mathcal{X}}

\def\calB{\mathcal{B}}

\usepackage{bibunits}
\makeatletter
\newcounter{mybibunit}
\AtBeginEnvironment{bibunit}{\stepcounter{mybibunit}}
\renewcommand{\hyper@natlinkstart}[1]{%
  \Hy@backout{#1}%
  \hyper@linkstart{cite}{cite.\themybibunit @#1}%
  \def\hyper@nat@current{#1}%
}
\renewcommand{\hyper@natlinkbreak}[2]{%
  \hyper@linkend#1\hyper@linkstart{cite}{cite.\themybibunit @#2}%
}
\renewcommand{\hyper@natanchorstart}[1]{%
  \Hy@raisedlink{\hyper@anchorstart{cite.\themybibunit @#1}}%
}
\makeatother

\title{Toward an Asymptotic Efficiency Theory on Regular Parameter Manifolds}

\author{Lvfang~Sun, Zhenhua~Lin\orcidlink{0000-0003-1690-9713}, and Lin~Liu\orcidlink{0000-0002-9883-7962}
\thanks{Lvfang Sun and Zhenhua Lin are with the Department of Statistics and Data Science, National University of Singapore, Singapore (e-mail: \href{sunlyufang@u.nus.edu}{sunlyufang@u.nus.edu}; \href{linz@nus.edu.sg}{linz@nus.edu.sg}). Lin Liu is with the Institute of Natural Sciences, MOE--LSC, School of Mathematical Sciences, CMA--Shanghai, SJTU--Yale Joint Center for Biostatistics and Data Science, Shanghai Jiao Tong University, Shanghai, China (e-mail: \href{linliu@sjtu.edu.cn}{linliu@sjtu.edu.cn}). Lin Liu and Zhenhua Lin are corresponding authors and are alphabetically ordered with equal contribution.

Lin Liu's research is funded by the National Key R\&D Program of China Project Number 2025YFA1016700, NSFC Grant No.~12471274 and Science and Technology Talent and Platform Program of Yunnan Province Grant No.~202605AF35007. Zhenhua Lin's research is partially funded by the NUS startup grant A-0004816-00-00. 

The authors would like to thank Tom Kolokotrones, Shuyang Ling, Rajarshi Mukherjee, Adam Quinn Jaffe, Thomas Richardson, James Robins, Aad van der Vaart, Linbo Wang, Yingcun Xia, and Qingyuan Zhao for insightful discussions. 

The authors sincerely appreciate three anonymous referees and the associate editor for their careful reading and highly constructive and useful comments and critiques (in particular on the distinction between reference to a collection of paths and reference to the tangent space generated by the collection of paths) that significantly improved our paper. 

MSC2020 subject classifications: Primary 62E20; secondary 62R30.}}

\begin{document}
\maketitle

\begin{abstract}
Asymptotic efficiency theory is one of the pillars in the foundations of modern mathematical statistics. Not only does it serve as a rigorous theoretical benchmark for evaluating statistical methods, but it also sheds light on how to develop and unify novel statistical procedures. For example, the calculus of influence functions has led to many important statistical breakthroughs in the past decades. Responding to the pressing challenge of analyzing increasingly complex datasets, particularly those with non-Euclidean/nonlinear structures, many novel statistical models and methods have been proposed in recent years. However, the existing efficiency theory is not always readily applicable to these cases, as the theory was developed, for the most part, under the often neglected premise that both the sample space and the parameter space are normed linear spaces. As a consequence, efficiency results outside normed linear spaces are quite rare and isolated, obtained on a case-by-case basis. This paper aims to develop a more unified asymptotic efficiency theory, allowing the sample space, the parameter space, or both to be Riemannian manifolds satisfying certain regularity conditions. We build a vocabulary that helps translate essential concepts in efficiency theory from normed linear spaces to Riemannian manifolds, such as (locally) regular estimators, differentiable functionals, etc. Efficiency bounds are established under conditions parallel to those for normed linear spaces. We also demonstrate the conceptual advantage of the new framework by applying it to two concrete examples in statistics: the population \Frechet{} mean and the regression coefficient vector of Single-Index Models.
\end{abstract}

\begin{IEEEkeywords}
Asymptotic Efficiency Theory, Semiparametric Theory, Riemannian Manifolds, Geometric Statistics, \Frechet{} Mean, Influence Functions
\end{IEEEkeywords}

\newpage

\section{Introduction}
\label{sec:intro}

\IEEEPARstart{A}{symptotic} efficiency theory (abbreviated as ``efficiency theory'' in the sequel) is one of the central pillars of the foundations of classical and modern mathematical statistics \cite{stein1956efficient, le1960locally, hajek1970characterization, begun1983information, schick1986asymptotically, le1986asymptotic, bickel1998efficient, van1988statistical, van1991asymptotic, van1991differentiable, ibragimov1991asymptotically, le2000asymptotics, van2002part, hallin2003semi, ai2003efficient, pfanzagl2011parametric, chen2018overidentification, ibragimov2013statistical}. Efficiency in what follows is understood to be \textit{\`{a} la} Le Cam, H\'{a}jek, Ibragimov, and Has'minskii (we will simply use ``efficiency theory \textit{\`{a} la} Le Cam'' in the sequel to simplify our presentation). Theoretically, it lays a foundational framework under which the optimality of statistical procedures can be rigorously defined and assessed \cite{wolfowitz1965asymptotic, van1997superefficiency}. For example, efficiency theory has helped reconcile the intriguing paradox between Fisher's conjecture on the asymptotic optimality of maximum likelihood estimators (MLE) and the super-efficiency phenomenon revealed by Hodges' estimator \cite{van1997superefficiency}. Practically, several important concepts rooted in this theory, such as regular and asymptotic linear (RAL) estimators and (efficient) influence functions or (efficient) scores, have been instrumental in shedding light on the development of novel estimators \cite{robins1994estimation} and the unification of existing methods \cite{graham2026towards} in substantive fields such as epidemiology and economics. The celebrated augmented inverse probability weighting (AIPW) or double machine learning estimator in the causal inference and missing data literature \cite{robins1994estimation, chernozhukov2018double, van2003unified, van2006targeted} stands out as a poster-child in this regard (see Remark~\ref{rem:AIPW} toward the end of our paper).

A large fraction of statistical problems can be summarized as follows. Given a sample of observations, denoted abstractly as $\bX_{n}$ indexed by a natural number $n$, that is drawn from some unknown $\P$ belonging to a collection $\calP$ of probability distributions, often referred to as the \emph{statistical model}, the statistician is interested in learning about some functional of $\P$ (referred to as the \emph{parameter of interest}), $\chi \equiv \chi (\P)$. Efficiency theory, in a nutshell, concerns the following question: as $n \rightarrow \infty$, what is the optimal limiting variance of $r_{n} (\hat{\chi} - \chi)$, for an estimator $\hat{\chi}$ of $\chi$, upon appropriately rescaled by a diverging sequence $r_{n}$, called ``norming rate operators'' in \cite{van2002part}, such that $r_{n} (\hat{\chi} - \chi)$ is centered and asymptotic normal? In many cases, the index $n$ represents the sample size of repeated i.i.d. draws from $\P$ and then $r_{n} = n^{1 / 2}$.

When it is natural or convenient to parameterize $\P$ by $\theta \in \Theta$, with $\Theta$ denoting the parameter space, we also write $\chi \equiv \chi (\P_{\theta}) \equiv \psi (\theta) \equiv \psi$. $\theta$ is allowed to be infinite-dimensional, in order to incorporate nonparametric statistical models. In the most general form, efficiency theory has been established under the following predicates: 
\begin{enumerate}[label = (\arabic*)]
\item $\chi$, or equivalently $\psi$, takes values in a normed linear space $(\bbB, \Vert \cdot \Vert)$, with Banach spaces (e.g. $\bbR$) being the most general space that had been dealt with to date \cite{van1988statistical};
\item $\P \in \calP$ is smooth in the sense of being differentiable in quadratic mean (DQM) (see Definition~\ref{def:par dqm} later);
\item $\chi (\P)$ is a differentiable functional as defined in \cite{van1991differentiable}.
\end{enumerate}

Due to rapidly evolving data-collection technologies, statisticians are confronted with datasets with increasingly complex geometric structures. Relevant examples include covariance matrices \cite{yuan2012local,lin2019riemannian}, probability measures \cite{petersen2021wasserstein, chen2023wasserstein, lin2023causal}, networks \cite{dubey2019frechet}, and general manifolds \cite{hallin2024quantiles}, among many others. The importance of the analysis of such complex data structures has recently been highlighted in \cite{marron2014overview, wang2007object}. Such data also arise in numerous applications, ranging from diffusion tensor imaging (DTI) to functional connectivity analysis and developmental single-cell biology \cite{schiebinger2019optimal}. Also see Section 3.3.3 of \cite{cinelli2025challenges} for a related discussion focusing on causal inference.

To address these challenges, a broad spectrum of research has focused on developing statistical concepts suitable for nonlinear spaces. For instance, the notion of \Frechet{} mean has become a cornerstone in the analysis of random objects in general metric spaces \cite{bhattacharya2003large, bhattacharya2005large}; and recently \cite{hallin2024quantiles} studied quantiles and quantile regression for manifold-valued data. A parallel line of work \cite{pennec1999probabilities, pennec2006intrinsic} also develops some basic probabilistic tools on Riemannian manifolds, such as the concept of bias, covariance, and asymptotic normality. The above progress has led to the development of \Frechet{} regression \cite{petersen2019frechet}, where conditional \Frechet{} means are estimated as responses to Euclidean or manifold-valued predictors.

Building on these geometric foundations, various new statistical methods have been developed to accommodate nonlinear sample spaces, such as shape analyses \cite{harms2023geometry} and $M$-estimation \cite{brunel2023geodesically,brunel2024concentration}. The scope of regression models for responses and predictors has broadened significantly, beginning with
relatively simple spaces such as a circle or sphere \cite{fisher1993statistical}, then onto smooth manifolds \cite{cheng2013local, yuan2012local, jeon2022nonparametric, steinke2010nonparametric, davis2010population, cornea2017regression, thomas2013geodesic, zhong2024uncertainty} and even more general metric spaces \cite{marron2021object, petersen2019frechet,lin2021total}. Generalizations to regression models and dimension reduction methods can also be found in \cite{lin2023additive, huckemann2010intrinsic, qiu2024random, zhang2024dimension}. Examples of such manifold-valued modeling abound in modern data applications. As a concrete one, it can arise in the study of age-related changes in corpus callosum shape,  where each observed response is a shape object lying in a non-Euclidean manifold, namely Kendall’s shape space. To model this age effect, geodesic regression \cite{thomas2013geodesic}, the manifold analogue of simple linear regression,  was introduced, in which the fitted trajectory is 
$\gamma(t)=\Exp_{p}(tv)$ 
for a scalar predictor $t$ (e.g.,  age), where $p\in \bbM$ is a starting point on the manifold, $v$ is a tangent-vector direction, and $\Exp_p$ denotes the Riemannian exponential map. In this formulation, both $p$ and $v$ are manifold-valued parameters: the point $p$ represents a baseline shape on the manifold, while the direction $v$ describes the direction of shape change with the predictor and lies in the tangent bundle, itself a smooth manifold. We also refer readers to Remark~\ref{rem:Hessian} later in Section~\ref{sec:M-estimation} for more related examples.

In light of the recent development on statistical methods and theory over nonlinear spaces, the following inquiry naturally arises: 
\begin{quote}
\emph{Can we develop an efficiency theory for parameters lying in a possibly nonlinear space, \textit{\`{a} la} Le Cam? If so, how?}
\end{quote}
The goal of this paper is to advance the efficiency theory to a point closer to addressing the above questions. However, there are challenges in bringing this research program to fruition. In the aforementioned predicates of efficiency theory, (1) is often taken for granted. However, due to the absence of linear structures, ``+'' or ``-'' is generally not defined for nonlinear parameter spaces. To make progress, we consider a particular type of nonlinear parameter spaces, Riemmanian manifolds. We coin such parameter spaces as parameter manifolds. Though generally nonlinear, parameter manifolds still have a rich enough structure such as the associated tangent spaces and tangent bundle to leverage. We leave extensions to more general nonlinear spaces, such as metric spaces or infinite-dimensional Banach manifolds, to future work. We will point out a possible idea in Remark~\ref{rem:nonlinear}.

\subsection{Statistical Motivations}
\label{sec:motivation}

\subsubsection{The echoing issue of super-efficiency}
\label{sec:hodges}
\leavevmode

For the problem of estimating the normal mean over the Euclidean space, Hodges' super-efficient estimator dominates MLE under $L^{2}$ loss in a pointwise manner, thereby challenging Fisher's conjecture regarding the optimality of MLE. Not surprisingly, super-efficiency also occurs when the parameter space is nonlinear. Consider a Riemannian manifold $(\bbM, d)$, given a homogeneous Gaussian measure $\P$ on $\bbM$, as defined in Example~\ref{ex:riemannian gaussian} later, we assume to have access to a sample of independent and identically distributed (i.i.d.) draws from $\P$: $\bX_{n} \coloneqq \{X_{i}\}_{i = 1}^{n} \overset{\rm i.i.d.}{\sim} \P$. It is often of interest to estimate the population \Frechet{} mean $\mu \coloneqq \arg\min_{x \in \bbM} \E [d (X, x)^2]$. The sample \Frechet{} mean estimator based on $\bX_{n}$, defined as $\hat{\mu}_{n} \coloneqq \arg\min_{x \in \bbM} \sum_{i = 1}^n d (X_i, x)^2$, is a natural estimator of $\mu$, and also happens to be the MLE. The corresponding Hodges' estimator of $\mu$ can then be defined as
\begin{equation}
\label{hodges}
\hat{\mu}_{n, \Hodges} \coloneqq \left\{ 
\begin{aligned}
\hat{\mu}_n &,& d (\hat{\mu}_{n}, \mu_{0}) > n^{-\frac{1}{4}}, \\
\mu_{0} &,& d (\hat{\mu}_{n}, \mu_{0}) \leq n^{-\frac{1}{4}},
\end{aligned} \right.
\end{equation}
where $\mu_{0}$ is an arbitrarily chosen reference point in $\bbM$. Building on \cite{bhattacharya2003large, bhattacharya2005large}, in Appendix~\ref{app:frechet mean}, we show that $\hat{\mu}_{n, \Hodges}$ is super-efficient in the sense that $\E [(\sqrt{n} \Exp_{\mu_{n}}^{-1} \hat{\mu}_{n, \Hodges})^2] \rightarrow \infty$ as $n \rightarrow \infty$ for $\mu_{n}$ in a neighborhood of $\mu_{0}$ shrinking toward $\mu_{0}$ at a certain rate, where $\Exp^{-1}_{\mu_{n}} \cdot$ is the inverse of the so-called \emph{exponential map} at $\mu_{n}$, a natural extension of the subtraction operation from normed linear spaces to Riemannian manifolds (see Section~\ref{sec:manifolds} for details).
\subsubsection{The Intrinsic Viewpoint}
\label{sec:intrinsic}
\leavevmode
Even for parameters lying in a normed linear space, as quite persuasively argued by \cite{oller1995intrinsic}, the intrinsic or coordinate-free perspective to statistical modeling and inference can still be fruitful and mathematically elegant. Thus, it is also theoretically appealing to recast efficiency theory in an intrinsic or coordinate-free fashion. In this regard, the theoretical framework to be developed in the sequel can be viewed as a marriage between the intrinsic approach in the spirit of \cite{oller1995intrinsic} and the efficiency theory in the spirit of Le Cam. A conceptual advantage of the intrinsic viewpoint will be demonstrated in Section~\ref{sec:SIM}. It is arguable that one could instead first choose a chart or coordinate system on the manifold and then try to apply the classical efficiency theory developed for linear spaces. However, such a coordinate-dependent approach can be cumbersome to work with, when the parameter space carries a natural manifold structure (see examples provided at the beginning of the Introduction or Remark~\ref{rem:Hessian} later). For example, certain estimators of manifold-valued parameters are not easily expressible in closed form in a local chart, such as the empirical \Frechet{} mean. A possibly more convincing advantage is that, when choosing a chart, eventually one needs to take the supremum over all possible charts for deriving efficiency bounds, which may complicate analysis; otherwise two different statisticians may choose different charts and one run into the trouble of further verifying if their efficiency bounds agree. Working with intrinsic geometric concepts as to be developed later in our paper avoids such a complication. Our viewpoint is that it is conceptually natural to establish the efficiency theory using the intrinsic geometric concepts adopted here. In practice, one could resort to convenient numerical tools (such as retraction \cite{absil2008optimization} or other numerical methods for approximations) for computation.
\subsection{Related Literature}
\label{sec:literature}
Our work belongs to the broader program of (semi-)parametric efficiency theory in statistics. As indicated, parametric efficiency theory dates back to classical works such as \cite{le1960locally, hajek1970characterization, le1986asymptotic, le2000asymptotics}. Semiparametric efficiency was, to our knowledge, first formulated in \cite{stein1956efficient}, with the more general theory developed in \cite{begun1983information, van1988statistical, bickel1998efficient, van1991differentiable}, among many others. \cite{van1991differentiable} developed a general semiparametric efficiency theory and differentiable functionals, on which much classical and contemporary work on semiparametric efficiency theory is heavily based \cite{robins1994estimation, luedtke2024one, graham2026towards, takatsu2024generalized}. Textbook or tutorial-level treatment on semiparametric efficiency theory can be found in \cite{bickel1998efficient, van2000asymptotic, tsiatis2006semiparametric}.
All the above works assume that the parameter space is a normed linear space. The relevant literature on parameter manifolds is relatively scarce. Classical works such as \cite{bhattacharya2003large, bhattacharya2005large}, together with a more recent work \cite{brunel2023geodesically}, establish asymptotic normality and convergence rates for $M$/$Z$-estimators over parameter manifolds, albeit leaving the efficiency of such estimators aside. In \cite{oller1995intrinsic} just mentioned, foundational statistical concepts, including unbiasedness, are extended to the case where the underlying parameter space is a Riemannian manifold, with further development in \cite{smith2005covariance, jupp2010van} and references therein. More recently, \cite{chen2025semiparametric} also investigated the problem of estimating the integral of a regression function on a smooth submanifold through the lens of semiparametric theory. Closer to our work, there exist a few isolated papers establishing the efficiency bound for parameter manifolds on a case-by-case basis, which also partly motivates our work to offer a more unified framework. \cite{bigot2012semiparametric} briefly considers the efficiency of finite-dimensional parameters in a generalized white noise model defined on Abelian Lie groups, motivated by image registration problems in practice. However, as also mentioned in that paper, their results cannot be directly extended to efficiency theory over more general manifolds. We refer readers to Remark~\ref{rem:Lie groups} for more details on the results in \cite{bigot2012semiparametric}. Another example is the single-index model, in which the regression coefficient parameter is assumed to belong to a unit hemisphere (a geometric constraint) to ensure identifiability. \cite{kuchibhotla2020efficient} rigorously derived the semiparametric efficiency bound by carefully crafting a parametric submodel that respects the geometric constraint. As mentioned in \cite{kuchibhotla2020efficient}, this construction can be a nontrivial problem (in the words of \cite{cui2011efm}, this problem is nonregular and challenging). However, in Section~\ref{sec:SIM} we will demonstrate that the differential geometric language to be adopted in our framework can drastically simplify the construction, at least conceptually. In a sense, the challenge or difficulty disappears once we adopt the correct mathematical language; see Remark~\ref{rem:intuition} for a related comment. Finally, a very interesting recent work by \cite{huang2026high} established methods for statistical inference on parameter manifolds that can achieve higher-order asymptotic accuracy, which can be a natural next step for our work.
\subsection{Main Contributions and Organization}
\label{sec:contributions}
To our knowledge, this is the first paper to systematically generalize the semiparametric theory from \textit{linear spaces} to \textit{parameter manifolds}, unifying previous isolated results. The take-home message of our paper can be aptly summarized in Table~\ref{tab:summary}, which provides a vocabulary that translates the essential concepts in efficiency theory for normed linear parameter spaces into those for parameter manifolds. The precise meaning of the mathematical notation that appears in Table~\ref{tab:summary} will be defined later in Section~\ref{sec:geometry}.

\begin{table*}[!t]
\centering
\resizebox{\textwidth}{!}{%
\begin{tabular}{c|c|c}
\toprule
& linear parameter spaces & parameter manifolds \\
\midrule
\makecell{perturbation/path \\ (Definition~\ref{def:riemann regular} and \ref{def:semipar regular estimators})} & $\theta_{h} = \theta + h$ & $\theta_{h} = \Exp_{\theta} h$ \\
\hline
\makecell{estimation error \\ (Definition~\ref{def:bias})} & $\hat{\psi}_{n} - \psi (\theta)$ & $\Exp^{-1}_{\psi (\theta)} \hat{\psi}_{n}$ \\
\hline
\makecell{regular estimators \\ (Definition~\ref{def:riemann regular} and \ref{def:semipar regular estimators}; \\ Theorem~\ref{thm:par} and \ref{thm:semipar convolution})} & $\sqrt{n} (\hat{\psi}_{n} - \psi (\theta_{h / \sqrt{n}})) \rightsquigarrow \mathsf{Z}_{\theta}$ & $\sqrt{n} \Pi_{\psi (\theta_{h / \sqrt{n}})}^{\psi (\theta)} \Exp_{\psi (\theta_{h / \sqrt{n}})}^{-1} \hat{\psi}_{n} \rightsquigarrow \mathsf{Z}_{\theta}$ \\
\hline
\makecell{functional derivative \\ (parametric version; \\
Section~\ref{sec:par})} & \makecell{$\dot{\psi} (\theta) = \underset{h \rightarrow 0}{\lim} \dfrac{\psi (\theta_{h}) - \psi (\theta)}{h} = \E [\IF \cdot \S]$} & \makecell{$\dot{\psi} (\theta) = \underset{h \rightarrow 0}{\lim} \dfrac{\Exp^{-1}_{\psi (\theta)} \psi (\theta_{h})}{h} = \E [\IF \cdot \S]$} \\
\hline
\makecell{functional derivative \\ (general version; \\
Section~\ref{sec:semipar})} & \makecell{$\dot{\chi}_{\P} (\s) = \underset{t \rightarrow 0}{\lim} \dfrac{\chi (\P_{t}) - \chi (\P)}{t} = \E [\IF \cdot \S]$} & \makecell{$\dot{\chi}_{\P} (\s) = \underset{t \rightarrow 0}{\lim} \dfrac{\Exp^{-1}_{\chi (\P)} \chi (\P_{t})}{t} = \E [\IF \cdot \S]$} \\
\bottomrule
\end{tabular}
}
\vspace{1em} 
\caption{A vocabulary of concepts in efficiency theory: From normed linear parameter spaces to parameter manifolds. Here $\psi (\theta)$ stands for the parameter of interest and $\hat{\psi}_{n}$ stands for its estimator based on $n$ i.i.d. samples. $\S$ denotes the score corresponding to the joint distribution. The reference to Definitions and Theorems in the parentheses orients readers to the most relevant parts of later sections. But all listed concepts are related to the paper as a whole.}
\label{tab:summary}
\end{table*}

Our main contribution of this paper is to rigorously establish the vocabulary in Table~\ref{tab:summary} to facilitate the use of efficiency theory over parameter manifolds, as the familiar concepts in normed linear parameter spaces can be transported to parameter manifolds. But establishing Table~\ref{tab:summary} involves geometric analysis unique to parameter manifolds. More specifically, we make contributions in the following aspects. 
\begin{itemize}
\item First, efficiency theory mainly concerns regular statistical models. We generalize concepts related to regular statistical models from normed linear spaces to manifolds, including \emph{differentiable in quadratic mean} (DQM) \cite{pollard1997another} and \emph{local asymptotic normality} (LAN) \cite{le1960locally}. 
\item Second, in finite-dimensional parametric models over parameter manifolds, we extend the convolution theorem and local asymptotic minimax theorem to manifold settings. As a byproduct, we justify the calculus of influence functions for manifold-valued parameters, allowing practitioners interested in constructing estimators via influence functions to transport the familiar approach from linear spaces to manifolds. 
\item Third, we also consider semiparametric problems through the more general lens of differentiable functionals, akin to the approach adopted in \cite{van1991differentiable}. 
\item Finally, using two concrete examples, we demonstrate that our proposed framework can be used to (1) establish efficiency bound for common geometrical statistical problems such as \Frechet{} mean and (2) conceptually simplify the existing efficiency results on manifold-valued parameters in the geometric language of our new framework. 
In particular, as a byproduct of the efficiency bound of \Frechet{} mean, we also derive the influence function of the \Frechet{} mean subject to missing-at-random (MAR) missingness, which generalizes the influence function of the mean under MAR originally derived in the seminal work of \cite{robins1994estimation, hahn1998role} when the sample space is linear; see Remark~\ref{rem:AIPW}.
\end{itemize}
Below is a roadmap for the rest of this paper. Section~\ref{sec:geometry} introduces basic geometric concepts and collects frequently used notation. The information bound, the H\'{a}jek-Le Cam convolution theorem, and the Local Asymptotic Minimax theorem for parametric models are generalized to parameter manifolds in Section~\ref{sec:par}. Semiparametric efficiency theory for manifold-valued parameters is presented in Section~\ref{sec:semipar}. Applications of the generalized efficiency theory are provided in Section~\ref{sec:examples}. We summarize the results in the paper and propose some remarks on future directions of research in Section~\ref{sec:conclusion}. Technical details are deferred to the Appendix. During the presentation, we will point out the place where the nonlinearity of the parameter manifolds necessitates a revamp of the existing efficiency theory over normed linear spaces.
\section{Geometric Background and Notation}
\label{sec:geometry}
In this section, we first briefly review some geometric concepts that are essential to our developments (Section~\ref{sec:manifolds}) and prepare some basic notation (Section~\ref{sec:notation}); more technical geometric concepts and results are delegated to Appendix~\ref{app:technical}. We steer interested readers to Section S.1 of the supplementary materials of \cite{shao2022intrinsic} for a more comprehensive treatment on metric and Riemannian geometry. We will adopt the Einstein summation convention throughout this paper.
\subsection{Concepts Related to Riemannian Manifolds}
\label{sec:manifolds}
A $p$-dimensional Riemannian manifold $\bbM$ is a smooth manifold equipped with a Riemannian metric $g$, which is a $(0, 2)$-tensor field (see Appendix~\ref{app:tensor}). Associated with each point $x \in \bbM$ is the tangent space $\T_{x} \bbM$, comprised of tangent vectors $\nu_{\gamma}$ induced by all smooth curves $\gamma: [0, 1] \rightarrow \bbM$ such that $\gamma (0) \equiv x$ and $\dot{\gamma} (0) \equiv \left. \frac{\diff \gamma (t)}{\diff t} \right\vert_{t = 0} \equiv \nu_{\gamma}$, where $\dot{\gamma}$ formally denotes the differential of $\gamma$; we will write $\nu \equiv \nu_{\gamma}$ if not leading to any ambiguity. The associated cotangent space, denoted by $\T_{x}^{\ast} \bbM$, is simply the dual space of $\T_{x} \bbM$. Following the convention in differential geometry, $\T \bbM \equiv \bigcup_{x \in \bbM} \T_{x} \bbM$ is the tangent bundle of $\bbM$ and $\Gamma (\T \bbM)$ the space of vector fields on $\bbM$.
The metric $g$ defines an inner product $g_{x} (\cdot, \cdot) \coloneqq \langle \cdot, \cdot \rangle_{x}$, that varies smoothly in $x$, on the tangent space $\T_{x} \bbM$ locally at $x \in \bbM$. The associated norm is denoted as $\Vert \nu \Vert_{x} \coloneqq \sqrt{\langle \nu, \nu \rangle_{x}}$ for any tangent vector $\nu \in \T_{x} \bbM$. Again, we often drop the subscript $x$ in both the inner product and the norm if it is clear from the context. Heretofore we also drop the prefix ``Riemannian'' and any mentioned manifold will be understood to be a Riemannian manifold. One should use $(\bbM, g)$ to represent a manifold, though we sometimes write $\bbM$ for short. When written in chart forms, at each $x\in\bbM$, $g$ can be  represented in a form of $p \times p$-matrix $G \coloneqq (g_{i j})$ and we let $(g^{i j})$ denote elements of $G^{-1}$. $g$ also gives rise to a reference measure over $\bbM$ called the volume form: $\diff V_{g} (x) \coloneqq \sqrt{\det G (x)} \diff x$, which is the analogue of the Lebesgue measure over Euclidean spaces. $\diff x$ can be rewritten as $\diff x^1 \cdots \diff x^p$, where $(x^1,\cdots,x^p)$ is understood as a local coordinate. Subsequently, we slightly abuse notation and write $\diff x=\diff V_{g}(x)$ to avoid clutter.
A curve $\frakc$ is called a geodesic if it is a critical point of the energy functional $E (\gamma) \coloneqq \frac{1}{2} \int_{0}^{1} \Vert \dot{\gamma} (t) \Vert^{2} \diff t$. By the Euler-Lagrangian equations, a geodesic must satisfy the following ordinary differential equation (ODE), often referred to as the geodesic equation, written in the coordinate chart form:
\begin{equation}
\label{ODE}
\ddot{\frakc}^{i} (t) + \Gamma_{j, k}^{i} (\frakc (t)) \dot{\frakc}^{j} (t) \dot{\frakc}^{k} (t) = 0, \text{ for } i = 1, \cdots, p,
\end{equation}
where $\Gamma_{j, k}^{i} \coloneqq \frac{1}{2} g^{i l} \left( g_{j l, k} + g_{k l, j} - g_{j k, l} \right)$ is referred to as the Christoffel symbol and $g_{j l, k} \coloneqq \frac{\partial}{\partial \frakc^{k}} g_{j l}$. This ODE formulation can be helpful when one needs to compute certain manifold-related quantities in concrete applications, such as the exponential map and its inverse to be defined in the next paragraph. In general, when analytical expressions for these quantities cannot be found, one can still solve \eqref{ODE} and evaluate related quantities numerically.
By applying the Picard–Lindel\"{o}f theorem to the geodesic equation~\eqref{ODE}, given $x \in \bbM$ and $\nu \in \T_{x} \bbM$, there is a unique geodesic $\frakc_{\nu}: [0, \eps] \rightarrow \bbM$ for some $\eps > 0$ such that $\frakc_{\nu} (0) = x, \dot{\frakc}_{\nu} (0) = \nu$ and $\frakc_{\nu}$ vary smoothly in both $x$ and $\nu$. Let $V_{x} \coloneqq \{\nu \in \T_{x} \bbM: \text{ $\frakc_{\nu}$ is uniquely defined over $[0, 1]$}\}$. We can then introduce the exponential map $\Exp_{x}: V_{x} \subseteq \T_{x} \bbM \rightarrow \bbM$ where $\Exp_{x} (\nu) = \frakc_{\nu} (1)$. We further assume $\bbM$ to be a complete manifold throughout this paper, so $\Exp_{x}$ is well defined over $\T_{x} \bbM$. Its inverse, denoted by $\Exp_{x}^{-1}: \bbM \rightarrow V_{x}$, if it exists, is often referred to as the logarithmic map. $\Exp_{x}$ is a locally diffeomorphic map so $\Exp_{x}^{-1}$ is uniquely defined locally, but not necessarily globally. $\bbM$ is often additionally attached with an \textit{(affine) connection} $\nabla \equiv \nabla_{\cdot} \cdot: \Gamma (\T \bbM) \times \Gamma (\T \bbM) \rightarrow \Gamma (\T \bbM)$. In this paper, we adopt the torsion-free (Levi-Civita) connection $\nabla$ induced by $g$ in the following sense: in chart form, $\nabla_{\partial_{j}} \partial_{i} \equiv \Gamma_{i, j}^{k} \partial_{k}$ for $i, j = 1, \cdots, p$, where $\partial_{i}$ refers to the coordinate basis vector field associated with the $i$-th coordinate in the local chart for $i = 1, \cdots, p$.
With $\nabla$, the (Riemannian) curvature tensor field $\calR: \Gamma (\T \bbM) \times \Gamma (\T \bbM) \times \Gamma (\T \bbM) \rightarrow \Gamma (\T \bbM)$ can be represented as
\begin{equation}
\label{curvature tensor field}
\begin{split}
\calR (\nu, \mu) \omega & \equiv \nabla_{\nu} \nabla_{\mu} \omega - \nabla_{\mu} \nabla_{\nu} \omega - \nabla_{[\nu, \mu]} \omega \\
& \equiv \left( [\nabla_{\nu}, \nabla_{\mu}] - \nabla_{[\nu, \mu]} \right) \omega,
\end{split}
\end{equation}
where $\nu, \mu, \omega \in \Gamma (\T \bbM)$, $[\nu, \mu]$ is the Lie bracket of vector fields, and $[\nabla_{\nu}, \nabla_{\mu}]$ denotes the commutator between connections. In differential geometry, depending on the context, various different curvatures related to the (Riemannian) curvature have been introduced. Later in this paper, we will also encounter the \emph{sectional curvature}, denoted as $\calK$. Formally, let $x \in \bbM$ be a point and $M_{2} \subset \T_{x} \bbM$ be a 2-dimensional subspace of the tangent space at $x$. Choose an orthonormal basis $\{\mu_1, \mu_2\}$ of $M_{2}$, then the sectional curvature of $M_{2}$ at $x$ is defined as $\calK (M_{2}) \coloneqq \frac{\langle \calR (\mu_2, \mu_1) \mu_1, \mu_2 \rangle}{\|\mu_1 \wedge \mu_2\|^2}$, where $\|\mu_1 \wedge \mu_2\|^2$ denotes the squared area of the parallelogram formed by $\mu_1$ and $\mu_2$.
$\nabla$ also allows one to define the so-called \textit{parallel transport} map: given a smooth curve $\gamma: [0, 1] \rightarrow \bbM$ and $s, t \in [0, 1]$, $\Pi (\gamma)_{s}^{t}$ denotes the parallel transport map along $\gamma$ from $\T_{\gamma (s)} \bbM$ to $\T_{\gamma (t)} \bbM$. It means that for any $\nu \in \T_{\gamma (s)} \bbM$, there is an associated vector field $\nu_{\cdot} \in \Gamma (\T \bbM)$ such that $\nabla_{\dot{\gamma}} \nu = 0$, $\nu_{\gamma (s)} = \nu$ and $\nu_{\gamma (t)} = \Pi (\gamma)_{s}^{t} \nu$. When the smooth curve $\gamma$ is clear from the context, we also write $\Pi_{x_{1}}^{x_{2}} \equiv \Pi (\gamma)_{s}^{t}$ for short, with $x_{1} = \gamma (s), x_{2} = \gamma (t)$. Geometrically, it moves tangent vectors in $\T_{x_{1}} \bbM$ to tangent vectors in $\T_{x_{2}} \bbM$ smoothly along the curve, with the inner product preserved. The parallel transport map plays a significant role in our paper because it licenses comparisons between tangent vectors belonging to different tangent spaces. Whenever parallel transport is used between two nearby points, it is taken along the unique distance-minimizing geodesic inside a chosen normal neighborhood.
\begin{remark}
Another concept that distinguishes a manifold from a Euclidean space is cut locus. The cut locus of a point $\mu \in \bbM$, denoted as $\calC (\mu)$, is the set of points in $\bbM$ to which the distance-minimizing geodesics from $\mu$ are not unique \cite{hotz2024central}. For instance, points on a Euclidean space have an empty cut locus, but the cut locus of a point on a sphere is a singleton composed of its antipodal point. In general, $\Exp_{\mu}$ only has a unique inverse map $\Exp_{\mu}^{-1}$ (also called the logarithmic map) outside the cut locus $\calC (\mu)$ of $\mu$, but the inverse is not unique within $\calC (\mu)$. However, since we have assumed that the manifold is smooth, the exponential map is smooth and its differential at the origin of $\T_{\mu}\bbM$ is the identity. By the Inverse Function Theorem, there exists a neighborhood of $\mu$ on which $\Exp_{\mu}$ is a diffeomorphism, $\Exp_{\mu}^{-1}$ always exists locally around $\mu$.
\end{remark}
\begin{remark}
\label{rem:intuition}
As mentioned previously, for parameters that lie in a manifold, such as the regression coefficients of a single-index model, it is not immediately clear how to perturb such parameters to ensure that the perturbed parameters still lie in the manifold. It should not be surprising to readers familiar with differential geometry that this desiderata can be guaranteed by perturbing the parameters via the exponential map. Then it only remains to compute the exponential map given a particular manifold. We direct readers to Section~\ref{sec:examples} for concrete examples. This is why we believe that adopting the geometric language can conceptually simplify the application of efficiency theory, as it is expected to become easier to find out how to compute exponential and logarithmic maps with the improvement of modern tools with a focus on mathematics and the formalization of mathematics and mathematical statistics (see, for instance, the \href{https://stat-lib.github.io/}{stat-lib project}).
\end{remark}
\subsection{Notation}
\label{sec:notation}
Before proceeding, we collect some extra notation that is frequently used throughout the paper. Capital letters are reserved for random variables, including the symbol $\IF$, which is often used to denote influence functions in the modern literature on (semiparametric) efficiency theory. Given a random object $A$, we let $\bbL (A)$ denote its probability law. $\P_{n} (\cdot)$ is reserved for the empirical measure over $n$ observations. Given a sequence of random objects $\{X_{n}\}$ indexed by $n$, $X_{n} \rightsquigarrow_{\P} X$ means that the sequence $X_n$ weakly converges to $X$ under a probability measure $\P$. When it is clear from the context what the probability measure $\P$ is, we may also write $X_{n} \rightsquigarrow X$. We let $\E, \Var$, and $\Cov$ be the expectation, variance, and covariance under the probability measure $\P$. If $\P$ is absolutely continuous with respect to some standard reference measure $\lambda$\footnote{For example, if the underlying sample space is Euclidean, then $\lambda$ is understood to be the Lebesgue measure, whereas if the underlying sample space is a Riemmanian manifold, then $\lambda$ is understood to be the volume form.}, we let $\p \coloneqq \frac{\diff \P}{\diff \lambda}$ be the corresponding Radon-Nikodym derivative (or density) of $\P$, $\r \coloneqq \sqrt{\p}$ be the square-root density, and $\ell \coloneqq \log \p$ be the log-likelihood function. We denote $L_{0}^{2} (\P)$ be the space of square integrable functions on tangent space with zero mean under $\P$. When needed, we also write $\P_{\theta}, \E_{\theta}, \cdots$ to emphasize the dependence on some (unknown) parameter $\theta$.
Given two linear spaces $\bbX_{1}, \bbX_{2}$, we denote $\bbX_{1} \otimes \bbX_{2}$ as their  (algebraic) tensor product. Let $\bbX_{1}^{\ast}, \bbX_{2}^{\ast}$ be the dual spaces of $\bbX_{1}, \bbX_{2}$, respectively. For a linear space $\bbX$ and its dual $\bbX^{\ast}$, we use $\langle \mathsf{x}^{\ast}, \mathsf{x} \rangle$ to denote the value of the functional $\mathsf{x}^{\ast} \in \bbX^{\ast}$ acting on $\mathsf{x} \in \bbX$. Given $\mathsf{x}_{1}^{\ast} \in \bbX_{1}^{\ast}$, $\mathsf{x}_{2}^{\ast} \in \bbX_{2}^{\ast}$, we identify their tensor product $\mathsf{x}_{1}^{\ast} \otimes \mathsf{x}_{2}^{\ast}: \bbX_{1} \otimes \bbX_{2} \rightarrow \bbR$ with the following bilinear form:
\begin{equation}
\label{operator tensor}
\mathsf{x}_{1}^{\ast} \otimes \mathsf{x}_{2}^{\ast} [\mathsf{x}_{1}, \mathsf{x}_{2}] \coloneqq \langle \mathsf{x}_{1}^{\ast}, \mathsf{x}_{1} \rangle \langle \mathsf{x}_{2}^{\ast}, \mathsf{x}_{2} \rangle,
\end{equation}
for any $(\mathsf{x}_{1}, \mathsf{x}_{2}) \in \bbX_{1} \otimes \bbX_{2}$. Finally, $\rmi$ is reserved for the imaginary unit of complex numbers.
Throughout the paper, unless otherwise stated, all probability measures are assumed to be dominated by a reference measure $\lambda$ on the sample space $(\bbX, \calX)$. If $\bbX$ is Euclidean, $\lambda$ is the Lebesgue measure; if $\bbX$ is a Riemannian manifold, $\lambda$ is the associated volume form. For each parameter value $\theta$, we write
\[
\p(\cdot;\theta)=\frac{\diff \P_\theta}{\diff \lambda} (\cdot),
\qquad
\r (\cdot;\theta)=\sqrt{\p}(\cdot;\theta).
\]
Whenever a parametric density is used, $(x,\theta)\mapsto \p(x;\theta)$ is assumed to be jointly measurable. We also assume that $\theta\mapsto \psi(\theta)$ is measurable, and that all expectations and integrals appearing in the paper are well defined with respect to the relevant dominating measure.
\section{Efficiency Theory on Parameter Manifolds for Parametric Models}
\label{sec:par}
We commence by assuming that the probability law generating the sample $\bX$ belongs to a parametric family of distributions (or a statistical model), denoted as 
\begin{equation}
\label{parametric model}
\calP \equiv \calP (\Theta) \coloneqq \{\P_\theta: \theta \in \Theta\}, 
\end{equation}
where $\P_{\theta}$'s are probability measures defined over a measurable sample space $(\bbX, \calX)$ parameterized by $\theta \in \Theta$. $(\Theta, d)$ is a manifold equipped with metric $d$ and of dimension $p$. Given $\bX$, we are interested in conducting inference on the parameter of interest $\psi (\theta)$, with $\psi: \Theta \rightarrow \Psi$, where $(\Psi, g)$ is the $q$-dimensional parameter manifold of $\psi (\theta)$. When $\psi = \id$, the parameter of interest is simply $\theta$ itself.
The equivalence between the smoothness of $\psi$ and the existence of regular estimators when both $\Theta$ and $\Psi$ are normed linear spaces \cite{hirano2012impossibility} prompts us to impose the following differentiability condition on $\psi$.
\begin{assumption}
\label{as:smooth transformation}
$\psi$ is assumed to be a differentiable map from $\Theta$ to $\Psi$, with differential (a linear map) denoted by $\dot{\psi} (\theta)$.
\end{assumption}
\begin{lemma}
\label{lem:directional derivative}
Under Assumption~\ref{as:smooth transformation}, for any $h \in \T_{\theta} \Theta$,
\begin{align*}
\dot{\psi} (\theta) (h) = \lim_{t \rightarrow 0} t^{-1} \Exp^{-1}_{\psi (\theta)} \psi (\Exp_{\theta} (t \cdot h)) \in \T_{\psi (\theta)} \Psi.
\end{align*}
Note that the exponential map $\Exp_{\psi (\theta)}$ and inverse exponential map $\Exp_{\theta}^{-1}$ in the above display relate to different manifolds.
\end{lemma}
\begin{proof}
Given a tangent vector $h \in \T_{\theta} \Theta$, we take a smooth curve $\gamma (t)$ on $\Theta$ such that $\gamma (0) = \theta$ and $\dot{\gamma} (0) = h \in \T_{\theta} \Theta$. A canonical choice is the geodesic starting at $\theta$ along the direction $h$, i.e., $\gamma (t) = \Exp_\theta (t \cdot h)$. By definition, the differential of $\psi$ at $\theta$ applied to $h$ is $\dot{\psi} (\theta)(h) \equiv \left.\frac{\diff}{\diff t}\right|_{t=0} \psi (\gamma(t)) \in \T_{\psi (\theta)} \Psi$. Then $\left. \frac{\diff}{\diff t}\right|_{t = 0} \psi (\gamma (t)) = \lim_{t \rightarrow 0} \frac{1}{t} \Exp_{\psi(\theta)}^{-1} (\psi (\gamma (t)))$. Since $\gamma (t) \equiv \Exp_{\theta} (t \cdot h)$, the proof is complete by substituting this identity into $\left. \frac{\diff}{\diff t}\right|_{t = 0} \psi (\gamma (t))$.
\end{proof}
\begin{remark}
\label{rem:retraction}
As pointed out by a referee, retractions and vector transports, instead of exponential maps and parallel transports, are often used in practice for computational purposes \cite{absil2008optimization, huang2026high}. It is therefore practically useful to investigate whether our theoretical results are still robust if we change exponential maps and parallel transports to retractions and vector transports. In Appendix~\ref{app:retraction}, we discuss this issue thoroughly. In a nutshell, it is possible to reframe our results using appropriate retractions and vector transports that agree with exponential/logarithmic maps and parallel transports up to certain approximation accuracy, so as to further narrow the gap between our theoretical results and practice in certain areas. However, it indeed becomes quite cumbersome to do so, given that we already have a set of convenient geometric tools to leverage.
\end{remark}
\subsection{Parametric Models that are Differentiable in Quadratic Mean over Regular Parameter Manifold}
\label{sec:par dqm}
The first step toward establishing efficiency theory is to impose appropriate regularity conditions on the underlying statistical model. In the classical settings where parameters take values in Euclidean spaces, models that are differentiable in quadratic mean (DQM) constitutes the most general class of statistical models to which efficiency theory is applicable. For parameters taking values in manifolds, we generalize the definition of DQM as follows.
\begin{definition}
\label{def:par dqm}
Suppose $\theta \in \Theta$ is the parameter of interest taking values in a manifold $\Theta$. Let $\p (\cdot; \theta)$ be the probability density function of $\P_{\theta}$ and recall that $\r (\cdot; \theta) \equiv \sqrt{\p (\cdot; \theta)}$ dominated by some reference measure $\lambda$ (which is not necessarily unique). Given any tangent vector $h \in \T_{\theta} \Theta$, let $\theta_{h} \coloneqq \Exp_{\theta} (h)$ and thus $\theta_{0} \equiv \theta$. The statistical model $\calP \equiv \calP_{\Theta}$ is said to satisfy DQM at $\P_{\theta}$ if there exists an operator $\s (\cdot; \theta): \bbX \rightarrow \T_{\theta}^{\ast} \Theta$, where $\T_{\theta}^{\ast} \Theta$ denotes the dual of $\T_{\theta} \Theta$ (or the cotangent space), such that as $h \rightarrow 0$,
\begin{equation}
\label{dqm}
\begin{split}
\int_{x \in \bbX} \left\{ \r (x; \theta_{h}) - \r (x; \theta) - \frac{1}{2} \s (x; \theta) (h) \cdot \r (x; \theta) \right\}^2 \diff x \\
= o (\left\Vert h \right\Vert^2).
\end{split}
\end{equation}
Here, $\s (\cdot; \theta)$ is the score function of $\theta$. For convenience, we often write $\S_{\theta} \equiv \s (X; \theta)$. We define the \emph{model tangent space} as
\begin{equation}
\label{parametric model tangent space}
\Lambda \coloneqq \mathrm{cl} \left[ \, \linspan \left\{ \S_{\theta} (h): h \in \T_{\theta}  \Theta \right\} \right],
\end{equation}
where $\mathrm{cl} [\cdot]$ denotes the closure of a set.
\end{definition}
\begin{remark}
\label{rem:tangent space terminology}
Note that here we deliberately use ``model tangent space'' instead of ``tangent space'' and introduce a different notation to avoid confusions between the tangent space of a statistical model (or a statistical manifold in the language of information geometry \cite{lauritzen1987statistical, kass1989geometry, kass2011geometrical}) and the tangent space of the sample/parameter space that is itself a manifold.
\end{remark}
\begin{remark}
\label{rem:DQM}
Checking whether a model is DQM as in Definition~\ref{def:par dqm} requires geometric analysis machinery, which we collect in Appendix~\ref{app:geometric analysis}. Given a square-root density $\r (x; \theta)$, to verify \eqref{dqm}, we first need to analyze how close $\theta_{h}$ is to $\theta$. This can be done by invoking Lemma~\ref{lem:log map expansion} and Lemma~\ref{lem:connection} in Appendix~\ref{app:geometric analysis}.
\end{remark}
With the score function introduced, we are in place to define the (Fisher) Information Operator $\bbG_{\theta}$ and its inverse at $\P_{\theta}$, another two important concepts toward building the efficiency theory.
\begin{definition}
\label{def:information operator}
Let $\nu_{1}$ and $\nu_{2}$ be two arbitrary tangent vectors in $\T_{\theta} \Theta$. We define the Fisher information operator $\bbG_{\theta} (\cdot, \cdot): \T_{\theta} \Theta \times \T_{\theta} \Theta \rightarrow \bbR$ as $\bbG_{\theta} (\nu_{1}, \nu_{2}) = \E [\S_{\theta} (\nu_{1}) \cdot \S_{\theta} (\nu_{2})]$. $\bbG$ can be viewed as a $(0, 2)$-tensor field (see Remark~\ref{rem:tensor examples} in Appendix~\ref{app:tensor}). We can also write $\bbG_{\theta} = \E [\S_{\theta} \otimes \S_{\theta}]$. Furthermore, if $\bbG_{\theta}$ is non-singular, we let $\bbG_{\theta}^{-1}$ denote its inverse, that is, $\bbG_{\theta}^{-1} \cdot \bbG_{\theta} = \bbG_{\theta} \cdot \bbG_{\theta}^{-1} = \mathbb{I}$.
\end{definition}
\begin{remark}
\label{rem:Fisher 2nd}
Following \cite{smith2005covariance} (and also earlier works such as \cite{hendriks1991cramer}), as in the case of normed linear spaces, Fisher information operator $\bbG_{\theta}$ can be equivalently characterized in the form of a \emph{Hessian}. Recall that $\p (\cdot; \theta)$ is the probability density function of $\P_{\theta}$ and $\r (\cdot; \theta) \equiv \sqrt{\p (\cdot; \theta)}$ is dominated by the reference measure $\lambda$. Suppose that $\p (\cdot; \theta)$ is twice differentiable in $\theta$, $\nabla_{\theta}\int_{x \in \bbX} \nabla_{\theta} \p (x; \theta)\diff x=\int_{x \in \bbX} \nabla^2_{\theta} \p (x; \theta)\diff x$, and $\nabla_{\theta}\int_{x \in \bbX} \p (x; \theta)\diff x=\int_{x \in \bbX} \nabla_{\theta} \p (x; \theta)\diff x$ for $\theta \in \Theta$, we have
\begin{equation}
\label{Fisher 2nd}
\bbG_{\theta} \equiv - \E [2 \nabla_{\theta}^{2} \r (X; \theta)],
\end{equation}
where $\nabla$ is the connection induced by the metric $d$.
\end{remark}
The following proposition states that DQM is a sufficient condition for LAN under the i.i.d. sampling scheme from a parametric statistical model with manifold-valued parameter $\theta \in \Theta$.
\begin{proposition}[Local Asymptotic Normality of DQM Experiments]
\label{prop:lan}
Suppose that the statistical model $\calP \coloneqq \{\P_{\theta}: \theta \in \Theta\}$ is DQM at $\theta$, and that we have access to a sample of $n$ i.i.d. observations $\{X_{i}\}_{i = 1}^{n}$ drawn from $\P_{\theta}$. As in Definition~\ref{def:par dqm}, for every $h \in \T_{\theta} \Theta$, let $\theta_{n, h} \coloneqq \Exp_{\theta} \left( \frac{h}{\sqrt{n}} \right)$, we have
\begin{equation}
\label{lan}
\begin{split}
& \sum_{i=1}^n \log\frac{\p (X_i; \theta_{n, h})}{\p (X_i; \theta)} = \sum_{i = 1}^n \s (X_{i}; \theta) \left( \frac{h}{\sqrt{n}} \right) \\
& \qquad\qquad\qquad - \frac{1}{2} \bbG_{\theta} (h, h) + o_{\P_{\theta}} (1).
\end{split}
\end{equation}
\end{proposition}
The proof of Proposition~\ref{prop:lan} is given in Appendix~\ref{app:prop lan}. Since, by DQM, we have assumed the existence of a score function in the cotangent space that is locally isomorphic to a linear space, the proof is not essentially different from the case where $\Theta$ is normed linear space. An important consequence of Proposition~\ref{prop:lan} is 
\begin{equation}
\label{lan_clt}
\sum_{i = 1}^{n} \log \frac{\p (X_i; \theta_{n, h})}{\p (X_i; \theta)} \rightsquigarrow_{\P_{\theta}} \calN \left( -\frac{1}{2} \bbG_{\theta} (h, h), \bbG_{\theta} (h, h) \right),
\end{equation}
that is, the log-likelihood ratio process between two contiguous probability distributions converges weakly to a tight Gaussian measure.
\begin{remark}
\label{rem:LAN}
The i.i.d. assumption imposed in Proposition~\ref{prop:lan} is not necessary for LAN. Other conditions may also suffice, such as rapidly mixing dependent data. Often the LAN expansion itself can be utilized as a high-level condition imposed on a sequence of regular statistical models.
\end{remark}
\begin{example}
\label{ex:riemannian gaussian}
Consider the family of Gaussian probability measures $\{\P_{\mu, \sigma^2}: \mu \in \bbM, \sigma^{2} > 0\}$ parameterized by $\mu \in \bbM$ where $(\bbM, d)$ is a manifold and $\sigma^2 > 0$ \cite{said2021statistical}. The probability density function is
\begin{equation}
\p (x; \mu, \sigma^2) = \mathrm{Z}_{\mu, \sigma^{2}}^{-1} \exp \left\{ - \frac{d (x, \mu)^2}{2 \sigma^2} \right\}
\end{equation}
and $\mathrm{Z}_{\mu, \sigma^{2}}$ is the partition function. If we further assume that $\bbM$ is a homogeneous Hadamard manifold with strictly bounded curvature, the partition function $\mathrm{Z}_{\mu, \sigma^{2}} \equiv \mathrm{Z}_{\sigma^{2}}$ does not depend on $\mu$, in which case the score function w.r.t. $\mu$ is simplified as
\begin{align*}
\s (x; \mu) \coloneqq - \frac{d (x, \mu)}{\sigma^{2}} \nabla_{\mu} d (x, \mu).
\end{align*}
One can straightforwardly check that $\s (x; \mu)$ satisfies \eqref{dqm}, by considering a local perturbation of $\mu$ by $\mu_{h} = \Exp_{\mu} (h)$ using results from Appendix~\ref{app:geometric analysis}. Apart from Riemannian Gaussian distributions, exponential family distributions have also been extended to manifolds \cite{mardia2009directional}.
\end{example}
\subsection{The \Hajek-Le Cam Convolution Theorem}
\label{sec:convolution theorem}
Before stating the \Hajek-Le Cam convolution theorem for parameters taking values in a manifold, we need to first discuss several instrumental concepts in efficiency theory: unbiasedness, the \Cramer-Rao lower bound (CRLB), and regular estimator sequences.
Starting from this section, since we need to compare estimators and parameters using the logarithmic map, we always impose the following two regularity assumptions without explicitly mentioning them.
\begin{assumption}
\label{as:cut locus}
We assume the following for any estimator $\hat{\psi}$ of $\psi (\theta)$:
\begin{align*}
\sup_{\theta \in \Theta} \P_{\theta} (\hat{\psi} \in \calC (\psi (\theta))) = 0.
\end{align*}
\end{assumption}
\begin{assumption}
    \label{perturbation}
    For every $\theta \in \Theta$ and every fixed $h \in \T_\theta \Theta$, there exists a normal neighborhood $U_{\psi(\theta)} \subset \Psi$ of $\psi(\theta)$ such that $\psi(\theta_{n, h}) \in U_{\psi(\theta)}$ for all sufficiently large $n$, where $\theta_{n, h}=\Exp_\theta(h / \sqrt{n})$. In particular, the distance-minimizing geodesic from $\psi(\theta)$ to $\psi(\theta_{n, h})$ is unique, and $\Pi_{\psi\left(\theta_{n, h}\right)}^{\psi(\theta)}$ is understood as parallel transport along this geodesic.
\end{assumption}
\begin{remark}
Since $\theta_{n, h} \to \theta$ as $n \to \infty$ and $\psi$ is smooth, we have $\psi (\theta_{n, h}) \to \psi(\theta)$. Hence Assumption~\ref{perturbation} is a local condition that is automatically satisfied for each fixed $(\theta, h)$ after choosing a sufficiently small normal neighborhood of $\psi (\theta)$.
\end{remark}
Assumption~\ref{as:cut locus} ensures that $\Exp^{-1}_{\psi (\theta)} \hat{\psi}$ is well defined. For example, when $\Psi$ is a sphere, it requires that $\hat{\psi}$ cannot be the antipodal point of $\psi (\theta)$ almost surely. This is not a strong assumption, since the antipodal point $\psi (\theta)$ is a null set for any continuous distribution over the sphere. This assumption also holds for Hadamard spaces since $\calC(\psi)=\emptyset$ if $\Psi$ is a Hadamard manifold. Next, we generalize the concept of bias of an estimator $\hat{\psi}$ of $\psi (\theta)$ from linear spaces to manifolds, as done in \cite{oller1995intrinsic}.
\begin{definition}
\label{def:bias}
Given an estimator $\hat{\psi}$ of $\psi (\theta) \in \Psi$, the bias (vector) field $\Bias_{\psi (\theta)} (\hat{\psi}) \in \T_{\psi (\theta)} \Psi$ of $\hat{\psi}$ is defined as  
\begin{align*}
\Bias_{\psi (\theta)} (\hat{\psi}) \coloneqq \E [\Exp_{\psi (\theta)}^{-1} \hat{\psi}]. 
\end{align*}
The estimator $\hat{\psi}$ is said to be unbiased if and only if $\Bias_{\psi (\theta)} (\hat{\psi}) \equiv 0$, which is the zero vector field.
\end{definition}
\begin{definition}
\label{def:pop curvature}
Let $\delta \coloneqq \Exp_{\psi (\theta)}^{-1} \hat{\psi} - \Bias_{\psi (\theta)} (\hat{\psi})$ denote the centered residual and $\mathbb{C} \coloneqq \Cov (\delta)$. When $\hat{\psi}_{n}$ is a sequence indexed by $n$, we also write $\mathbb{C}_{n}$. Given any tangent vector $\nu \in \T_{\psi (\theta)} \Psi$, the following population quadratic form related to the curvature tensor $\calR (\cdot, \cdot) (\cdot)$ of $\Psi$
\begin{equation}
\label{pop curvature}
\E [\langle \calR (\delta, \nu) \nu, \delta \rangle]
\end{equation}
depends on the distribution of $\delta$ only via its covariance operator $\mathbb{C}$ (see Remark~\ref{rem:pop curvature}). Therefore we introduce the operator $\calR (\mathbb{C})$ and write \eqref{pop curvature} as a quadratic form $\langle \calR (\mathbb{C}) \nu, \nu \rangle$ for short. Formally, in the notation $\langle \calR (\mathbb{C}) \nu, \nu \rangle$, $\mathbb{C}$ can be replaced by any non-singular operator, e.g. the inverse Fisher information operator $\bbG_{\theta}^{-1}$.
\end{definition}
As proved in Lemma~\ref{lem:linearity of curvature} in Appendix~\ref{app:geometric analysis}, $\calR: A \mapsto \calR (A)$ is a linear map for any non-singular operator $A$. Armed with the above definition, we are ready to introduce the generalized CRLB over parameter manifolds; see \cite{oller1995intrinsic, smith2005covariance, jupp2010van}. In particular, Theorem 2 of \cite{smith2005covariance} characterized the following information-theoretic lower bound on any unbiased estimator of $\psi (\theta) \in \Psi$.
\begin{lemma}[Theorem 2 of \cite{smith2005covariance}]
\label{lem:smith}
Suppose that $\hat{\psi}$ is an unbiased estimator of $\psi (\theta)$ computed from $X \sim \P_{\theta}$. Then the covariance $\mathbb{C}$ of the residual $\Exp_{\psi (\theta)}^{-1} \hat{\psi}$ satisfies the following lower bound:
\begin{equation}
\label{nonlinear CR}
\begin{split}
\mathbb{C} & \succeq \dot{\psi} (\theta) \Big\{ \bbG_{\theta}^{-1} - \frac{1}{3} \calR (\bbG_{\theta}^{-1}) \bbG_{\theta}^{-1} \\
& \qquad - \frac{1}{3} \bbG_{\theta}^{-1} \calR (\bbG_{\theta}^{-1}) \Big\} \dot{\psi} (\theta)^{\ast}.
\end{split}
\end{equation}
\end{lemma}
Compared with the classical CRLB, the CRLB over parameter manifolds has two extra terms related to the curvature. For flat spaces such as Euclidean spaces, these two extra terms vanish to zero as $\calR (\bbG_{\theta}^{-1}) \equiv 0$. The following direct consequence of Lemma~\ref{lem:smith} is more relevant to the i.i.d. observation scheme, which is the main focus of this paper.
\begin{corollary}
Let $\hat{\psi}_{n}$ be an unbiased estimator of $\psi \in \Psi$ computed from $n$ i.i.d. observations $\{X_{i}\}_{i = 1}^{n} \overset{\rm i.i.d.}{\sim} \P_{\theta}$. The covariance of the $\sqrt{n}$-scaled estimation error, $\sqrt{n} \Exp_{\psi (\theta)}^{-1} \hat{\psi}_{n}$, satisfies the inequality
\begin{equation}
\label{n CRLB}
\begin{split}
& \ \Cov \left( \sqrt{n} \Exp_{\psi (\theta)}^{-1} \hat{\psi}_{n} \right) \equiv n \mathbb{C}_{n} \\
\succeq & \ \dot{\psi} (\theta) \Big\{ \bbG_{\theta}^{-1} - \frac{1}{3} \calR \left( n^{-1} \bbG_{\theta}^{-1} \right) \bbG_{\theta}^{-1} \\
& - \frac{1}{3} \bbG_{\theta}^{-1} \calR \left( n^{-1} \bbG_{\theta}^{-1} \right) \Big\} \dot{\psi} (\theta)^{\ast}.
\end{split}
\end{equation}
Asymptotically, \eqref{n CRLB} can be reduced to
\begin{equation}
\label{asymptotic CRLB}
\lim_{n \rightarrow \infty} \Cov \left( \sqrt{n} \Exp_{\psi (\theta)}^{-1} \hat{\psi}_{n} \right) \succeq \dot{\psi} (\theta) \bbG_{\theta}^{-1} \dot{\psi} (\theta)^{\ast}.
\end{equation}
\end{corollary}
\begin{proof}
\eqref{asymptotic CRLB} is a simple consequence of \eqref{n CRLB} and the linearity of $\calR$.
\end{proof}
\newtheorem{innercustomthm}{Example}
\newenvironment{customeg}[1]
{\renewcommand\theinnercustomthm{#1}\innercustomthm}
{\endinnercustomthm}
As mentioned in the Introduction, the super-efficiency phenomenon has led statisticians to search for a more appropriate criterion for justifying the optimality of MLE, which asymptotically achieves the CRLB. One common strategy to reconcile the super-efficiency phenomenon in efficiency theory is by restricting the attention only to the class of (locally) regular estimators. This strategy results in the so-called H{\'a}jek-Le Cam convolution theorem \cite{hajek1970characterization}. Regular estimators are of particular interests in substantive fields including epidemiology and economics, due to the ease of constructing uniformly valid Wald confidence intervals for conducting statistical inference.
We first define the notion of (locally) regular estimators for $\psi (\theta) \in \Psi$ for  a manifold $\Psi$.
\begin{definition}
\label{def:riemann regular}
Suppose that $\psi (\theta) \in \Psi$ is the parameter of interest. For every tangent vector $h \in \T_{\theta} \Theta$, we let $\theta_{n, h} \coloneqq \Exp_{\theta} \left( \frac{h}{\sqrt{n}} \right)$. Then a sequence of estimators $\hat{\psi}_{n}$ indexed by $n$ for $\psi (\theta)$ is called (locally) regular at $\theta$ if there exists a random variable $Z_{\theta}$ whose law, denoted as $\L_{\theta}$, is independent of $h$ such that
\begin{equation}
\Pi_{\psi (\theta_{n, h})}^{\psi (\theta)} \sqrt{n} \, \Exp_{\psi (\theta_{n, h})}^{-1} \hat{\psi}_{n} \rightsquigarrow_{\P_{\theta_{n, h}}} Z_{\theta},
\end{equation}
as $n \rightarrow \infty$.
\end{definition}
As stated above, the definition of regular estimators over parameter manifolds deviates from that over a normed linear parameter space due to the nonlinearity of manifolds. Since the regularity of an estimator concerns its stability when the true parameter is locally perturbed, it involves a comparison between two residuals $\sqrt{n} \Exp_{\psi (\theta_{n, h})}^{-1} \hat{\psi}_{n}$ and $\sqrt{n} \Exp_{\psi (\theta)}^{-1} \hat{\psi}_{n}$. For normed linear spaces, both residuals lie in the same space and thus can be compared directly; whereas for parameter manifolds, the two residuals reside in two different tangent spaces, one locally at $\psi (\theta_{n, h})$ and the other locally at $\psi (\theta)$. To facilitate a direct comparison, the parallel transport map $\Pi_{\psi (\theta_{n, h})}^{\psi (\theta)}$ is introduced to transport the residual with perturbation to $\T_{\psi (\theta)} \Psi$, where the residual without perturbation belongs.
\begin{remark}
\label{rem:nonlinear}
In fact, regular estimator sequences can be defined for more general nonlinear spaces without utilizing the manifold structure, such as relying on the notion of tangent spaces. For example, a possibility is to exploit the metric-geometric machinery exploited in \cite{lin2021total} for more general geodesic spaces. We decide to focus on parameter manifolds in this paper, and the more general case is left for future work.
\end{remark}
\begin{remark}
In view of the above definition, the sample \Frechet{} mean $\hat{\mu}_{n}$ is a regular estimator of the population \Frechet{} mean $\mu$ over a Hadamard manifold $\bbM$; see Appendix~\ref{app:frechet mean} for further details. On the other hand, also in Appendix~\ref{app:frechet mean}, we show that Hodges' estimator $\hat{\mu}_{n, \Hodges}$ defined in \eqref{hodges} is not regular.
\end{remark}
\begin{remark}
\label{rem:inference}
Given a RAL estimator $\hat{\psi}_{n}$ of $\psi = \psi (\theta)$ that satisfies $\sqrt{n} \Exp_{\psi}^{-1} \hat{\psi}_{n} \rightsquigarrow \calN (0, \bbV_{\psi})$ and a consistent estimator $\hat{\bbV}_{n}$ of its asymptotic variance $\bbV_{\psi}$, we can construct an asymptotically valid confidence region of $\psi$ as follows. First, locally at $\hat{\psi}_{n}$, we can construct a nominal asymptotically $(1 - \alpha)$ confidence set $\mathcal{E}_{n, 1 - \alpha} \coloneqq \left\{ v \in \T_{\hat{\psi}_{n}} \Psi: n \langle v, \hat{\bbV}_{n}^{-1} v \rangle \leq \chi^{2}_{p, 1 - \alpha} \right\}$ in the tangent space $\T_{\hat{\psi}_{n}} \Psi$, where $\chi^{2}_{p, 1 - \alpha}$ is the lower $(1 - \alpha)$ quantile of the chi-squared distribution with $p$-degrees of freedom and $p$ is the dimension of $\Psi$. We can show that the set $\mathcal{C}_{n, 1 - \alpha} \coloneqq \Exp_{\hat{\psi}_{n}} (\mathcal{E}_{n, 1 - \alpha})$ that maps the set $\mathcal{E}_{n, 1- \alpha}$ defined on the tangent space $\T_{\hat{\psi}_{n}} \Psi$ back to the manifold $\Psi$ is an asymptotically valid $1 - \alpha$ confidence set of $\psi$. A proof sketch is as follows. Based on our assumptions, we have $n \langle \Exp_{\psi}^{-1} \hat{\psi}_{n}, \bbV_{\psi}^{-1} \Exp_{\psi}^{-1} \hat{\psi}_{n} \rangle \rightsquigarrow \chi^{2}_{p}$. Since our confidence set is defined on the tangent space $\T_{\hat{\psi}_{n}} \Psi$, to prove the asymptotic coverage, we first simultaneously parallel transport $\Exp_{\hat{\psi}_{n}}^{-1} \psi$ and $\hat{\bbV}_{n}$ to the tangent space $\T_{\psi} \Psi$ locally at the true parameter value $\psi$. Then, by the identity $\Exp_{\hat{\psi}_{n}}^{-1} \psi = - \Pi_{\psi}^{\hat{\psi}_{n}} \Exp_{\psi}^{-1} \hat{\psi}_{n}$ and consistency of $\hat{\psi}_{n}$ and $\hat{\bbV}_{n}$, the desired conclusion is proved. Here, we essentially need the following regularity conditions to ensure that both $\Exp_{\psi}^{-1} \hat{\psi}_{n}$ and $\Exp_{\hat{\psi}_{n}}^{-1} \psi$ are well defined: There exists an open neighborhood $U$ of $\psi_0$ and a constant $\rho > 0$ such that $\P (\hat{\psi}_n \in U) \rightarrow 1$ and for every $x \in U$, the exponential map $\Exp_x$ is a diffeomorphism from the open ball $B_\rho (0_x) \subset \T_x \Psi$ onto its image $N_x \coloneqq \Exp_x \{B_\rho (0_x)\}$, and for every $x \in U$, the set $N_x$ contains $\psi_0$.
In terms of computing $\hat{\bbV}_{n}$, such a covariance estimator can be constructed by the usual plug-in methods. For example, in a semiparametric model, if an influence function $\IF_{\psi} \in \T_{\psi (\theta)} \Psi$ for the estimator is available, in particular the efficient influence function when the estimator is efficient, one may plug in estimators of the unknown distribution and nuisance components to obtain
estimated influence functions
$\hat{\IF}_{\psi,i}\in \T_{\hat{\psi}_{n}} \Psi$, and then take their empirical covariance
\begin{align*}
\hat{\bbV}_n = \frac{1}{n} \sum_{i = 1}^{n} \Big( \hat{\IF}_{\psi,i} - \frac{1}{n} \sum_{j = 1}^{n} \hat{\IF}_{\psi, j} \Big)^{\otimes 2}.
\end{align*}
In finite-dimensional parametric models, one may instead estimate the covariance by the plug-in information bound $\dot{\psi} (\hat{\theta}_{n}) \hat{\bbG}_{\hat{\theta}_{n}}^{-1} \dot{\psi} (\hat{\theta}_{n})^*$, or equivalently by the usual sandwich or profile-score covariance estimator. These estimators are naturally viewed as operators on $\T_{\hat{\psi}_{n}}\Psi$, and consistency is understood after parallel transport to $\T_{\psi}\Psi$.
\end{remark}
With all the above concepts in place, we are finally ready to present the convolution theorem for parameters taking values in a manifold.
\begin{theorem}
\label{thm:par}
Let $\mathcal{P} \coloneqq \left\{ \P_{\theta}: \theta \in \Theta \right\}$ be a statistical model parameterized by $\theta \in \Theta$. Assume that $\mathcal{P}$ is DQM. $\psi (\theta)$ is the parameter of interest with $\psi$ satisfying Assumption~\ref{as:smooth transformation}. Further let $\hat{\psi}_{n}$ be a regular estimator sequence based on $n$ i.i.d. observations drawn from $\P_{\theta}$ with limiting distribution $\L_{\psi (\theta)}$. Then there exists a probability measure $\Delta_{\psi (\theta)}$ such that
\begin{equation}
\label{convolution}
\L_{\psi (\theta)} = Z_{\psi (\theta)} \ast \Delta_{\psi (\theta)}
\end{equation}
with $Z_{\psi_{\theta}} \sim \calN (0, \dot{\psi} (\theta) \bbG_{\theta}^{-1} \dot{\psi} (\theta)^{\ast})$ and $Z_{\psi (\theta)} \independent \Delta_{\psi (\theta)}$. That is, the limiting law $\L_{\theta}$ is represented as a convolution between a Gaussian measure with the covariance operator equal to the inverse Fisher information operator and a probability measure $\Delta_{\psi (\theta)}$ independent of that Gaussian measure.
\end{theorem}
The proof of Theorem~\ref{thm:par} can be found in Appendix~\ref{app:thm par}. The proof deviates from that for the normed linear parameter space due to the parallel transport operation in the definition of regular estimators; see Lemma~\ref{lem:Un} in Appendix~\ref{app:thm par}.
\begin{remark}
\label{rem:almost everywhere}
Historically, almost sure convolution theorem was also proved to argue that the point in the parameter space at which superefficiency phenomenon occurs has Lebesgue measure zero asymptotically; see \cite[Theorem~8.9]{van2000asymptotic} for a relatively recent reference. It seems that convolution theorem in the flavor of Theorem~\ref{thm:par} or the local asymptotic minimax theorem to be presented later receive more attention in modern mathematical statistics. For the sake of completeness, we generalize such a result to parameter manifolds in Appendix~\ref{app:almost everywhere}.
\end{remark}  
An important consequence of the convolution theorem is the ``calculus of influence functions'', an important consequence of efficiency theory that offers practicing statisticians a streamlined procedure to follow when trying to construct (nearly) efficient estimators.
\begin{lemma}
\label{lem:calculus of IF, parametric}
Let $\hat{\psi}_{n}$ be an asymptotically linear estimator of $\psi (\theta)$ based on $n$ i.i.d. observations drawn from $\P_{\theta}$ in the following sense: there exists a random element $\IF_{\psi} (\theta) \equiv \IF_{\psi} \in \T_{\psi (\theta)} \Psi$, such that
\begin{align*}
\sqrt{n} \Exp^{-1}_{\psi (\theta)} \hat{\psi}_{n} = \frac{1}{\sqrt{n}} \sum_{i = 1}^{n} \IF_{\psi, i} + o_{\P_{\theta}} (1).
\end{align*}
If $\hat{\psi}_{n}$ is further assumed to be regular (in the sense of Definition~\ref{def:riemann regular}), then
\begin{equation}
\label{par functional equation}
\dot{\psi} (\theta) = \E \left[ \IF_{\psi} \cdot \s \left( X; \theta \right) \right].
\end{equation}
\end{lemma}
The proof of the above lemma can be found in Appendix~\ref{app:calculus of IF, parametric}. The proof also highlights certain distinctions between linear spaces and manifolds. In the proof, we ought to show the following (see \eqref{distinction}):
\begin{equation}
\begin{split}
& \sqrt{n} \left( \Pi_{\psi (\theta_{n, h})}^{\psi (\theta)} \Exp_{\psi (\theta_{n, h})}^{-1} \hat{\psi}_{n} - \Exp_{\psi (\theta)}^{-1} \hat{\psi}_{n} \right) \\
& = - \nabla_{\theta} \psi (\theta) (h) + o_{\P_{\theta_{n, h}}} (1).
\label{diff}
\end{split}
\end{equation}
In the case of normed linear spaces, the LHS of the above display reduces to $\sqrt{n} (\psi (\theta_{n, h}) - \psi (\theta))$, which no longer involves the estimator $\hat{\psi}_{n}$. For parameter manifolds, no such an algebraic simplification can be made due to the non-zero curvature effect. Instead, we need to conduct a more delicate analysis using the geometric machinery introduced in Appendix~\ref{app:geometric analysis} to prove \eqref{diff}.
The statistical implication of the functional equation~\eqref{par functional equation} in Lemma~\ref{lem:calculus of IF, parametric} is important. When the parameter of interest belongs to a normed linear space, a functional equation of almost the same form is often used to find the corresponding influence function, which can in turn be exploited to construct RAL or even efficient estimators. Lemma~\ref{lem:calculus of IF, parametric} showed that the same calculus rule also holds for parameter manifolds.
\subsection{The Local Asymptotic Minimax (LAM) Theorem}
\label{sec:LAM theorem}
Convolution theorem restricts estimators to regular ones. An alternative estimator-agnostic approach to rigorously establishing efficiency bound is the Local Asymptotic Minimaxity (LAM) \cite{hajek1972local}. When the parameter space is linear, van Trees inequality has emerged as a convenient tool for proving LAM, circumventing the technically more difficult route of equivalent experiments \cite{gill1995applications, gassiat2024van, wahl2019van, takatsu2024generalized}. In this section, we first establish the intrinsic analogue of van Trees inequality in the spirit of \cite{gassiat2024van} and consequently establish LAM for parameter manifolds. It is worth noting that \cite{jupp2010van} has also established intrinsic van Trees inequality over parameter manifolds, under a different definition of bias from our choice in Definition~\ref{def:bias} and a stronger regularity condition about the statistical model. We therefore decide to provide our own version of intrinsic van Trees inequality.
Since van Trees inequality is a Bayesian version of the CRLB, we let $\pi$ denote a generic prior distribution over the parameter manifold $\Theta$ and $\E_{\pi}$ be the expectation taken w.r.t. the parameter $\theta$ under $\pi$. With slightly abuse of notation, we let $\pi$ also denote the probability density function corresponding to $\pi$, with the square root density as $\r_{\pi}$ and the score function as $\s_{\pi}$ respectively. We also denote the mean squared error of an estimator $\hat{\psi}$ of $\psi (\theta) \in \Psi$ under $\pi$ as
\begin{align*}
\gamma (\pi) \coloneqq \E_{\pi} \E [(\Exp_{\psi (\theta)}^{-1} \hat{\psi})^{\otimes 2}].
\end{align*}
First, we generalize van Trees inequality obtained in \cite{gassiat2024van} in intrinsic forms, which is instrumental for proving the LAM theorem.
\begin{lemma}
\label{lem:van Trees}
Let $\bbG_{\pi} \coloneqq \int_{\Theta} \s_{\pi} (\theta)^{\otimes 2} \diff \pi (\theta)$ be the Fisher information operator of the prior $\pi$. Assume $\bbG_{\pi}\succ 0$, $\psi$ is continuously differentiable on $\Theta$, the prior density $\pi$ is continuous and compactly supported, the statistical model $\calP = \{\P_\theta: \theta \in \Theta\}$ is dominated by the volume form and the corresponding Fisher information $\bbG_{\theta}$ is continuous in $\theta$, the model also satisfies Definition \ref{def:par dqm} for $\theta\in \Theta\cap\operatorname{Supp}(\pi)$, and $\gamma (\pi) $, $\E_{\pi} [\dot{\psi} (\theta)] $ and $\int_{\Theta} \bbG_\theta \diff \pi (\theta) $ exist with finite norm. Then,
\begin{equation}
\label{van Trees}
\begin{split}
\gamma (\pi) & \succeq \E_{\pi} \E [\nabla \Exp_{\psi (\theta)}^{-1} \hat{\psi}] \left( \bbG_{\pi} + \int_{\Theta} \bbG_\theta \diff \pi (\theta) \right)^{-1} \\
& \quad \left( \E_{\pi} \E [\nabla \Exp_{\psi (\theta)}^{-1} \hat{\psi}] \right)^{\ast}.
\end{split}
\end{equation}
\end{lemma}
The proof of Lemma~\ref{lem:van Trees} can be found in Appendix~\ref{app:finite-dimensional LAM}. It is not difficult to see that when $\pi$ is a degenerate measure over $\theta$, van Trees inequality \eqref{van Trees} reduces to CRLB \eqref{asymptotic CRLB}. Aided by this lemma, we can establish the LAM theorem for finite-dimensional parameter manifolds.
\begin{theorem}
\label{thm:finite-dimensional LAM}
Let $N_{\theta_{0}}$ be a neighborhood of $\theta_{0} \in \Theta$. Suppose that every $\P_{\theta} \in \calP$ is DQM as long as $\theta \in N_{\theta_{0}}$, that the Fisher information operator $\bbG_{\theta_0}$ is bounded, and that $\psi$ satisfies the conditions stated in Lemma~\ref{lem:van Trees}. For all estimator sequences $\psi_{n} \in \Psi$ of $\psi (\theta)$ indexed by sample size $n$, we have
\begin{align}
& \liminf_{c \rightarrow \infty} \liminf_{n \rightarrow \infty} \sup_{\theta \in N_{\theta_{0}}: \Vert \Exp^{-1}_{\theta_{0}} \theta \Vert \leq \frac{c}{\sqrt{n}}} \E_{\theta} (\Pi_{\psi (\theta)}^{\psi (\theta_{0})} \sqrt{n} \Exp^{-1}_{\psi (\theta)} \psi_{n})^{\otimes 2} \nonumber \\ 
& \succeq \E [Z_{\theta_{0}}^{\otimes 2}], \label{LAM} \\
& \text{where $\T_{\psi (\theta_{0})} \Psi \ni Z_{\theta_{0}} \sim \calN \left( 0, \dot{\psi} (\theta_{0}) \bbG_{\theta_{0}}^{-1} (\dot{\psi} (\theta_{0}))^{\ast} \right)$}. \nonumber
\end{align}
\end{theorem}
The proof of Theorem~\ref{thm:finite-dimensional LAM} is deferred to Appendix~\ref{app:finite-dimensional LAM}. If using the LHS of \eqref{LAM} as the asymptotic criterion, Hodges' estimator $\hat{\mu}_{n, \Hodges}$ of the \Frechet{} mean $\mu$ mentioned in Section~\ref{sec:hodges} clearly diverges, whereas the sample \Frechet{} mean $\hat{\psi}_{n}$ attains the asymptotic lower bound on the RHS of \eqref{LAM}. As is clear from the proof, similar to the CRLB \eqref{asymptotic CRLB}, a term related to the curvature will appear in the lower bound before taking large-$n$ limit, due to taking the derivative of the logarithmic map. The large-$n$ asymptotics take care of the extra curvature term.
\section{Efficiency Theory for Differentiable Functionals on Parameter Manifolds}
\label{sec:semipar}
In this section, we turn our attention to a class of semiparametric problems, namely estimating low-dimensional manifold-valued functionals over infinite-dimensional statistical models \cite{van1991differentiable}. Here, the underlying statistical model is parameterized by a possibly infinite-dimensional nuisance parameter $\theta \in \Theta$, $\calP = \{\P \equiv \P_{\theta}, \theta \in \Theta\}$, but the parameter of interest $\chi (\P): \calP \rightarrow \Psi$ is defined as a functional taking values in a $q$-dimensional manifold $(\Psi, g)$. Similarly, we can write $\chi (\P)$ as $\psi (\theta)$, as in Section~\ref{sec:par}. An example of such a parameter is the \Frechet{} mean $\chi (\P) \coloneqq \mu$ of $X \in \bbM$, without making concrete parametric modeling assumptions on $\P$, the distribution of $X$. Here we mainly focus on the case where the infinite-dimensional nuisance parameters lie in a normed linear space. We plan to tackle nonlinear infinite-dimensional nuisance parameter space in a separate work, as it requires introducing many new concepts that may complicate the presentation.
We first recall and introduce some notation to ease exposition. As in Section~\ref{sec:par} regarding the theory for regular parametric models, we assume that every $\P \in \calP$ is absolutely continuous with respect to some dominating measure and we denote the corresponding density function as $\p$. We recall that $\r \equiv \sqrt{\p}$ is the square-root density. Since no parametric modeling assumption is imposed on $\calP$, to generalize the efficiency theory from finite-dimensional to infinite-dimensional statistical models, a common strategy is resorting to low-dimensional parametric submodels in $\calP$ that locally intersect with $\P$. In general, one-dimensional parametric submodels are sufficient to help characterize efficiency bounds for differentiable functionals. To be precise, a (one-dimensional parametric) submodel (also referred to as a path) passing through $\P \in \calP$ is $t \mapsto \P_t: [0, \delta) \rightarrow \calP$ for some $\delta > 0$ with $\P_{t = 0} \equiv \P$. We also restrict our focus to submodels that are DQM at $\P$, in the sense of Definition~\ref{def:DQM parametric submodels} below. To reduce clutter, we sometimes abuse notation and simply write $\P_{t}$ as a submodel without adding additional ingredients in the above definition.
\begin{definition}
\label{def:DQM parametric submodels}
A submodel $\P_{t}$ locally at $\P \in \calP$ is said to be DQM at $t = 0$ with score function $\s$, if there exists some $\s \in \L_{2, 0} (\P)$ such that
\begin{equation}
\label{semipar: differentiable paths}
\int \left( \frac{\r_{t} (x) - \r (x)}{t} - \frac{1}{2} \s (x) \cdot \r (x) \right)^2 \diff x \rightarrow 0, \ \ \ \ \text{as $t \rightarrow 0$},
\end{equation}
where $\s$ is referred to as the score function associated with $\P_{t}$ and we let $\S \coloneqq \s (X)$.
\end{definition}
We denote a collection of submodels passing through $\P \in \calP$ that are DQM as $\mathscr{P}_{\P}$. We are now ready to define the model tangent space locally at $\P \in \calP$, corresponding to the collection of submodels $\mathscr{P}_{\P}$.
\begin{definition}
\label{def:semi tangent}
The model tangent space at $\P \in \calP$ corresponding to the collection of submodels $\mathscr{P}_{\P}$ is defined as
\begin{align*}
\Lambda_{\P} \coloneqq \mathrm{cl} \left[ \linspan \left\{ \s: \s \textrm{ induced by $\mathscr{P}_{\P}$} \right\} \right].
\end{align*}
\end{definition}
When no additional restriction is imposed on $\calP$ and $\mathscr{P}_{\P}$ is the set of all one-dimensional DQM submodels at $\P$ of $\calP$, the model is locally nonparametric and $\Lambda_{\P} \equiv L_{2, 0} (\P)$. In words, the model tangent space at $\P$ consists of all $\P$-mean zero squared integrable functions with respect to $\P$.
As mentioned, one of the most important insights garnered from semiparametric efficiency theory is the intimate connection between ``differentiability'' of the parameter of interest $\chi$ and the existence of RAL estimators for $\chi$; see the first part of Theorem 2.1 of \cite{van1991differentiable}. Since $\chi$ takes values in manifold $\Psi$, we generalize the definition of differentiable functionals as follows. The definition below essentially excludes the irregular case documented in \cite{hundrieser2024lower}, where the authors showed that the non-uniqueness of the population \Frechet{} mean precludes even uniformly consistent estimators.
\begin{definition}
\label{def:differentiable functionals}
$\chi \equiv \chi (\P)$ is said to be differentiable at $\P$ relative to $\mathscr{P}_{\P}$, if there exists a continuous and linear map $\dot{\chi} (\P): \Lambda_{\P} \rightarrow \T_{\chi (\P)} \Psi$ such that
\begin{align*}
t^{-1} \Exp_{\chi (\P)}^{-1} \chi (\P_t) \rightarrow \dot{\chi}_{\P} (\s),
\end{align*}
for every path $t \mapsto \P_{t}$ in $\mathscr{P}_{\P}$, as $t \rightarrow 0$. Here $\s$ is the score function induced by the given path $\P_{t}$, as defined in Definition~\ref{def:DQM parametric submodels}.
\end{definition}
\begin{remark}
When $\Psi = \bbR^{q}$, the new definition for differentiable functional reduces to the one in \cite{van1991differentiable}: there $\psi: \calP \rightarrow \bbR^{q}$ is said to be differentiable at $\P$ relative to $\mathscr{P}_{\P}$ if there exists a continuous linear map $\dot{\chi}_{\P}: L_2 (\P) \rightarrow \bbR^{q}$ such that  $t^{-1} \left( \chi (\P_t) - \chi (\P) \right) \rightarrow \dot{\chi}_{\P} (\s)$ as $t \rightarrow 0$.
\end{remark}
\begin{remark}
\label{rem:stat manifold}
The parameter of interest $\chi (\P)$ can also be viewed as a mapping from manifold to manifold. The former manifold in this case is the (potentially) infinite-dimensional statistical manifold $\calP$ with each element $\P$ identified with their square-root density $\r$, whilst the latter corresponds to the parameter manifold.
\end{remark}
Since $\dot{\chi} (\P)$ takes values in $\T_{\chi (\P)} \Psi$, which is locally isomorphic to $\bbR^{q}$, Riesz representation theorem implies the following identity:
\begin{equation}
\label{Riesz representator}
\langle \upsilon, \dot{\chi} (\P) (\s) \rangle \equiv \E \left[ \IF_{\chi} (\upsilon) \cdot \s (X) \right],
\end{equation}
where $\upsilon \in \T_{\chi (\P)}^{\ast} \Psi$ is any element in the cotangent space at $\chi (\P)$ and $\IF_{\chi} (\P) \equiv \IF_{\chi}$ is the unique element in $\Lambda_{\P}$, known as the (efficient) influence function/the canonical gradient. We let $\bbV_{\P} \coloneqq \E [\IF_{\chi}^{\otimes 2}]$ denote the semiparametric variance bound, which corresponds to the inverse of the Fisher information $\bbG_{\P}$ of the least favorable submodel (i.e., the parametric submodel whose score coincides with the canonical gradient up to a constant).
\begin{remark}
\label{rem:term}
In the sequel, we only add the qualification ``efficient'' for parameters with more than one influence function, with the efficient  influence function being the one with the smallest variance in positive semi-definite sense.
\end{remark}
Parallel to the development in Section~\ref{sec:par}, we also need to extend the concept of (locally) regular estimators of $\chi$ from finite-dimensional parametric models to infinite-dimensional semi- or non-parametric models.
\begin{definition}
\label{def:semipar regular estimators}
Let $\chi (\P) \in \Psi$ be the parameter of interest. Given any collection $\mathscr{P}_{\P}$ of submodels $\P_{t} \in \calP$ with $\P_{t = 0} = \P \in \calP$, and any fixed scalar $h$ and $\P_{h / \sqrt{n}} \in \calP$ for all sufficiently large $n$, a sequence of estimators $\hat{\chi}_{n}$ indexed by $n$ for $\chi (\P)$ is called (locally) regular at $\P$ relative to $\mathscr{P} _\P$ if
\begin{equation*}
\begin{split}
&\Pi_{\chi (\P_{h / \sqrt{n}})}^{\chi (\P)} \sqrt{n} \Exp_{\chi (\P_{h / \sqrt{n}})}^{-1} \hat{\chi}_n \rightsquigarrow_{\P_{h / \sqrt{n}}} Z_{\chi (\P)}, \,\,\,\, \\&\textrm{with } Z_{\chi (\P)} \sim \L_{\chi (\P)},
\end{split}
\end{equation*}
as $n \rightarrow \infty$, where $Z_{\chi (\P)}$ is a tight Borel measurable random element on $\T_{\chi (\P)} \Psi$ with law $\L_{\chi (\P)}$. This law does not vary with $h$.
\end{definition}
We now generalize the convolution theorem for differentiable functionals from normed linear spaces to manifolds.
\begin{theorem}
\label{thm:semipar convolution}
Suppose that $\chi: \calP \rightarrow \Psi$ is differentiable at $\P \in \calP$ relative to $\mathscr{P}_{\P}$ in the sense of Definition~\ref{def:differentiable functionals}. Assume that $\hat{\chi}_{n}$ is a regular estimator sequence of $\chi \equiv \chi (\P)$ relative to $\mathscr{P}_\P$, based on $n$ i.i.d. observations, with limiting law $\L_{\P}$. Then there exists a probability measure $\Delta_{\P}$ such that
\begin{equation}
\label{semipar convolution}
\L_{\P} = \calN (0, \bbV_{\P}) \ast \Delta_{\P}, \text{ with } \calN (0, \bbV_{\P}) \independent \Delta_{\P}.
\end{equation}
That is, the limit law $\L_{\P}$ is represented as a convolution between a mean-zero Gaussian measure with the semiparametric information bound as the covariance operator and a probability measure $\Delta_{\P}$ independent of that Gaussian measure.
\end{theorem}
The proof of Theorem~\ref{thm:semipar convolution} can be found in Appendix~\ref{app:semipar convolution}. The next result is a natural extension of the first part of Theorem~2.1 of \cite{van1991differentiable} to differentiable functionals on a manifold. The proof of Lemma~\ref{lem:calculus of IF, semiparametric} can be found in Appendix~\ref{app:calculus of IF, semiparametric}. 
\begin{lemma}
\label{lem:calculus of IF, semiparametric}
If there exists a sequence of RAL estimators $\hat{\chi}_{n}$ of the parameter of interest $\chi \equiv \chi (\P)$ and for any $\s \in \Lambda_{\P}$ of a submodel in $\mathscr{P}_{\P}$, $(\sqrt{n} \Exp_{\chi (\P)}^{-1} \hat{\chi}_{n}, n^{-1 / 2} \sum_{i = 1}^{n} \s (X_{i}))^{\top}$ jointly converges weakly, then $\chi$ is differentiable in the sense of Definition~\ref{def:differentiable functionals}.
\end{lemma}
We also establish the following lemma, which, as stated before Lemma~25.23 of \cite{van2000asymptotic}, justifies the exchangeability of ``canonical gradient'' and ``influence function''. These results together facilitate the calculus of influence functions for differentiable functionals over non- or semi-parametric models. A proof can be found in Appendix~\ref{app:iff}.
\begin{lemma}
\label{lem:iff}
Under the same assumptions as in Theorem~\ref{thm:semipar convolution}, for a differentiable functional $\chi$, the following are equivalent:
\begin{enumerate}[label = (\roman*)]
\item the estimator sequence $\hat{\chi}_n$ is regular at $\P$ relative to $\mathscr{P}_\P$ in the sense of Definition~\ref{def:semipar regular estimators}, with limiting law $\calN (0, \mathbb{V}_{\P})$;
\item $\sqrt{n} \Exp_{\chi(\P)}^{-1}\left(\hat{\chi}_n\right) = \dfrac{1}{\sqrt{n}} \sum\limits_{i=1}^n \IF_\chi \left(X_i\right) + o_\P(1)$.
\end{enumerate}
\end{lemma}
\begin{remark}
\label{rem:semiparametric models}
In this remark, we consider a common scenario in the application of efficiency theory, in which the data-generating probability distribution $\P$ has additional restrictions such that the corresponding model tangent space $\Lambda_{\P}$ is not the entire $L_{2, 0} (\P)$. Then there exists more than one influence function, and the unique efficient influence function $\IF_{\chi}$ is the one in $\Lambda_{\P}$. Let $\Pi_{\P} [\cdot \mid \Lambda_{\P}]$ denote the $L_{2} (\P)$-projection operator onto the model tangent space $\Lambda_{\P}$ \cite{ai2003efficient, fortunati2026nuisance}. We could first identify some influence function $\IF_{\chi}^{\dag}$ by finding one solution to the functional equation \eqref{Riesz representator}. To find the efficient influence function reaching the semiparametric efficiency bound, we can simply compute $\IF_{\chi} \equiv \Pi_{\P} [\IF_{\chi}^{\dag} \mid \Lambda_{\P}]$. Since influence functions are elements of $\T_{\psi} \Psi$ isomorphic to $\bbR^{p}$, the projection operation is well defined. In summary, the ``calculus of influence functions'' can be carried over to manifold-valued parameters of interest. In Section~\ref{sec:SIM}, we will present an example to instantiate this calculus rule for a manifold-valued parameter of interest.
\end{remark}
Our last main result extends the LAM theorem to differentiable functionals defined on manifolds, parallel to Theorem~\ref{thm:finite-dimensional LAM} for parameters of parametric statistical models. Intuitively, the semiparametric efficiency bound corresponds to the CRLB of the least-favorable submodel, so the efficiency bound is expected to $\bbV_{\P}$. A proof is deferred to Appendix~\ref{app:functional LAM}.
\begin{theorem}
\label{thm:semipar LAM} 
Suppose that $\chi$ is differentiable at $\P$ relative to $\mathscr{P}_{\P}$ with model tangent space $\Lambda_\P$ in the sense of Definition~\ref{def:differentiable functionals} with influence function $\IF_{\chi}$ and bounded $\bbV_{\P}$. For every $\s \in \Lambda_\P$, let $t \mapsto \P_{t, \s}$ be a DQM one-dimensional submodel in $\mathscr{P}_{\P}$ with score $\s$, and put $\P_{n, \s} \coloneqq \P_{1 / \sqrt{n}, \s}^{\otimes n}$. If $\Lambda_\P$ is a linear subspace of $L_{2, 0} (\P)$, then for every estimator sequence $\hat{\chi}_n$,
\begin{equation*}
\begin{split}
&\sup_{I} \liminf_{n \rightarrow \infty} \sup_{\s \in I} \E_{\P_{n, \s}} \left\{ \Pi_{\chi (\P_{1 / \sqrt{n}, \s})}^{\chi (\P)} \sqrt{n} \Exp_{\chi (\P_{1 / \sqrt{n}, \s})}^{-1} \hat{\chi}_{n} \right\}^{\otimes 2} \\&\succeq \E (Z^{\otimes 2}), \text{ where } Z \sim \calN \left( 0, \bbV_{\P} \right).
\end{split}
\end{equation*}
Here, the first supremum is taken over all finite subsets $I \subset \Lambda_\P$.
\end{theorem}
\begin{remark}
\label{rem:eff score}
For completeness, we discuss a special case where the data-generating distribution $\P_{\chi, \eta}$ is parameterized by a manifold-valued parameter of interest $\chi$ and a (possibly infinite-dimensional) nuisance parameter $\eta$ in a linear space. In such a simplified scenario, it is often easier to derive the efficient score $\s_{\chi, \eff}$, defined as the projection of the usual score $\s_{\chi} = \frac{1}{2} \frac{\dot{\r}_{\chi, \eta}}{\r_{\chi, \eta}}$ onto the orthocomplement to the nuisance tangent space (i.e. the tangent space corresponding to the score with respect to $\eta$), where $\dot{\r}_{\chi, \eta}$ denotes the differential of $\r_{\chi, \eta}$ with respect to $\chi$. It is a known fact that efficient influence function and efficient score function with respect to $\chi$ can be identified as follows:
\begin{equation}
\label{score IF}
\begin{split}
&\S_{\chi, \eff} \equiv \left( \{\E [\IF_{\chi} \otimes \IF_{\chi}]\}^{-1} \IF_{\chi} \right)^{\ast} \text{ and } \\&\IF_{\chi} \equiv \left( \{\E [\S_{\chi, \eff} \otimes \S_{\chi, \eff}]\}^{-1} \S_{\chi, \eff} \right)^{\ast}.
\end{split}
\end{equation}
A related example can be found in Section~\ref{sec:SIM}, in which one can derive the semiparametric efficiency bound via the strategy outlined above.
\end{remark}
\begin{remark}
\label{rem:nonlinear nuisance}
Before moving to applications of our general theoretical results, we discuss possible directions of dealing with nonlinear infinite-dimensional nuisance parameter spaces. When the infinite-dimensional nuisance parameter is a probability density function, a well established approach is to view the space of probability distributions as an (infinite-dimensional) manifold, with each point on the manifold the square-root of corresponding probability density function. In fact, this manifold is a Hilbert sphere and is a subject of study in the field of information geometry \cite{amari2012differential}. In this special case, it is well known that the tangent space at a particular point corresponds to the model tangent space mentioned in our paper and tangent vectors are score functions associated with parametric submodels. For more general nonlinear infinite-dimensional nuisance parameter spaces, as pointed out by a referee, one could in principle try to approximate the underlying nonlinear space locally using a Hilbert space, as suggested in Chapter~25.5 of \cite{van2000asymptotic}. We leave this direction to future work, as it seems to be not straightforward to handle the approximation error in a unified manner.
\end{remark}
\section{Applications}
\label{sec:examples}
In this section, we demonstrate how the general theory developed thus far can be applied in concrete examples. In the first example, we demonstrate that results from the previous section can be used to establish the asymptotic efficiency of the MLE (when the model is correctly specified) and the semiparametric efficiency of the sample \Frechet{} mean estimator over manifold that is equipped with a Riemannian metric $d$ such that $d^{2}$ is geodesically convex w.r.t. the parameter. Though as expected, these efficiency results seem to be new to the literature to the best of our knowledge. In the second example, we revisit the classical single-index model, for which the sample space is Euclidean but the unknown regression coefficient vector $\beta \in \bbR^{p}$ belongs to the unit hemisphere of dimension $p - 1$. Results from the second example are not entirely new. Nonetheless, our general framework developed in Sections~\ref{sec:par} and~\ref{sec:semipar} unifies these results using the powerful differential geometric language.
\subsection{Semiparametric Efficiency Bound for Parameters Defined via Geodesically Convex \texorpdfstring{$M$}{M}-Estimation}
\label{sec:M-estimation}
The problem of $M$-estimation, or empirical risk minimization, is pervasive in statistics and machine learning, including MLE, Generalized Method-of-Moments, and many other common estimation schemes as special cases. $M$-estimation pertains to minimizing a population risk function defined as follows:
\begin{equation*}
M (\theta) \coloneqq \E [m (X; \theta)]
\end{equation*}
where $m: \bbX \times \Theta \rightarrow \bbR_{\geq 0}$ is referred to as the loss function. The distribution of $X \in \bbX$ is denoted by $\P$. Let
\begin{equation}
\label{pop target}
\theta_{0} \coloneqq \arg \min_{\theta \in \Theta} M (\theta)
\end{equation}
be the true value of the target parameter. To estimate $\theta_{0}$, one can solve the empirical version of \eqref{pop target}
\begin{equation}
\label{emp target}
\hat{\theta}_{n} \coloneqq \arg\min_{\theta \in \Theta} \hat{M}_{n} (\theta) \coloneqq \frac{1}{n} \sum_{i = 1}^{n} m (X_{i}; \theta)
\end{equation}
based on $\{X_{i}\}_{i = 1}^{n}$, a sample of size $n$ drawn independently from $\P$. When both the sample space $\bbX$ and the parameter space $\Theta$ are Euclidean spaces, the theory of $M$-estimation is quite mature.
When $\Theta$ or $\bbX$ or both are manifolds, \cite{brunel2023geodesically} has exhibited the $\sqrt{n}$-consistency and the asymptotic linearity of $\hat{\theta}_{n}$, under the assumption that $m (x; \cdot)$ is almost surely geodesically convex as defined below.
\begin{definition}
A function $f: \Theta \rightarrow \bbR$ is called geodesically convex if for all $\theta_{1}, \theta_{2} \in \Theta$, $\gamma: [0, 1] \rightarrow \Theta$ a geodesic connecting $\theta_{1}$ and $\theta_{2}$ and $t \in [0, 1]$, we have $f (\gamma(t)) \leq (1 - t) f(\theta_{1}) + t f (\theta_{2})$.
\end{definition}
We first restate the following result from \cite{brunel2023geodesically}.
\begin{proposition}
\label{prop:M-estimation upper bound}
Suppose that the following hold:
\begin{itemize}
\item $m (x; \cdot)$ is geodesically convex almost surely in $x \in \bbX$;
\item $\Theta$ is a manifold of a fixed dimension $p$;
\item $M$ has a unique minimizer $\theta_{0} \in \Theta$;
\item $M$ is twice differentiable at $\theta_{0}$ and $\nabla^{2} M (\theta_{0})$ is positive definite;
\item Let $\phi (x; \theta) \coloneqq \nabla m (x; \theta)$. There exist $\eta > 0$ such that $\E [\Vert \phi (X; \theta) \Vert_{\theta_{0}}^{2}] < \infty$, for all $\theta \in \{\theta \in \Theta: d (\theta, \theta_{0}) \leq \eta\}$.
\end{itemize}
Then as $n \rightarrow \infty$, $\hat{\theta}_{n}$ has the following asymptotically linear representation
\begin{equation*}
\sqrt{n} \Exp^{-1}_{\theta_{0}} \hat{\theta}_{n} = \frac{1}{\sqrt{n}} \sum_{i = 1}^{n} \{\nabla^{2} M (\theta_{0})\}^{-1} \cdot \phi (X_{i}; \theta_{0}) + o_{\P} (1).
\end{equation*}
Let $V_{\theta_{0}} \coloneqq \{\nabla^{2} M (\theta_{0})\}^{-1} \cdot \E [\phi (X; \theta_{0})^{\otimes 2}] \cdot \{\nabla^{2} M (\theta_{0})\}^{-1}$. In other words, as $n \rightarrow \infty$,
\begin{align*}
\sqrt{n} \Exp^{-1}_{\theta_{0}} \hat{\theta}_{n} \rightsquigarrow_{\P} Z, \textrm{ where $Z \in \T_{\theta_{0}} \Theta$ and $Z \sim \calN (0, V_{\theta_{0}})$.}
\end{align*}
\end{proposition}
\begin{remark}
\label{rem:Hessian}
The conditions in Proposition~\ref{prop:M-estimation upper bound} imply a nondegenerate Hessian. We provide the following examples that justify such conditions.
\begin{itemize}
    \item \Frechet{} mean on a Hadamard manifold
    \begin{itemize}
        \item Let $\Theta$ be a Hadamard manifold, and consider the squared-distance loss $m (x ; \theta)=\frac{1}{2} d^2(x, \theta)$ where $d$ is the Riemannian metric of $\Theta$. Then $M (\theta) = \E \{m (X; \theta)\}$. Hadamard manifolds are complete, simply connected, with nonpositive sectional curvature. There is no cut locus on $\Theta$. The squared distance $\theta \mapsto m (x; \theta)$ is globally smooth and geodesically convex, and in many cases geodesically strictly convex. If the distribution of $X$ is not concentrated in a way that creates degeneracy (if the law of $X$ is not supported on a lower-dimensional geodesic subset that leaves some tangent direction uninformative), then the \Frechet{} mean $\theta_0$ is unique, $M$ is twice differentiable at $\theta_0$, and the Hessian is positive definite.
    \end{itemize}
    \item Sphere $\mathbb{S}^d$ : intrinsic mean away from antipodal degeneracy
    \begin{itemize}
        \item Take $\Theta = \mathbb{S}^d$ with the standard Riemannian metric $d$ and $m(x ; \theta)=\frac{1}{2} d^2(x, \theta)$. On the sphere, global smoothness fails because of the cut locus, the antipode of $\theta$ is singular for the squared distance. However, if the distribution of $X$ is supported in an open hemisphere or in a geodesically convex ball of radius less than $\pi / 2$ (the intrinsic mean is still uniquely defined), and the distribution is not concentrated on a lower-dimensional great subsphere, then $M$ is twice differentiable near the true $\theta_{0}$, and the Hessian is positive definite. If the distribution is symmetric about two antipodal points, then uniqueness and Hessian definiteness can fail completely.
    \end{itemize}
    \item SPD manifold with affine-invariant or log-Euclidean type losses
    \begin{itemize}
        \item Let $\Theta=\operatorname{SPD}(m)$, the manifold of symmetric positive definite matrices. $m(X ; \Theta)=\frac{1}{2} d^2(X, \Theta)$, where $d$ is chosen to be the affine-invariant Riemannian distance, or similarly a log-Euclidean criterion. Because $\mathrm{SPD}(m)$ is geodesically complete, under the affine-invariant metric it is a Hadamard manifold, squared-distance losses are globally smooth, and the intrinsic mean is often uniquely defined, then $M$ is twice differentiable at the intrinsic mean $\theta_0$, and $\nabla^2 M\left(\theta_0\right)$ is positive definite as long as the distribution of $X$ is not degenerate in a way that leaves some symmetric matrix direction uninformative. This example covers covariance-matrix-type objects, diffusion tensors, and other matrix-valued data.
    \end{itemize}
    \end{itemize}
\end{remark}
\subsubsection{Maximum likelihood estimators on parameter manifolds}
\label{sec:MLE manifolds}
\leavevmode
Maximum likelihood estimation is a special case of the above $M$-estimation problem. Suppose that $X \sim \P_{\theta_{0}}$ with probability density function $\p (\cdot; \theta_{0})$. If we identify $m (\cdot; \theta) \equiv - \log \p (\cdot; \theta)$, then the $M$-estimator $\hat{\theta}_{n}$ reduces to the MLE and $\phi (\cdot; \theta)$ corresponds to the score function. As an immediate corollary of Proposition~\ref{prop:M-estimation upper bound}, we have the following result.
\begin{corollary}
\label{cor:MLE upper bound}
Suppose that $X \sim \P_{\theta_{0}}$ for $\theta_{0} \in \Theta$ and $m (\cdot; \theta) \equiv - \log \p (\cdot; \theta)$. Under the assumptions of Proposition~\ref{prop:M-estimation upper bound}, we have
\begin{align*}
\sqrt{n} \Exp^{-1}_{\theta_{0}} \hat{\theta}_{n} = \frac{1}{\sqrt{n}} \sum_{i = 1}^{n} \phi (X_{i}; \theta_{0}) + o_{\P} (1),
\end{align*}
and thus $\sqrt{n} \Exp^{-1}_{\theta_{0}} \hat{\theta}_{n} \rightsquigarrow_{\P} Z$, where $Z \in \T_{\theta_{0}} \Theta$ and $Z \sim \calN (0, \mathbb{V}_{\theta_{0}})$ with $\mathbb{V}_{\theta_{0}} \equiv \{\E [\phi (X; \theta_{0})^{\otimes 2}]\}^{-1} \equiv \bbG_{\theta_{0}}^{-1}$. Recall that $\bbG_{\theta_{0}}$ is the Fisher information operator.
\end{corollary}
It is of statistical interest whether $V_{\theta_{0}}$ given in the above corollary, corresponding to the inverse Fisher information operator, is the smallest possible variance-covariance operator that an estimator of $\theta_{0}$ can possibly have. Indeed, both Theorems~\ref{thm:par} and~\ref{thm:finite-dimensional LAM} in Section~\ref{sec:par} firmly address this question.
\begin{theorem}
\label{thm:MLE lower bound}
Under the assumptions of Corollary~\ref{cor:MLE upper bound}, the MLE $\hat{\theta}_n$ of $\theta_0$ is asymptotically efficient in the following sense: either
\begin{enumerate}[label = \emph{(\roman*)}]
\item $\hat{\theta}_n$ is asymptotically efficient among all estimators that are regular at $\theta_0$ relative to the parametric geodesic local alternatives $\theta_{n, h} \coloneqq \Exp_{\theta_{0}} \left( \frac{h}{\sqrt{n}} \right)$, $h \in \T_{\theta_0} \Theta$, in the sense of Definition~\ref{def:riemann regular}; or
\item $\hat{\theta}_n$ is asymptotically efficient in local asymptotic minimax sense of Theorem~\ref{thm:finite-dimensional LAM} for the parametric model $\calP = \{\P_\theta: \theta \in \Theta\}$.
\end{enumerate}
\end{theorem}
\begin{proof}
The first statement immediately follows from Theorem~\ref{thm:par}, if the MLE $\hat{\theta}_{n}$ is regular. Following Proposition~\ref{prop:M-estimation upper bound}, we have, under $\P_{\theta_{n, h}}$ with $\theta_{n, h} \coloneqq \Exp_{\theta_{0}} \left( \frac{h}{\sqrt{n}} \right)$,
\begin{align*}
\sqrt{n} \Exp^{-1}_{\theta_{n, h}} \hat{\theta}_{n} \rightsquigarrow_{\P_{\theta_{n, h}}} Z,
\end{align*}
where $Z \sim \calN (0, \mathbb{V}_{\theta_{0}})$ is independent of $h$. Thus $\hat{\theta}_{n}$ is indeed regular.
By checking the conditions in Theorem~\ref{thm:finite-dimensional LAM}, the second statement is a direct consequence of Theorem~\ref{thm:finite-dimensional LAM} by taking $\psi$ as the identity map so $\psi (\theta) \equiv \theta$.
\end{proof}
\begin{remark}
\label{rem:riemmanian gaussian}
Recall Example~\ref{ex:riemannian gaussian} about the Riemannian Gaussian distribution where the sample space and the parameter space coincide as the manifold $\bbM$. Theorem~\ref{thm:MLE lower bound} implies that given $n$ i.i.d. samples $\{X_{i}\}_{i = 1}^{n}$, the sample \Frechet{} mean estimator $\hat{\mu}_{n}$ is an asymptotically efficient estimator of $\mu$ if the metric $d^{2} (x; \cdot)$ is geodesically convex almost surely in $x \in \bbM$.
\end{remark}
\subsubsection{The empirical \Frechet{} mean estimator}
\label{sec:Frechet mean}
\leavevmode
Another important special case of geodesically convex $M$-estimation is the \Frechet{} mean $\mu_{0} \coloneqq \arg\min_{\mu \in \bbM} \E \frac{1}{2} d^{2} (X, \mu)$ over the manifold $\bbM$ with $\frac{1}{2} d^{2} (x, \cdot)$ geodesically convex. Again, the sample space and the parameter space coincide in this example. Here the loss function $m \equiv \frac{1}{2} d^{2}$ and $M (\mu) \equiv \E \frac{1}{2} d^{2} (X, \mu)$. $\mu$ can be viewed as a functional $\mu: \calP \rightarrow \bbM$ that maps the observed data distribution to manifold $\bbM$.
\begin{theorem}
\label{thm:Frechet mean}
Suppose that the assumptions of Proposition~\ref{prop:M-estimation upper bound} hold with $m$ replaced by $d^{2}$, and for any $\P \in \calP$, let $\mathscr{ P}_\P$ be the collection of all one-dimensional DQM submodels through $\P$ contained in the model $\calP$. Then, $\mu_{0} \equiv \mu (\P)$ is differentiable in the sense of Definition~\ref{def:differentiable functionals} and has the unique influence function
\begin{align*}
\IF_{\mu_{0}} = \{\E [\nabla \Exp_{\mu_{0}}^{-1} X]\}^{-1} \cdot \Exp_{\mu_{0}}^{-1} X.
\end{align*}
Among all regular estimators of $\mu$, the empirical \Frechet{} mean $\hat{\mu}_{n}$ has the asymptotically linear representation
\begin{align*}
\sqrt{n} \Exp^{-1}_{\mu_{0}} \hat{\mu}_{n} = \frac{1}{\sqrt{n}} \sum_{i = 1}^{n} \IF_{\mu_{0}, i} + o_{\P} (1),
\end{align*}
and attains the semiparametric efficiency bound $\mathbb{V}_{\mu_{0}} = \E [\IF_{\mu_{0}}^{\otimes 2}]$.
\end{theorem}
In Appendix~\ref{app:frechet mean}, we have established the regularity of the empirical \Frechet{} mean estimator. The differentiability of $\mu_{0}$ is proved in Appendix~\ref{app:Frechet mean}. The derivation of $\IF_{\mu_{0}}$ is then immediate by noticing that $\nabla \left( \frac{1}{2} d^{2} (X, \mu_{0}) \right) \equiv - \, \Exp^{-1}_{\mu_{0}} X$.
\begin{remark}
\label{ex:non-asymptotic}
Consider the special case of $\bbM$ being the space of all symmetric positive definite matrices of size $m \times m$ equipped with the Fisher-Rao metric $d$. Suppose that we are interested in estimating the \Frechet{} mean $\mu \in \bbM$. Let $\{X_{i} \in \bbM\}_{i = 1}^{n} \overset{\rm i.i.d.}{\sim} \P_{\mu}$. In this example, the sample \Frechet{} mean $\hat{\mu}_{n} \coloneqq \arg\min_{\mu \in \bbM} \sum_{i = 1}^n \frac{1}{2} d^{2} (X_i, \mu)$. The existence and uniqueness of $\hat{\mu}_{n}$ have been proved in Proposition 1.7 in \cite{sturm2003probability}. Based on Theorem~\ref{thm:Frechet mean}, the asymptotic covariance of $\hat{\mu}_{n}$ should be $\mathbb{V}_{\mu_{0}} = \E [\IF_{\mu_{0}}^{\otimes 2}]$, which is in accordance with Section 6.3 of \cite{pennec2019curvature} and Theorem 2.2 of \cite{bhattacharya2005large}. See Appendix~\ref{app:Frechet mean} for detailed calculations.
\end{remark}
\begin{remark}
\label{rem:AIPW}
In this remark, we point out an immediate consequence of Theorem~\ref{thm:Frechet mean}. When $X$ is subject to the missing-at-random (MAR) missing data mechanism, and we additionally observe $(Z, R)$ where $Z$ denotes the covariates and $R$ the binary missing data indicator with $R = 1$ meaning the observables, the efficiency theory developed here suggests that the unique influence function of the true \Frechet{} mean $\mu_{0}$ of $X$ takes the following form:
\begin{align}
& \IF_{\mu_{0}} = \E^{-1} \Big\{ \frac{R}{\pi (Z)} \nabla \Exp_{\mu_{0}}^{-1} (X) \Big\} \label{IF_MAR} \\
& \times \Big\{ \frac{R}{\pi (Z)} \Exp_{\mu_{0}}^{-1} X - \left( \frac{R}{\pi (Z)} - 1 \right) \E (\Exp_{\mu_{0}}^{-1} X \mid Z, R = 1) \Big\}. \nonumber
\end{align}
When $\bbM$ reduces to Euclidean space, and thus $\Exp_{\mu_{0}}^{-1} (X) = X - \mu_{0}$, $\IF_{\mu_{0}}$ reduces to the influence function of the mean of $X$ subject to MAR, originally derived in the seminal work of \cite{robins1994estimation}. By constructing estimating equations using $\IF_{\mu_{0}}$, we can also obtain estimators of $\mu_{0}$ that resemble the popular one-step/AIPW or double machine learning estimators in the classical case \cite{robins1994estimation, graham2026towards, chernozhukov2018double}. In contrast to the classical case (expectation of $X$ when $X$ lies in a linear space), however, estimators of $\mu_{0}$ with closed-form expression generally do not exist and we need to adopt the localized double machine learning (LDML) estimation strategy and its generalizations for quantile treatment effect, developed recently in \cite{kallus2024localized, zhang2026higher}. Following the LDML approach, we split the sample into three folds $I_{1}, I_{2}, I_{3}$. We first estimate the missing data probability $\pi$ by standard regression techniques using $I_{3}$, with the resulting estimator denoted by $\hat{\pi}$. We then obtain an initial estimator $\hat{\mu}_{\rm init}$ of $\mu_{0}$ by solving the following inverse probability weighting estimating equation in sample $I_{2}$:
\begin{align*}
\frac{1}{n} \sum_{i \in I_{2}} \frac{R_{i}}{\hat{\pi} (Z_{i})} \Exp_{\hat{\mu}_{\rm init}}^{-1} X_{i} = 0.
\end{align*}
With $\hat{\mu}_{\rm init}$, we can in turn estimate the outcome regression function $\E (\Exp_{\mu_{0}}^{-1} X \mid Z, R = 1)$ by $\hat{\E} (\Exp_{\hat{\mu}_{\rm init}}^{-1} X \mid Z, R = 1)$, again using standard techniques. Eventually, we solve the following non-normalized estimating equation based on $\IF_{\mu_{0}}$ to obtain the one-step estimator $\hat{\mu}$ using the sample in $I_{1}$:
\begin{equation*}
\begin{split}
\frac{1}{n} \sum_{i \in I_{1}} &\frac{R_{i}}{\hat{\pi} (Z_{i})} \Exp_{\hat{\mu}}^{-1} (X_{i}) \\&- \left( \frac{R_{i}}{\hat{\pi} (Z_{i})} - 1 \right) \hat{\E} (\Exp_{\hat{\mu}_{\rm init}}^{-1} X \mid Z = Z_{i}, R = 1) = 0.
\end{split}
\end{equation*}
\end{remark}
\begin{remark}
\label{rem:linbo}
The above result also relates to Theorem 5 in \href{https://arxiv.org/pdf/2101.01599v1}{the first arXiv version} of \cite{lin2023causal}. In Theorem~5 therein, the authors analyze the statistical properties of an estimator of $\Exp_{\mu}^{-1} \mu_{0}$ based on \eqref{IF_MAR} given any reference point $\mu \in \bbM$, but did not prove whether that estimator achieves the efficiency bound. Since we can define and establish the efficiency bound of $\mu_{0}$ via \eqref{IF_MAR}, the efficiency bound of $\Exp_{\mu}^{-1} (\mu_{0})$ can be established by using its influence function $\IF_{\Exp_{\mu}^{-1} (\mu_{0})} = \nabla \Exp_{\mu}^{-1} (\mu_{0}) \cdot \IF_{\mu_{0}}$. Finally, it should be noted that in the published version of \cite{lin2023causal}, the authors focus only on the case where $X$ is a (probability) distribution function. By metricizing the space of distribution functions with the Wasserstein distance, one equips the space of distribution functions with a formal Riemannian structure \cite{ambrosio2005gradient}. However, the theoretical framework developed here cannot yet be generalized to handle this particular problem, due to the infinite-dimensional nature of $X$ when it is a probability distribution function. As pointed out by a referee, following Chapter~7 of \cite{le2000asymptotics}, it is possible to generalize DQM to only consider non-singular part of the submodels as long as the singular part decays to zero at a sufficiently fast rate. But it is difficult to verify that a sufficiently rich and relevant collection of submodels even satisfies such a generalized DQM definition and that the target functional is differentiable along those paths. On infinite-dimensional sample spaces, many natural perturbations of probability laws are mutually singular. Submodels that are DQM will be necessarily more restrictive compared to the finite-dimensional case. Thus, further investigation is necessary to delineate the statistical consequence of restricting only to a small set of submodels for which DQM can be satisfied when tackling this more nuanced problem.
\end{remark}
\subsection{Semiparametric Efficiency Bound for Parameters of Single-Index Models}
\label{sec:SIM}
In this section, we apply the general theory developed thus far to the problem of deriving semiparametric efficiency bound for the regression coefficients in the following single-index model (SIM):
\begin{equation}
\label{SIM}
Y = g (\beta^{\top} X) + \epsilon, \text{ with } \E (\epsilon \mid X) = 0 \text{ almost surely},
\end{equation}
where the index parameter $\beta \in \bbR^d$ is our parameter of interest and the unknown link function $g: \bbR \rightarrow \bbR$ with $g \in \mathcal{G}$ is the nuisance parameter, and $\epsilon$ is the homoscedastic white noise with finite variance $\sigma^{2}$. The corresponding data generating law can then be parameterized as $\P_{\beta, g}$. Here $\mathcal{G}$ denotes the class of all real-valued functions with absolutely continuous first derivative on its domain $\mathbb{D} \coloneqq \{\beta^{\top} x: x \in \bbX \subseteq \bbR^{d}, \beta \in \calB\}$, where $\bbX$ is a bounded subset of $\bbR^{d}$. To ensure identification, the parameter space $\calB$ of $\beta$ is restricted to be the $(d - 1)$-dimensional unit hemisphere, which is a manifold:
\begin{equation}
\label{hemisphere}
\begin{split}
\calB& \coloneqq \{ \beta \equiv (\beta_1, \cdots, \beta_d)^{\top} \in \bbR^d: \Vert \beta \Vert = 1, \beta_1 \geq 0\} \\&\subset \bbS^{d - 1} \subset \bbR^{d}.
\end{split}
\end{equation}
We assume that the true regression coefficient $\beta_{0}$ satisfies $\beta_{0, 1} > 0$ to avoid lying on the boundary of the parameter manifold.
\subsubsection*{The difficulty without the new framework}
\leavevmode
Though seemingly innocuous, deriving semiparametric efficiency bound for $\beta$ is not as straightforward as it may seem to be -- we quote the following sentence from \cite[page 1590]{kuchibhotla2020efficient}:
\begin{quote}
\emph{... \cite{cui2011efm} points out that the assumption $\beta \in \bbS^{d - 1}$ makes the parameter space irregular and the construction of paths on the sphere is hard. In this paper, we construct paths on the unit sphere to study the semiparametric efficiency of the finite dimensional parameter ...}
\end{quote}
We will reveal the paths constructed in \cite{kuchibhotla2020efficient} later in Remark~\ref{rem:kp}. The main difficulty lies in the nonlinearity of the parameter space $\calB$ of $\beta$. The usual ``trick'' of constructing paths, e.g. $\beta_{t, h} = \beta + t \cdot h$, simply does not work, because such $\beta_{t, h}$ is outside of $\calB$ for any $t \neq 0$. By definition, the parametric model induced by such a $\beta_{t, h}$ cannot be a valid parametric submodel. Hence any results related to efficiency theory based on such a ``parametric submodel'' are at least mathematically non-rigorous. In the next subsection, we will demonstrate that this problem is, in a sense, ``geometrical'' rather than ``hard'' -- the construction of parametric submodels can be conceptually simplified by adopting the differential-geometric language such as the exponential and logarithmic maps. Our theoretical framework licenses the use of the powerful differential-geometric language.
\subsubsection*{The construction of paths under the new framework}
\leavevmode
Given the efficiency theory over parameter manifolds developed so far, we argue that ``the construction of paths on the sphere'' suddenly becomes easy, at least conceptually. 
To see this, we shall interpret $\beta$ as a representation of points in $\calB$ using a particular coordinate system. Then we can simply construct a path by perturbing the parameter of interest via $\beta_{t, h} \coloneqq \Exp_{\beta_{0}} (t \cdot h)$ for $t \in \bbR$ and $h \in \T_{\beta_{0}} \calB$, which ensures that $\beta_{t, h}$ still lies in $\calB$ after the perturbation. The nuisance parameters $\p_{X}, \p_{\eps \mid X}$, and $g$ can be perturbed as usual by the same one-dimensional parameter $t$. We denote the induced parametric submodel by $\P_{t}$ with $\P_{t = 0} \equiv \P$. The only task left is to compute the perturbation $\beta_{t, h}$ using the structure of $\calB$. To this end, we simply recall the following elementary result from differential geometry \cite{lee2018introduction}.
\begin{lemma}
\label{lem:exp_sphere}
Given a $\beta \in \calB$, the tangent space $\T_{\beta} \calB$ of the unit hemisphere $\calB$ can be explicitly written as follows:
\begin{align*}
\T_{\beta} \calB = \left\{ \eta \in \bbR^{d}: \eta^{\top} \beta = 0 \right\}.
\end{align*}
Furthermore, the exponential map on $\calB$ has the following explicit form: for $\beta \in \calB$, $t \in \bbR$, and $h \in \T_{\beta} \calB$ such that $h \neq 0$,
\begin{equation}
\label{exp_sphere}
\Exp_{\beta} (t \cdot h) = \cos (t \cdot \Vert h \Vert) \beta + \sin (t \cdot \Vert h \Vert) \frac{h}{\Vert h \Vert},
\end{equation}
and as a consequence
\begin{equation}
\label{deriv_exp_sphere}
\left. \frac{\diff}{\diff t} \Exp_{\beta} (t \cdot h) \right\vert_{t = 0} = h.
\end{equation}
\end{lemma}
It follows from Lemma~\ref{lem:exp_sphere} that the score function at $\P$ reads as: for any $h \in \T_{\beta_{0}} \calB$,
\begin{equation}
\label{SIM score}
\begin{split}
\S_{\beta_{0}} (h) &\equiv \s (X, Y; \beta_{0}) (h) \\&= - \ell_{\eps \mid X}' (Y - g_{0} (X^{\top} \beta_{0}), X) g_{0}' (X^{\top} \beta_{0}) X^{\top} h,
\end{split}
\end{equation}
where $\ell_{\eps \mid X}$ is the log conditional probability density function of $\eps$ given $X$ and $\ell_{\eps \mid X}'$ denotes its derivative. 
As the corresponding nuisance model tangent space $\Lambda_{\P, \nuis}$ is a well-known result for single-index models (see e.g. the set-off equation after equation (13) in \cite{kuchibhotla2020efficient}), the efficient score function of $\beta_{0}$ is simply, for any $h \in \T_{\beta_{0}} \calB$,
\begin{equation}
\label{SIM eff score}
\begin{split}
\S_{\eff, \beta_{0}} (h) \equiv \s_{\eff} (X, Y; \beta_{0}) (h) = \Pi (\S_{\beta_{0}} \mid \Lambda_{\P, \nuis}^{\perp}) (h),
\end{split}
\end{equation}
which has the following explicit form
\begin{equation*}
\S_{\eff, \beta_{0}} (h) = \frac{1}{\sigma^{2}} (Y - g_{0} (\beta_{0}^{\top} X)) g_{0}' (\beta_{0}^{\top} X) (X - \zeta_{\beta_{0}} (\beta_{0}^{\top} X))^{\top} h,
\end{equation*}
where $\zeta_{\beta} (u) \coloneqq \E (X \mid \beta^{\top} X = u)$. By \eqref{score IF}, we can also obtain the efficient influence function $\IF_{\beta_{0}}$ of $\beta_{0}$. We summarize the above reasoning as the following theorem.
\begin{theorem}
\label{thm:SIM}
Suppose that Model \eqref{SIM} is correctly specified, then the semiparametric efficiency bound of $\beta_{0}$ is $V_{\beta_{0}} \coloneqq \{\E [\S_{\eff, \beta_{0}}^{\ast} \otimes \S_{\eff, \beta_{0}}^{\ast}]\}^{-1} \equiv \E [\IF_{\beta_{0}} \otimes \IF_{\beta_{0}}]$.
\end{theorem}
\begin{remark}
\label{rem:kp}
In fact, \cite{kuchibhotla2020efficient} consider the following perturbation of $\beta_{0}$ denoted by $\beta_{t, \mathsf{h}}^{\dag}$ to ensure $\beta_{t, \mathsf{h}}^{\dag} \in \calB$: let $\mathsf{h} \in \bbR^{d - 1}$,
\begin{equation}
\label{kp_path}
\begin{split}
\beta_{t, \mathsf{h}}^{\dag} & = \sqrt{1 - t^{2} \Vert \mathsf{h} \Vert^{2}} \beta_{0} + t \Gamma_{\beta_{0}} \mathsf{h} \\
& \equiv \sqrt{1 - t^{2} \Vert \Gamma_{\beta_{0}} \mathsf{h} \Vert^{2}} \beta_{0} + t \Gamma_{\beta_{0}} \mathsf{h},
\end{split}
\end{equation}
where for any $\beta$, $\Gamma_{\beta} \in \bbR^{d \times (d - 1)}$ is a $d \times (d - 1)$ matrix whose columns form an orthonormal bases of the $(d - 1)$-dimensional hyperplane normal to $\beta$: $\left\{ \eta \in \bbR^{d}: \eta^{\top} \beta = 0 \right\}$. Therefore, the space spanned by $h \equiv \Gamma_{\beta} \mathsf{h}$ is exactly $\T_{\beta} \calB$ as characterized in Lemma~\ref{lem:exp_sphere}. Finally, we note the identity $\sqrt{1 - x^{2}} \equiv \cos (\arcsin (x))$. Therefore, the path \eqref{kp_path} is exactly the exponential map $\Exp_{\beta_{0}} (t \cdot h)$, although \cite{kuchibhotla2020efficient} did not explicitly refer to this interpretation based on differential geometric languages.
\end{remark}
It is worth noting that \cite{kuchibhotla2020efficient} quantifies the difference between an estimator $\hat{\beta}$ and the truth $\beta_{0}$ in their ambient space $\bbR^{d}$ in terms of $\sqrt{n} (\hat{\beta} - \beta_{0})$, but we consider $\sqrt{n} \Exp_{\beta_{0}}^{-1} \hat{\beta}$ as throughout this paper. To summarize, the application of the efficiency theory over parameter manifolds to single-index model parameters further highlights the conceptual advantage of establishing this new framework: the formal geometric language helps abstract away conceptual difficulties, and the only remaining work to do is the computation of exponential map and its inverse in concrete examples. For cases such as the unit hemisphere, a closed form expression exists, whereas for more complicated cases, one could resort to numerical \cite{carone2019toward} or even symbolic techniques \cite{luedtke2026simplifying} for the computation leveraging the geodesic equation \eqref{ODE}. Nonetheless, our framework separates the issues between the representation of the efficiency bound in the language of differential geometry and the actual computation in practice, where the former is more on the conceptual side.
\begin{remark}
\label{rem:Lie groups}
As alluded to in the Introduction, \cite{bigot2012semiparametric} considered a model similar to the single-index model \eqref{SIM}, formulated as a white noise model with the regression coefficient subject to certain invariance constraints. As a consequence, the space of the parameter of interest in \cite{bigot2012semiparametric} is an Abelian Lie group, due to the invariance structure that their applied problem (i.e. image restoration) entails. Similarly to our general framework, \cite{bigot2012semiparametric} established the semiparametric efficiency bound of the parameter of interest by using the linear structure of the Lie algebra and the exponential map associated with the Abelian Lie group. An Abelian Lie group is also a special type of manifold, with the Lie algebra playing the role of the tangent space. However, since different points in the Lie group can be associated with the same Lie algebra, \cite{bigot2012semiparametric} did not need to deal with the subtle issues encountered for a general parameter manifold.
\end{remark}
\section{Discussions}
\label{sec:conclusion}
In this paper, we extend the classical semiparametric efficiency theory for parameters defined on normed linear spaces to parameter manifolds. We generalize central notions in efficiency theory, such as regular estimators and DQM or LAN models, to parameter manifolds. The take-home message of our work is succinctly summarized in Table~\ref{tab:summary}, which essentially tells us how to extend the (semiparametric) efficiency theory currently construed from linear parameter spaces to parameter manifolds by adopting the appropriate geometric vocabulary. The theoretical results obtained in our paper legitimize this vocabulary.

We have demonstrated the applicability of the efficiency theory developed in this paper with two examples in statistics. They help reveal the conceptual advantage of our theoretical framework, echoing the points made in Remark~\ref{rem:intuition}. These examples are just the tip of the iceberg; there are many more interesting examples that could be worked out in the current framework, including the efficiency bound for quantum observables, which belong to the manifold of strictly positive density matrices \cite{gill2013asymptotic, hall2013quantum, liu2026fermi}. We restrict our attention to parameters that lie in a particular type of nonlinear space, namely the Riemannian manifolds. We make such a choice to take advantage of the linear structure associated with the tangent spaces. It will be interesting to extend our framework to parameters in metric spaces. We have also excluded singular cases in our development. For nonlinear spaces, singularities or phenomena such as smeariness \cite{hotz2015intrinsic, eltzner2019smeary} may emerge more naturally. Developing a general efficiency theory that can handle singularities is an interesting and deep problem to pursue, even when the parameter space is linear \cite{evans2020model}. In such cases, the classical parametric rate $n^{-1 / 2}$ is no longer expected, and lower bound techniques may involve more complicated tools such as Fano's or Assouad's lemma \cite{yu1997assouad}, which generally fails to give sharp constants that are the essence of efficiency theory. A possible direction is to explore the direction considered in \cite{takatsu2024generalized}, but a concrete strategy is still beyond reach because they directly resort to Hammersley-Chapman-Robbins bound, which will not match the smeary CLT obtained in \cite{eltzner2019smeary}.

Finally, the generally non-zero curvature of the parameter manifold complicates the mathematical analysis. But once the parameter manifold itself is fixed and the curvature is bounded in the asymptotic framework considered here, the curvature does not affect the lower bound in the first-order asymptotic (at $n^{-1 / 2}$-scale) considered in this paper. But this may not be the case if one considers a non-asymptotic or higher-order asymptotic perspective of efficiency theory \cite{pfanzagl1990estimation, robins2008higher, van2014higher}. To be concrete, as in Definition~\ref{def:par dqm}, the standard DQM assumption on the underlying statistical model essentially considers first-order asymptotics at $n^{-1 / 2}$-scale, under which the curvature effect asymptotically vanishes. If, however, we consider higher-order expansions of the square-root density in the sense of Pfanzagl \cite{pfanzagl1990estimation, van2014higher}, the curvature may eventually appear and the associated higher-order scores may also motivate higher-order refined estimators in a sense similar to the higher-order influence functions developed in \cite{robins2008higher}. Another related open problem is whether one can derive Berry--Esseen-type bounds on the rate at which the distribution of the scaled estimation error approaches its Gaussian limit and how the geometric features, such as curvature, injectivity radius, and cut-locus effects, enter these bounds. To our knowledge, such results are not yet available even in relatively simple settings, such as for the empirical \Frechet{} mean. Our paper provides a first-order theoretical foundation for these potentially important questions that merit more in-depth investigation and are left for future work.

\bibliographystyle{IEEEtran}
\bibliography{References.bib}

\allowdisplaybreaks

\begin{appendices}
\appendix

\subsection{Concepts and Results Related to Analysis on Manifolds}
\label{app:technical}
\subsubsection{Tensor, Vector Fields, and Tensor Fields}
\label{app:tensor}
\leavevmode
In this section, we review tensor and tensor fields in differential geometry, which include most of the concepts introduced in the paper.
\begin{definition}
\label{def:tensor}
Given a vector space $\bbV$, with its dual space denoted by $\bbV^{\ast}$, an $(r, s)$-tensor $\calT$ over $\bbV$ is a multilinear map
\begin{align*}
\calT: \underbrace{\bbV^{\ast} \times \bbV^{\ast} \cdots \times \bbV^{\ast}}_{\textrm{$r$ times}} \times \underbrace{\bbV \times \bbV \cdots \times \bbV}_{\textrm{$s$ times}} \rightarrow \bbR.
\end{align*}
\end{definition}
\begin{definition}
\label{def:tensor fields}
Given a manifold $\bbM$, an $(r, s)$-tensor field $\calT$ is a $C^{\infty} (\bbM)$ multilinear map:
\begin{equation*}
\begin{split}
\calT: &\underbrace{\Gamma (\T^{\ast} \bbM) \times \Gamma (\T^{\ast} \bbM) \times \cdots \times \Gamma (\T^{\ast} \bbM)}_{\textrm{$r$ times}} \\&\times \underbrace{\Gamma (\T \bbM) \times \Gamma (\T \bbM) \times \cdots \times \Gamma (\T \bbM)}_{\textrm{$s$ times}} \rightarrow C^{\infty} (\bbM).
\end{split}
\end{equation*}
\end{definition}
\begin{remark}
\label{rem:tensor examples}
We list a few examples of tensor or tensor fields appeared in this paper. Let $\mu \in \bbM$, where $\bbM$ is taken to be the parameter manifold.
\begin{itemize}
\item The score function $\s_{\cdot}$ is a $(0, 1)$-tensor field (varying with $\mu$) while $\S_{\mu}$ is a $(0, 1)$-tensor when fixing the parameter value to $\mu$. When $\s_{\mu} (\cdot)$ is applied to a tangent vector, the output is a scalar.
\item The Fisher information operator $\bbG$ is a $(0, 2)$-tensor field (varying with $\mu$) while $\bbG_{\mu}$ is a $(0, 2)$-tensor when fixing the parameter value to $\mu$. When $\bbG_{\mu} (\cdot, \cdot) \equiv \E [\S_{\mu} (\cdot) \otimes \S_{\mu} (\cdot)]$ is applied to two tangent vectors, the output is a scalar. Here $\S_{\mu}$ is the random variable version of $\s_{\mu}$.
\end{itemize}
\end{remark}
\subsubsection{Several Useful Lemmas}
\label{app:geometric analysis}
\leavevmode
In this section, we collect useful technical results related to analysis on manifolds. We denote the parameter manifold as $(\bbM, g)$ with Riemannian metric $g$.
\begin{lemma}
\label{lem:log map expansion}
Given $\mu, \mu_{1} \in \bbM$, the Taylor expansion of $\Exp_{\mu}^{-1} (\mu_{1}) $ with respect to the argument $\mu_{1}$ around the point $\mu_{2} \in \bbM$ can be written as 
\begin{equation*}
\begin{split}
\Exp_{\mu}^{-1} (\mu_{1}) &= \Exp_{\mu}^{-1} \mu_{2} + \nabla \Exp_{\mu}^{-1} \mu_{2} \cdot \Pi_{\mu_{2}}^{\mu} \Exp_{\mu_{2}}^{-1} \mu_{1} \\&+ O \left( \left\Vert \Pi_{\mu_{2}}^{\mu} \Exp_{\mu_{2}}^{-1} \mu_{1} \right\Vert^2 \right).
\end{split}
\end{equation*}
\end{lemma}
\begin{lemma}
\label{rem:pop curvature}
As given in Definition~\ref{def:pop curvature} about $\delta$ and $\mu$, $\E [\langle \calR (\delta, \mu) \mu, \delta \rangle]$ depends on the distribution of $\delta$ only via its covariance operator $\mathbb{C}$.
\end{lemma}
\begin{proof}
Using normal coordinates, $\E [\langle \calR (\delta, \mu) \mu, \delta \rangle]$ can be written as $-3\sum_{ij,kl} a_{ij,kl} \mathbb{C}_{kl} \mu^i \mu^j$, where $\mathbb{C}_{kl}$ is the $kl$-th element of $\mathbb{C} \equiv \Cov (\delta)$, such that $\E [\langle \calR (\delta, \mu) \mu, \delta \rangle]$ can be shown as the quadratic form $\langle \calR (\mathbb{C}) \mu, \mu \rangle$, where $\calR (\mathbb{C})$ is symmetric and depends linearly on $\mathbb{C}$.
\end{proof}
\begin{lemma}
\label{lem:parallel expansion}
Pick any $\mu^{\dag} \in \bbM$ and, for $\mu, \mu' \in \bbM \setminus \calC (\mu^{\dag})$ sufficiently close. Denote by $\gamma$ the unit speed geodesic segment such that $\gamma (0) = \mu$ and $\gamma (g (\mu, \mu')) = \mu'$, where $g (\cdot, \cdot)$ denotes the distance function on $\bbM\times\bbM$ induced by the Riemannian metric. If $\gamma (t) \in \bbM \setminus \calC (\mu^{\dag})$ for all $t \in (0, g (\mu, \mu'))$, we have
\begin{align*}
\Pi_{\mu'}^{\mu} \Exp_{\mu'}^{-1} \mu^{\dag} = \Exp_{\mu}^{-1} \mu^{\dag} + \nabla_{\Exp^{-1}_{\mu} \mu'} \Exp_{\mu}^{-1} \mu^{\dag} + o \left( g (\mu, \mu') \right).
\end{align*}
\end{lemma}
\begin{proof}
To simplify the notation, let $f: \bbM \setminus \calC (\mu^{\dag}) \rightarrow \T_{\mu} \bbM$ be a function defined as $f (\mu') \coloneqq \Pi_{\mu'}^{\mu} \Exp_{\mu'}^{-1} \mu^{\dag}$, given any $\mu' \in \bbM \setminus \calC (\mu^{\dag})$. Obviously, $f (\mu) \equiv \Exp_{\mu}^{-1} \mu^{\dag}$.
Based on the conditions of this Lemma, $\Exp_{\gamma (t)}^{-1} \mu^{\dag}$ is a smooth vector field along $\gamma$, and the two tangent vectors $f (\mu')$ and $f (\mu)$ are both in $\T_{\mu} \bbM$. Thus we have
\begin{align*}
f (\mu') = f (\mu) + \nabla_{\Exp_{\mu}^{-1} \mu'} f (\mu) + o \left( d (\mu, \mu') \right),
\end{align*}
by the definition of the covariant derivative.
\end{proof}
\begin{lemma}
\label{lem:covariant derivative}
We denote $\delta \coloneqq \Exp_{\psi (\theta)}^{-1} \hat{\psi}_{n}, \eps_{h} \coloneqq \sqrt{n} \Exp_{\psi (\theta)}^{-1} \psi (\theta') \in \T_{\psi (\theta)} \Psi$. Consider $\nabla_{\eps_{h}} \delta$, we have $- \nabla_{\eps_{h}} \delta = \eps_{h} + Q (\delta) \eps_{h} + O (\Vert \delta \Vert^{3}), \text{ and } \eps_{h} = \dot{\psi} (\theta) h + o (\Vert h \Vert)$, where $Q (\delta)$ is a quadratic form, the exact formula of which is given in the proof below.
\end{lemma}
\begin{proof}
To derive the result, we define the tangent vectors $\delta = \Exp_{\theta}^{-1} \hat{\theta}_{n} \in \T_{\theta} \bbM$ and $\delta_{t} \in \T_{\Exp_{\theta} t h} \bbM$ by the equations 
\begin{equation}
\label{boundary}
\hat{\theta}_n = \Exp_{\theta} \delta = \Exp_{\Exp_{\theta}th} \delta_{t}.
\end{equation}
We denote $\{e_{i}, i = 1, \cdots, p\}$ as an orthonormal basis of $\T_{\theta} \bbM$. Expressing all of the quantities in terms of the normal coordinates $\theta^i=\Exp_{\theta}(\sum_{i}\theta^i E_i)$, the geodesic from $\theta$ to $\hat{\theta}_n$ is simply the curve $\hat{\theta}_n=\theta^i(t)=t\delta^i$, where $\delta=\sum_{i}\delta^i E_i$. The geodesic curve from $\Exp_{\theta}th$ to $\hat{\theta}_n$, denoted by $\gamma(s)$, satisfies $\nabla_{\gamma'(s)}\gamma'(s)=0$ under the boundary condition \eqref{boundary}. By original local analysis of curvature, the Riemannian metric $d$ of $(\bbM, d)$ is expressed as the Taylor series
\begin{equation}
\label{quadratic form}
    d_{ij}=\delta_{ij}+\frac{1}{2}\sum_{kl}\frac{\partial^2 d_{ij}}{\partial\theta^k \partial\theta^l}\theta^k \theta^l+ R^3_{ij}(\theta)
\end{equation}
where $\delta_{ij}$ is the Kronecker delta function, $R^3_{ij}(\theta)=O(\Vert \theta\Vert^3)$ denotes the remainder term after Taylor expansion to the second-order under the local coordinate system $(\theta^i)$ on $\bbM$. By the definition of covariant derivative, based on the form of Christoffel symbols, $\Gamma_{ij,k}=-2 \sum_{l} a_{ij,kl}\theta^l$ where $a_{ij,kl}=\frac{1}{2}\frac{\partial^2 d_{ij}}{\partial\theta^k \partial\theta^l}$ and $a_{ij,kl}=a_{ji,kl}=a_{ij,lk}=a_{kl,ij}$, and $a_{ij,kl}+a_{ik,jl}+a_{il,jk}=0$. Solving the geodesic equation above for $\delta_{t}$ yields
\begin{equation}
\delta_{t} = \delta - t h - t Q(\delta) h
\end{equation}
where $Q(\delta)$ is the quadratic form in \eqref{quadratic form}, $\sum_{kl} a_{ij,kl} \delta^k \delta^l$. The parallel transport of $\delta_{t}$ along the geodesic curve $\Exp_{\theta}th$ equals $\delta_{t} (t) = \delta_{t} + O (t^2)$ by Taylor expansion of $\delta_t(t)$ with respect to $\delta_{t}$ and the fact that covariant derivative of $\delta_t(t)$ along the curve vanishes.
\end{proof}
Lemma~\ref{lem:covariant derivative} will be used in the proof of Lemma~\ref{lem:Un} in Appendix~\ref{app:thm par} later.
\begin{lemma}
\label{lem:linearity of curvature}
For a $p$-dimensional manifold $(\bbM, d)$, the operator $\calR (\mathbb{C})$ defined in Definition~\ref{def:pop curvature} is linear in $\mathbb{C}$.
\end{lemma}
\begin{proof}
The statement of the lemma can be shown by writing $\calR (\mathbb{C})$ explicitly in coordinate form: for $i, j \in \{1, \cdots, p\}$,
\begin{align*}
\calR (\mathbb{C})_{i, j} = - 3\sum_{k,l} a_{i, j}^{k, l} \mathbb{C}_{k, l}
\end{align*}
where $a_{i, j}^{k, l}$ is given in Lemma~\ref{lem:covariant derivative}. Then for any two $\mathbb{C}_{1}$ and $\mathbb{C}_{2}$ and any two real numbers $c_{1}$ and $c_{2}$, we have
\begin{equation*}
\begin{split}
\calR (c_{1} \mathbb{C}_{1} + c_{2} \mathbb{C}_{2})_{i, j} & = - 3 a_{i, j}^{k, l} \left( c_{1} \mathbb{C}_{1} + c_{2} \mathbb{C}_{2} \right)_{k, l} \\&= - c_{1} 3 a_{i, j}^{k, l} (\mathbb{C}_{1})_{k, l} - c_{2} 3 a_{i, j}^{k, l} (\mathbb{C}_{2})_{k, l} \\
& = c_{1} \calR (\mathbb{C}_{1})_{i, j} + c_{2} \calR (\mathbb{C}_{2})_{i, j}.
\end{split}
\end{equation*}
Thus $\calR (c_{1} \mathbb{C}_{1} + c_{2} \mathbb{C}_{2}) = c_{1} \calR (\mathbb{C}_{1}) + c_{2} \calR (\mathbb{C}_{2})$.
\end{proof}
\begin{lemma}
\label{lem:connection}
Let $(\bbM, g)$ be a manifold. Given $\mu \in \bbM$, $\mu_{h} \coloneqq \Exp_{\mu} h$ for $h \in \T_{\mu} \bbM$, we have for $h \rightarrow 0$,
\begin{equation}
\label{deriv exp expansion}
\nabla_{h} \Exp_{\mu} h (\cdot) = \id (\cdot) - \frac{1}{6} \calR_{\mu} (\cdot, h) h + O (\Vert h \Vert^{3}), \textrm{ and}
\end{equation}
\begin{equation}
\label{deriv log expansion}
\nabla_{h} \Exp^{-1}_{\mu} \mu_{h} (\cdot) = \id (\cdot) + \frac{1}{6} \calR_{\mu} (\cdot, h) h + O (\Vert h \Vert^{3}).
\end{equation}
\end{lemma}
\begin{proof}
Let $\gamma_{\delta} (t)$ denote family of geodesics ended at $\gamma_{\delta} (1) = \mu_{h}$ and initialized at $\gamma_{\delta} (0) = \mu$ with initial velocities: $\dot{\gamma}_{\delta} (0) = h + \delta \cdot v$, for some $v \in \T_{\mu} \bbM$. The first variation of $\gamma_{\delta} (t)$ with respect to $\delta$ defines a Jacobi field as follows
\begin{align*}
J (t) \coloneqq \left.\frac{\partial}{\partial \delta} \gamma_{\delta} (t) \right|_{\delta = 0}.
\end{align*}
It is well known that $J$ solves the following Jacobi equation with boundary conditions (cf. Theorem 10.1 of \cite{lee2018introduction}):
\begin{equation}
\label{Jacobi ODE}
\begin{split}
\ddot{J} + \calR (\dot{\gamma}, J) \dot{\gamma} = 0, \textrm{ s.t. } J (0) = 0 \textrm{ and } \dot{J} (0) = v.
\end{split}
\end{equation}
Here the boundary conditions are enforced based on how the geodesics are chosen and $\dot{J}$ and $\ddot{J}$ respectively denote the first and second derivatives of $J$ w.r.t. argument $t$.
By directly matching the ODE \eqref{Jacobi ODE} together with boundary conditions, $J (t)$ can be approximated as follows when $t$ is bounded and $h$ is small.
\begin{align*}
J (t) = t \cdot v - \frac{t^{3}}{6} \calR_{\mu} (v, h) h + O (t^{4} h^{3}),
\end{align*}
where the factor by $h^{3}$ arises from taking derivative of the ODE \eqref{Jacobi ODE} w.r.t. $t$, resulting in the covariant derivative of the curvature tensor $\nabla \calR_{\mu} (h, v, h, h)$.
Setting $t = 1$, we obtain
\begin{align*}
J (1) = v - \frac{1}{6} \calR_{\mu} (v, h) h + O (\Vert h \Vert^{3}).
\end{align*}
By definition, the Jacobi field describes the endpoint variation of the perturbed geodesics. Since $\nabla \Exp_{\mu} h \cdot v = \frac{\partial}{\partial \delta} \Exp_{\mu}(h + \delta \cdot v) \left.\right|_{\delta = 0}$ and $J(t) = \left. \frac{\partial}{\partial \delta} \gamma_{\delta} (t) \right|_{\delta = 0}$, we have $J (1) \equiv \nabla \Exp_{\mu} h \cdot v$ so the first conclusion \eqref{deriv exp expansion} holds. The second conclusion \eqref{deriv log expansion} is an immediate consequence of \eqref{deriv exp expansion}.
\end{proof}
\subsection{Technical Details for Results in Section~\ref{sec:intro}}
\label{app:intro}
\subsubsection{Regularity of the Sample \Frechet{} Mean and Irregularity of Hodges' Superefficient Estimator for the \Frechet{} Mean}
\label{app:frechet mean}
\leavevmode
\begin{proof}[Proof of Example~\ref{ex:riemannian gaussian}]\leavevmode
Let 
\begin{align*}
\mu \coloneqq \arg \min_{m \in \bbM} \E \left[ \dfrac{1}{2} d^2 (X, m) \right]    
\end{align*} 
denote the unique population \Frechet{} mean under the target probability distribution $\P$ defined over the manifold $(\bbM, d)$. To emphasize the connection between $\P$ and $\mu$, we denote $\P$ as $\P_{\mu}$ (similar for the corresponding expectation operator) when it does not lead to ambiguity.
Given $n$ i.i.d. samples $X_1, \cdots, X_n \sim \P$, by the CLT for the sample \Frechet{} mean estimator $\hat{\mu}_{n}$ \cite{bhattacharya2005large, kendall2011limit, hotz2024central}, we have 
\begin{align*}
\sqrt{n} \Exp_{\mu}^{-1} (\hat{\mu}_n) \rightsquigarrow_{\P_{\mu}} \calN \left( 0, \bbV_{\mu} \right),
\end{align*}
where $\bbV_{\mu} \coloneqq \E^{-1} [\nabla \Exp_{\mu}^{-1} (X)] \Sigma_{\mu} \E^{-1} [\nabla \Exp_{\mu}^{-1} (X)]$, and $\Sigma_{\mu} \coloneqq \E [\Exp_{\mu}^{-1}(X) \otimes \Exp_{\mu}^{-1} (X)]$.
Let $\mu_{n, h} \coloneqq \Exp_{\mu} \left( \frac{h}{\sqrt{n}} \right)$ for every $h \in \T_{\mu} \bbM$, and denote the corresponding probability distribution as $\P_{\mu_{n, h}}$. Our goal is then to show that
\begin{align*}
\Pi_{\mu_{n, h}}^{\mu} \sqrt{n} \Exp_{\mu_{n, h}}^{-1} (\hat{\mu}_n) \rightsquigarrow_{\P_{\mu_{n, h}}} \calN \left( 0, \bbV_{\mu} \right).
\end{align*}
Equivalently, we are left to prove the following assertion:
\begin{equation}
\label{variance limit}
\lim_{n \rightarrow \infty} \Pi_{\mu_{n, h}}^{\mu} \bbV_{\mu_{n, h}} = \bbV_{\mu}.
\end{equation}
Here $\bbV_{\mu_{n, h}}$ is defined as $\bbV_{\mu}$ with $\mu$ replaced by $\mu_{n, h}$ verbatim. Using the property of the parallel transport map, we can distribute $\Pi_{\mu_{n, h}}^{\mu}$ to each of the terms in $\bbV_{\mu_{n, h}}$:
\begin{align*}
& \Pi_{\mu_{n, h}}^{\mu} \bbV_{\mu_{n, h}} = \Pi_{\mu_{n, h}}^{\mu} \E_{\mu_{n, h}}^{-1} \nabla \Exp_{\mu_{n, h}}^{-1} (X) \Pi_{\mu_{n, h}}^{\mu} \Sigma_{\mu_{n, h}} \\
& \quad \times \Pi_{\mu_{n, h}}^{\mu} \E_{\mu_{n, h}}^{-1} [\nabla \Exp_{\mu_{n, h}}^{-1} (X)].
\end{align*}
Further, by Lemma~1 of \cite{smith2005covariance}, we have, for any $\mu \in \bbM$,
\begin{align*}
- \E [\nabla \Exp_{\mu}^{-1} (X)] = \bbI - \frac{1}{3} \calR \left( \Sigma_{\mu} \right).
\end{align*}
Then combining the above two observations, we have
\begin{align*}
& \Pi_{\mu_{n, h}}^{\mu} \bbV_{\mu_{n, h}} = \Pi_{\mu_{n, h}}^{\mu} \Big\{ \bbI - \frac{1}{3} \calR (\Sigma_{\mu_{n, h}}) \Big\}^{-1} \Pi_{\mu_{n, h}}^{\mu} \\
& \quad \times \Sigma_{\mu_{n, h}} \Pi_{\mu_{n, h}}^{\mu} \Big\{ \bbI - \frac{1}{3} \calR (\Sigma_{\mu_{n, h}}) \Big\}^{-1}.
\end{align*}
We first analyze the limit of the middle term in the above display, our object is to prove the following equation: 
\begin{equation*}
\begin{split}
&\lim_{n \rightarrow \infty} \Pi_{\mu_{n, h}}^{\mu} \E [\Exp_{\mu_{n, h}}^{-1}(X') \otimes \Exp_{\mu_{n, h}}^{-1} (X')] \\&= \E [\Exp_{\mu}^{-1} (X) \otimes \Exp_{\mu}^{-1} (X)],
\end{split}
\end{equation*}
where $X$ is the random variable following $\P_{\mu}$ and $X'$ is the random variable following $\P_{\mu_{n, h}}$. Equivalently, 
\begin{align*}
\lim_{n \rightarrow \infty} \Pi_{\mu_{n, h}}^{\mu} \Sigma_{\mu_{n, h}} = \Sigma_{\mu}.
\end{align*}
By Taylor expansion and parallel transport, we have
\begin{equation*}
\begin{split}
\Pi_{\mu_{n, h}}^{\mu} (\Exp_{\mu_{n, h}}^{-1}(X')) &= \Exp_{\mu}^{-1} (X') \\&+ \nabla_{\Exp_{\mu}^{-1} (\mu_{n, h})} \Exp_{\mu}^{-1} (X') + \calR (\mu, \mu_{n, h}),
\end{split}
\end{equation*}
where $\Vert \calR (\mu, \mu_{n, h}) \Vert = o (g (\mu, \mu_{n, h}))$ and $\nabla_{\Exp_{\mu}^{-1} (\mu_{n, h})} \Exp_{\mu}^{-1} (X') = O (n^{- 1 / 2})$. Because of the properties of parallel transport and smoothness of Riemannian manifold,
\begin{equation*}
\begin{split}
&\lim_{n \rightarrow \infty} \E [\Pi_{\mu_{n, h}}^{\mu} \Exp_{\mu_{n, h}}^{-1} (X') \otimes \Pi_{\mu_{n, h}}^{\mu} \Exp_{\mu_{n, h}}^{-1} (X')] \\&= \lim_{n \rightarrow \infty} \Pi_{\mu_{n, h}}^{\mu} \E [\Exp_{\mu_{n, h}}^{-1} (X') \otimes \Exp_{\mu_{n, h}}^{-1} (X')].
\end{split}
\end{equation*}
Based on the above results, we obtain
\begin{equation*}
\begin{split}
&\lim_{n \rightarrow \infty} \E [\Exp_{\mu}^{-1} (X') \otimes \Exp_{\mu}^{-1} (X')] \\&= \lim_{n \rightarrow \infty} \Pi_{\mu_{n, h}}^{\mu} \E [\Exp_{\mu_{n, h}}^{-1} (X') \otimes \Exp_{\mu_{n, h}}^{-1} (X')].
\end{split}
\end{equation*}
By the assumption that the probability measure $\P_{\mu}$, $\mu \in \bbM$, has a continuous probability density function $\p_{\mu}$, we have that
\begin{equation*}
\begin{split}
&\lim_{n \rightarrow \infty} \E [\Exp_{\mu}^{-1} (X') \otimes \Exp_{\mu}^{-1} (X')] \\&= \E [\Exp_{\mu}^{-1} (X) \otimes \Exp_{\mu}^{-1} (X)].
\end{split}
\end{equation*}
This is because $\bbM$ is smooth and locally compact, the tangent vector $\Exp_{\mu}^{-1} (X')$ corresponding to $\p_{\mu_{n, h}}$ is also bounded, and based on the generalization of Bounded convergence theorem or Dominated convergence theorem, we can exchange the limit and the integral.
Since $\lim_{n \rightarrow \infty} \Pi_{\mu_{n, h}}^{\mu} \Sigma_{\mu_{n, h}} = \Sigma_{\mu}$, and by normal coordinates, the matrix $\calR (\Sigma_{\mu})$ depends linearly on the covariance matrix $\Sigma_{\mu}$, which is shown in Lemma 1 of \cite{smith2005covariance}:
\begin{equation*}
\begin{split}
\lim_{n \rightarrow \infty} \Pi_{\mu_{n, h}}^{\mu} \Big\{ \bbI - \frac{1}{3} \calR (\Sigma_{\mu_{n, h}}) \Big\}^{-1}& = \Big\{ \bbI - \frac{1}{3} \calR (\Sigma_{\mu}) \Big\}^{-1} \\&\equiv \E [\nabla \Exp_{\mu}^{-1} (X)].
\end{split}
\end{equation*}
Therefore, the sample \Frechet{} mean is a regular estimator of population \Frechet{} mean.
\end{proof}
\begin{proof}[Proof of Example~\ref{hodges}]
Based on the construction of the estimator, denoted by $\hat{\mu}_{n, \Hodges}$, the target parameter is the population Frech\'{e}t mean $\mu$. For any given tangent vector $h \in \T_{\mu} \bbM$, we let $\mu' \coloneqq \Exp_{\mu} (\frac{h}{\sqrt{n}})$. As in the above proof, we obtain 
\begin{equation}
\sqrt{n} \Exp_{\mu}^{-1} \hat{\mu}_{n} \rightsquigarrow_{\P} \calN (0, \bbV_{\mu}).
\end{equation}
Consequently, we have
\begin{equation}
\begin{split}
& \sqrt{n} \Exp_{\mu'}^{-1} \hat{\mu}_{n, \Hodges} \\
& = \sqrt{n} \Exp_{\mu'}^{-1} \hat{\mu}_{n} \cdot \mathbb{1} \{d (\hat{\mu}_{n}, \mu) > n^{-\frac{1}{4}}\} \\
& \quad + \sqrt{n} \Exp_{\mu'}^{-1} \mu \cdot \mathbb{1} \{d (\hat{\mu}_{n}, \mu) \leq n^{-\frac{1}{4}}\}.
\end{split}
\end{equation}
By Slutsky's theorem and the fact that the term $\sqrt{n} \Exp_{\mu'}^{-1} \mu \mathbb{1} \{d (\hat{\mu}_{n}, \mu) \leq n^{-\frac{1}{4}}\}$ depends on the choice of $h \in \T_{\mu} \bbM$ as $n \rightarrow \infty$, thus the limiting distribution of $\sqrt{n} \Exp_{\mu'}^{-1} \hat{\mu}_{n, \Hodges}$ depends on the perturbation with respect to the target parameter, it is hence not regular according to Definition~\ref{def:riemann regular}.
\end{proof}
\subsection{Technical Details for Results in Section~\ref{sec:par}}
\label{app:par}
\subsubsection{Proof of Proposition~\ref{prop:lan}}
\label{app:prop lan}
\leavevmode
\begin{proof}
    We first use the abbreviations $\p_n$, $\p$ and $g$ for $\p (\cdot; \theta_{n, h})$, $\p(\cdot;\theta)$, and $\s (\cdot; \theta)(h)$, respectively. By \eqref{dqm} the sequence $\sqrt{n}(\sqrt{\p_n}-\sqrt{\p})$ converges in quadratic mean to $\frac{1}{2}g\sqrt{\p}$. This implies that the sequence $\sqrt{\p_n}$ converges in quadratic mean to $\sqrt{\p}$. By the continuity of the inner product, $\E g=\int \frac{1}{2} g \sqrt{\p} 2 \sqrt{\p} \diff x=\lim \int \sqrt{n}\left(\sqrt{\p_n}-\sqrt{\p}\right)\left(\sqrt{\p_n}+\sqrt{\p}\right) \diff x .$ Thus $\E g=0$.
    The random variable $W_{ni}=2\left[\sqrt{\p_n / \p}\left(X_i\right)-1\right]$ is with $\P$-probability $1$ well defined. By \eqref{dqm}, 
\begin{equation}
\label{varandmean}
\begin{split}
& \ \Var \Big( \sum_{i=1}^n W_{ni} - \frac{1}{\sqrt{n}} \sum_{i = 1}^{n} g (X_i) \Big) \\
\leq & \ \E(\sqrt{n} W_{ni}-g(X_i))^2 \to 0, \\
& \ \E \sum_{i=1}^n W_{ni}=2n \left(\int \sqrt{\p_n}\sqrt{\p}\diff x-1\right)\\
= & - n\int (\sqrt{\p_n}-\sqrt{\p})^2\diff x\rightarrow -\frac{1}{4}\E g^2.
\end{split}
\end{equation}
Here $\E g^2=\bbG_{\theta} (h, h)$ by the definitions of $g$ and $\bbG_{\theta}$. If both the means and variances of a sequence of random variables converge to zero, then the sequence converges to zero in probability. Therefore, combining the preceding pair of displayed equations, we find 
\begin{equation}
\label{find}
\sum_{i=1}^n W_{ni}=\frac{1}{\sqrt{n}}\sum_{i=1}^ng(X_i)-\frac{1}{4}\E g^2+o_{\P} (1).
\end{equation}
Next, we express the log-likelihood ratio in $\sum_{i=1}^n W_{ni}$ through a Taylor expansion of the logarithm. If we write $\log (1 + x) = x - x^{2} / 2 + x^{2} r (2x)$, then $r (x) \rightarrow 0$ as $x \rightarrow 0$, and 
\begin{equation}
\label{expansion}
\begin{split}
& \log \prod_{i=1}^n \frac{\p_n}{\p}\left(X_i\right) =2 \sum_{i=1}^n \log \left(1+\frac{1}{2} W_{n i}\right) \\
& =\sum_{i=1}^n W_{n i}-\frac{1}{4} \sum_{i=1}^n W_{n i}^2+\frac{1}{2} \sum_{i=1}^n W_{n i}^2 r \left(W_{n i}\right).
\end{split}
\end{equation}
As a consequence of the right side of \eqref{varandmean}, it is possible to write $n W_{n i}^2=g^2(X_i)+A_{ni}$ for random variable $A_{ni}$ such that $\E\vert A_{ni}\vert>0$. The averages $\bar{A}_n$ converge in mean and therefore in probability to zero. Combination with the law of large numbers yields $\sum_{i=1}^n W_{n i}^2 = n^{-1} \sum_{i = 1}^{n} g^{2} (X_{i}) +\bar{A}_n \xrightarrow{\P} \E g^2 (X)$. By the triangle inequality followed by Markov's inequality, 
\begin{align*}
& n \P (|W_{n i}| > \varepsilon \sqrt{2}) \leq n \P (g^2 (X_i) > n \varepsilon^2) + n \P (|A_{n i}| > n \varepsilon^2) \\ 
& \leq \varepsilon^{-2} \E [g^2 (X) \mathbb{1} \{g^{2} (X) > n \varepsilon^{2}\}] + \varepsilon^{-2} \E |A_{n i}| \to 0. 
\end{align*}
The left hand side of the above display is an upper bound for $\P (\max _{1 \leq i \leq n} |W_{n i}|>\varepsilon \sqrt{2})$. Thus the sequence $\max _{1 \leq i \leq n}\left|W_{n i}\right|$ converges to zero in probability. By the property of the function $r$, the sequence $\max _{1 \leq i \leq n}\left|r(W_{n i})\right|$ converges in probability to zero as well. The last term on the right in \eqref{expansion} is bounded by $\max _{1 \leq i \leq n}\left|r(W_{n i})\right|\sum_{i=1}^n W_{n i}^2$. Thus it is $o_\P(1) O_\P(1) = o_\P (1)$. Combining them, we obtain that $\log \prod_{i=1}^n \frac{\p_n}{\p}\left(X_i\right)=\sum_{i=1}^n W_{n i}-\frac{1}{4} \E g^2+o_\P(1)$. Together with \eqref{find} this yields the proposition.
\end{proof}
\subsubsection{Proof of Theorem~\ref{thm:par}}
\label{app:thm par}
\leavevmode
In the course of the proof, we use the following technical results. 
\begin{lemma}[Prokhorov's theorem]
\label{thm:prokhorov}
Suppose that $(\bbM, g)$ is a metric space equipped with its Borel $\sigma$-field $\calM$, and let $\calP$ be a collection of probability measures on $(\bbM, \calM)$. For any $\P \in \calP$, if $\P$ is tight, it is relatively compact; in addition, suppose that $(\bbM, g)$ is separable and complete, if $\P$ is relatively compact, it is tight.
\end{lemma}
\begin{lemma}[Le Cam's third lemma]
\label{3rd lemma}
Let $\P_{n}$ and $\Q_{n}$ be sequences of probability measures on measurable spaces $(\Omega_{n}, \calA_{n})$, and let $X_{n}: \Omega_{n} \rightarrow \bbR^{k}$ be a sequence of random vectors of dimension $k$ with $k$ finite. Suppose that $\Q_{n} \triangleleft \P_{n}$ and 
\begin{align*}
\Big( X_n, \frac{\diff \Q_n}{\diff \P_n} \Big) \rightsquigarrow_{\P_n} (X, V).
\end{align*}
Then for any $B \in \calA_{n}$, $\L (B) \coloneqq \E \mathbb{1}_B (X) V$ defines a probability measure on $(\Omega_{n}, \calA_{n})$, and 
\begin{align*}
X_n \rightsquigarrow_{\Q_n} X^{\dag} \text{ where $X^{\dag} \sim \L$}.
\end{align*}
\end{lemma}
The following lemma describes the relationship between marginal tightness and joint tightness for a sequence of probability measures defined on a metric space.
\begin{lemma}
\label{lem:joint vs. marginal tightness}
Probability measures on a product space $\bbS_1 \times \bbS_2$ are tight if and only if the two sets of marginal distributions are tight in $\bbS_1$ and $\bbS_2$, where $\bbS_1$ and $\bbS_2$ are general metric spaces.
\end{lemma} 
\begin{lemma}
\label{lem:asymp tight}
Suppose that $(\bbM, g)$ is a separable metric space and $\P$ is a tight probability measure defined on $(\bbM, \calM)$. Let $\P_{n}$ be a sequence of probability measures defined on $(\bbM, \calM)$. If $\P_n \rightsquigarrow \P$, then $\P_n$ is asymptotically tight.
\end{lemma}
\begin{lemma}[Theorem 5.4 of \cite{billingsley2013convergence}]
\label{lem:billingsley}
Let $X_n \sim \P_{n}, n \in \mathbb{N}$ and $X \sim \P$ be, respectively, a sequence of random variables and a random variable defined on a measurable metric space $(\bbM, \calM, g)$, where $\calM$ is the Borel $\sigma$-field and $g$ is the metric. Suppose that $X_n \rightsquigarrow_{\P_{n}} X$. If $X_n$ is uniformly integrable for every $n \in \mathbb{N}$, then $\E X_n \rightarrow \E X$ as $n \rightarrow \infty$; if $X$ and $X_n$ are non-negative and integrable for every $n \in \mathbb{N}$, then $\E X_n \rightarrow \E X$ for every $n \in \mathbb{N}$ implies that $X_n$ is uniformly integrable for every $n \in \mathbb{N}$.
\end{lemma}
We are now ready to prove Theorem~\ref{thm:par}.
\begin{proof}[Proof of Theorem~\ref{thm:par}]
Fix $\theta \in \Theta$ and we automatically also have $\psi (\theta) \in \Psi$. Let $\S_{n} (\theta) \coloneqq \frac{1}{\sqrt{n}} \sum_{i = 1}^{n} \s (X_{i}; \theta)$ be the empirical score function. By CLT, we have
\begin{align*}
\S_{n} (\theta) \rightsquigarrow_{\P} \calN \left( 0, \bbG_{\theta} \right).
\end{align*}
Let $V$ be a random variable such that $V \sim \calN \left( 0, \bbG_{\theta} \right)$.
Let
\begin{align*}
(U_n, V_n) \coloneqq (\sqrt{n} \Exp_{\psi (\theta)}^{-1} \hat{\psi}_{n}, \S_n (\theta)).
\end{align*}
$\hat{\psi}_n$ being a regular estimator sequence and the underlying statistical model $\calP \coloneqq \{\P_{\theta}: \theta \in \Theta\}$ being DQM at $\theta$, taken together, imply that $U_{n}$ and $V_{n}$, respectively, converge marginally to some random variables $U$ and $V$, with $U$ to be defined later. Hence, by Lemma~\ref{lem:joint vs. marginal tightness} and~\ref{lem:asymp tight}, $(U_n, V_n)$ is jointly asymptotically tight.
Fix a subsequence $\{n'\}$ of the original sequence indexed by $n$. By Lemma~\ref{thm:prokhorov},~\ref{lem:joint vs. marginal tightness} and~\ref{lem:asymp tight}, there exists a further subsequence $\{n''\}$ of $\{n'\}$ such that the distribution of $(U_{n''}, V_{n''})$ converges to the joint distribution of $(U, V)$. For convenience, as commonly done in the literature \cite{van2023weak}, we abuse our notation and denote the subsequence $\{n''\}$ by $\{n\}$. Now let
\begin{align*}
W_n \coloneqq \sum_{i = 1}^{n} \log \frac{\p (X_{i}; \theta_{h / \sqrt{n}})}{\p (X_{i}; \theta)}.
\end{align*}
Then the joint distribution of $(U_n, W_n)$ has the limiting distribution $(U, V (h) - \frac{1}{2} \langle \bbG_{\theta} h, h \rangle)$, which follows from $\calP$ being LAN, a consequence of Proposition~\ref{prop:lan}. Next, since $\hat{\psi}_{n}$ is a regular estimator sequence, by Definition~\ref{def:par dqm}, for all $h \in \T_\theta \Theta$, 
\begin{equation}
\label{proof regular}
\Pi_{\psi (\theta_{h / \sqrt{n}})}^{\psi (\theta)} \sqrt{n} \Exp_{\psi (\theta_{h / \sqrt{n}})}^{-1} \hat{\psi}_{n} \rightsquigarrow_{\P_{\theta_{h / \sqrt{n}}}} Z_{\psi (\theta)}.
\end{equation}
We now let $U \coloneqq Z_{\psi (\theta)}$. To avoid notation clutter, we denote $\theta' \coloneqq \theta_{h / \sqrt{n}}$ for short.
To proceed the proof, we take the following key intermediate lemma as given, the proof of which is delayed to the end of this section.
\begin{lemma}
\label{lem:Un}
Under the assumptions of Theorem~\ref{thm:par}, the following holds:
\begin{equation}
\label{Un}
U_n \rightsquigarrow_{\P_{\theta'}} U + \dot{\psi} (\theta) h.
\end{equation}
\end{lemma}
The remaining steps of the proof follow the classical characteristic function approach closely as in \cite{van2023weak}. Recall that $\bbV_{\psi, \theta} = \dot{\psi} (\theta) \cdot \bbG_{\theta}^{-1} \cdot \dot{\psi} (\theta)^{\ast}$. Based on \eqref{Un}, we have, for any $a \in \T_{\psi (\theta)}^{\ast} \Psi$ from the cotangent space of $\Psi$ at $\psi (\theta)$,
\begin{equation}
\label{fourier}
\E_{\theta'} \left[ \exp \left\{ \rmi \langle U_n, a \rangle \right\} \right] \rightarrow \exp \left\{ \rmi \langle \dot{\psi} (\theta) h, a \rangle \right\} \E_{\theta} \left[ \exp \left\{ \rmi \langle U, a \rangle \right\} \right].
\end{equation}
On the other hand, by contiguity (by Le Cam's first lemma), 
\begin{equation}
\label{contiguity}
\E_{\theta'} \left[ \exp \left\{ \rmi \langle U_n, a \rangle \right\} \right] = \E_{\theta} \left[ \exp \left\{ \rmi \langle U_n, a \rangle + W_{n} \right\} \right] + o (1).
\end{equation}
We further observe that
\begin{equation*}
\bigl\vert \exp \{\rmi \langle U_n, a \rangle + W_n\} \bigr\vert
\equiv \exp \{W_{n}\}, n \in \mathbb{N},
\end{equation*}
where the identity simply follows from Euler's formula, is a uniformly integrable sequence of random variables by LAN and Lemma~\ref{lem:billingsley}. Combining \eqref{fourier} and \eqref{contiguity}, we conclude that, again by Lemma~\ref{lem:billingsley}, for all $h \in \T_{\theta} \Theta$ and $a \in \T_{\psi (\theta)}^{\ast} \Psi$,
\begin{equation}
\label{limit fourier}
\begin{split}
& \ \E_{\theta} \left[ \exp \left\{ \rmi \langle U, a \rangle + \langle V, h \rangle - \frac{1}{2} \langle \bbG_{\theta} h, h \rangle \right\} \right] \\
= & \ \exp \{\rmi \langle \dot{\psi} (\theta) h, a \rangle\} \cdot \E_{\theta} \left[ \exp \{\rmi \langle U, a \rangle\} \right].
\end{split}
\end{equation}
Both sides of \eqref{limit fourier}, for fixed $a$, are functions of $h$, where $h \in \T_\theta \Theta$. Take another $b \in \T_{\psi (\theta)}^{\ast} \Psi$ and set $h = \rmi \bbG_{\theta}^{-1} \cdot \dot{\psi} (\theta)^{\ast} \cdot (b - a)$. Then 
\begin{align*}
\text{LHS of } \eqref{limit fourier} & = \E_{\theta} \left[ \exp \left\{ \rmi \langle U, a \rangle + \rmi \langle \dot{\psi} (\theta) \cdot \bbG_{\theta}^{-1} \cdot V, b - a \rangle \right\} \right] \\
& \times \exp \left\{ \frac{1}{2} \langle \bbV_{\psi, \theta} \cdot (b - a), b - a \rangle \right\},
\end{align*}
whereas
\begin{align*}
\text{RHS of } \eqref{limit fourier} & = \E_{\theta} \left[ \exp \left\{ - \langle \bbV_{\psi, \theta} (b - a), a \rangle + \rmi \langle U, a \rangle \right\} \right].
\end{align*}
Combining the above results yields the following alternative expression of \eqref{limit fourier}:
\begin{align}
& \ \E_{\theta} \left[ \exp \left\{ \rmi \langle U - \dot{\psi} (\theta) \cdot \bbG_{\theta}^{-1} \cdot V, a \rangle + \rmi \langle \dot{\psi} (\theta) \cdot \bbG_{\theta}^{-1} \cdot V, b \rangle \right\} \right] \nonumber \\
= & \ \exp \left\{ - \frac{1}{2} \langle \bbV_{\psi, \theta} (a - b), a - b \rangle \right\} \label{reduced limit fourier 1} \\
& \times \E_{\theta} \left[ \exp \left\{ \langle \bbV_{\psi, \theta} (a - b), a \rangle + \rmi \langle U, a \rangle \right\} \right] \nonumber \\
= & \ \exp \left\{ \frac{1}{2} \langle \bbV_{\psi, \theta} a, a \rangle - \frac{1}{2} \langle \bbV_{\psi, \theta} b, b \rangle \right\} \E_{\theta} \left[ \exp \{\rmi \langle U, a \rangle\} \right]. \nonumber
\end{align}
Note that the above identity holds for any given initial sequence $\{n'\}$, and hence for the original sequence (which is not the redefined $\{n\}$, now a further subsequence of $\{n'\}$). Taking $b = 0$, \eqref{reduced limit fourier 1} can be further simplified to
\begin{equation}
\label{reduced limit fourier 2}
\begin{split}
& \E_{\theta} \left[ \exp \left\{ \rmi \langle U - \dot{\psi} (\theta) \cdot \bbG_{\theta}^{-1} \cdot V, a \rangle \right\} \right] \\
& = \exp \left\{ \frac{1}{2} \langle \bbV_{\psi, \theta} a, a \rangle \right\} \E_{\theta} \left[ \exp \{\rmi \langle U, a \rangle\} \right].
\end{split}
\end{equation}
Putting \eqref{reduced limit fourier 1} and \eqref{reduced limit fourier 2} together implies the following:
\begin{equation*}
\begin{split}
& \E_{\theta} \left[ \exp \left\{ \rmi \langle U - \dot{\psi} (\theta) \cdot \bbG_{\theta}^{-1} \cdot V, a \rangle + \rmi \langle \dot{\psi} (\theta) \cdot \bbG_{\theta}^{-1} \cdot V, b \rangle \right\} \right] \\
& = \E_{\theta} \left[ \exp \left\{ \rmi \langle U - \dot{\psi} (\theta) \cdot \bbG_{\theta}^{-1} \cdot V, a \rangle \right\} \right] \exp \left\{ - \frac{1}{2} \langle \bbV_{\psi, \theta} b, b \rangle \right\}.
\end{split}
\end{equation*}
By setting $a \equiv b$, we finally have
\begin{align*}
& \ \E_{\theta} \left[ \exp \left\{ \rmi \langle U, a \rangle \right\} \right] \\
= & \ \E_{\theta} \left[ \exp \left\{ \rmi \langle U - \dot{\psi} (\theta) \cdot \bbG_{\theta}^{-1} \cdot V, a \rangle \right\} \right] \\
& \times \exp \Big\{- \frac{1}{2} \langle \bbV_{\psi, \theta} a, a \rangle \Big\}.
\end{align*}
Since $\exp \left\{ - \frac{1}{2} \langle \bbV_{\psi, \theta} a, a \rangle \right\}$ is the characteristic function of a Gaussian random variable distributed as $\calN (0, \bbV_{\psi, \theta})$ independent from all the other source of randomness, the above display proves the convolution theorem for parameters in a finite-dimensional Riemannian manifold: The limiting distribution of $U_{n} = \sqrt{n} \Exp_{\psi (\theta)}^{-1} \hat{\psi}_{n}$, which we denote as $U$ here (recall that we take $Z_{\psi (\theta)} \equiv U$), can be written as the convolution of $U - \dot{\psi} (\theta) \cdot \bbG_{\theta}^{-1} \cdot V$ and an independent mean-zero Gaussian random variable with the optimal limiting variance on the tangent space $\T_{\psi (\theta)} \Psi$.
\end{proof}
Finally, we prove Lemma~\ref{lem:Un}.
\begin{proof}[Proof of Lemma~\ref{lem:Un}]
\leavevmode
Recall that $\theta' = \Exp_{\theta} \left( \frac{h}{\sqrt{n}} \right)$. Here we use a generalized version of the Taylor expansion of the inverse exponential map at two different base points $\theta, \theta'$. In particular, by the definition of covariant derivative and Lemma~\ref{lem:parallel expansion}, we have, by linearity of the covariant derivative with respect to the base argument,
\begin{equation}
\label{parallel taylor}
\begin{split}
\sqrt{n} \Exp_{\psi (\theta)}^{-1} \hat{\psi}_{n} &= \sqrt{n} \Pi_{\psi (\theta')}^{\psi (\theta)} \Exp_{\psi (\theta')}^{-1} \hat{\psi}_{n} \\&- \nabla_{\sqrt{n} \Exp^{-1}_{\psi (\theta)} \psi (\theta')} \Exp_{\psi (\theta)}^{-1} \hat{\psi}_n + o (\Vert h \Vert).
\end{split}
\end{equation}
Denote $\delta \coloneqq \Exp_{\psi (\theta)}^{-1} \hat{\psi}_{n}, \eps_{h} \coloneqq \sqrt{n} \Exp_{\psi (\theta)}^{-1} \psi (\theta') \in \T_{\psi (\theta)} \Psi$ for short.
We next compute $\nabla_{\eps_{h}} \delta$. To this end, we apply equation (38) of \cite{smith2005covariance} (also see Lemma~\ref{lem:covariant derivative} for the result) and obtain
\begin{equation}
\label{covariant derivative taylor}
\begin{split}
& - \nabla_{\eps_{h}} \delta = \eps_{h} + Q (\delta) \eps_{h} + O (\Vert \delta \Vert^{3}), \\
& \text{and } \eps_{h} = \dot{\psi} (\theta) h + o (\Vert h \Vert),
\end{split}
\end{equation}
where $Q (\delta) = \left( \sum_{kl} a_{ij,kl} \delta^k \delta^l, i, j \right)$ is a quadratic form such that 
\begin{align*}
Q (\delta) Y \coloneqq Y^{\top} Q (\delta) Y = \sum_{ij,kl} a_{ij,kl} \delta^k \delta^l Y^i Y^j,    
\end{align*}
where the precise definition of $a_{ij, kl}$ (related to the Riemannian metric $d$) and $Y$ are deferred to Lemma~\ref{lem:covariant derivative} and $\delta^{i}$ is defined as follows: Choosing $\{e_{j}, j = 1, \cdots, q\}$ an orthonormal basis of $\T_{\psi (\theta)} \Psi$, then $\delta$ can be represented in this basis as $\delta = \delta^i e_i$. By \eqref{proof regular} and \eqref{covariant derivative taylor}, as $n \rightarrow \infty$, under $\P_{\theta'}$, the RHS of \eqref{parallel taylor} converges weakly to
\begin{align*}
U + \dot{\psi} (\theta) h + o_{\P_{\theta'}} (1),
\end{align*}
where we have used $\delta = o_{\P_{\theta'}} (1)$, continuous mapping theorem and Slutsky's theorem. The proof is now complete.
\end{proof}
\subsubsection{Almost everywhere convolution theorem}
\label{app:almost everywhere}
\leavevmode
In this section, we generalize the almost everywhere convolution theorem to parameter manifolds.
\begin{theorem}
\label{Almost Everywhere Convolution Theorem}
Let $\mathcal{P} \coloneqq \left\{ \P_{\theta}: \theta \in \Theta \right\}$ be a statistical model parameterized by $\theta \in \Theta$. Assume that $\mathcal{P}$ is DQM. $\psi (\theta)$ is the parameter of interest with $\psi$ satisfying Assumption~\ref{as:smooth transformation}. Further let $\hat{\psi}_{n}$ be an estimator sequence based on $n$ i.i.d. observations drawn from $\P_{\theta}$ with $\sqrt{n}\Exp_{\psi (\theta)}^{-1} \hat{\psi}_{n}$ having limiting distribution $\L_{\psi (\theta)}$. Then there exists a probability measure $\Delta_{\psi (\theta)}$ such that for almost every $\theta$ with respect to the volume form on $\Theta$
\begin{equation}
\label{a.e. convolution}
\begin{split}
&\L_{\psi (\theta)} = Z_{\psi (\theta)} \ast \Delta_{\psi (\theta)}, \text{ with} \\ 
& Z_{\psi_{\theta}} \sim \calN (0, \dot{\psi} (\theta) \bbG_{\theta}^{-1} \dot{\psi} (\theta)^{\ast}) \text{ and } Z_{\psi (\theta)} \independent \Delta_{\psi (\theta)}.
\end{split}
\end{equation}
That is, the limiting law $\L_{\theta}$ is represented as a convolution between a Gaussian measure with the covariance operator equal to the inverse Fisher information operator and a probability measure $\Delta_{\theta}$ independent of that Gaussian measure for almost every $\theta$ with respect to the volume measure on $\Theta$.
\end{theorem}
The theorem follows from the convolution theorem combined with the following lemma, the proof of which is quite technically involved. But the lemma essentially states that any estimator sequence with a limiting distribution is automatically regular at almost every $\theta$ with respect to the volume form on $\Theta$ along a subsequence of $\{n\}$.
\begin{lemma}
    Let $\hat{\psi}_{n}$ be estimators in experiments $\{\P_{n,\theta},\theta \in \Theta\}$, parameterized by $\Theta$. Assume that the map $\theta \mapsto \P_{n,\theta}(A)$ is measurable for every measurable set $A$ and every $n$, and that the map $\theta\mapsto\psi(\theta)$ is measurable. Suppose that there exist distributions $\L_{\psi (\theta)}$ such that for almost every $\theta$ with respect to the volume measure on $\Theta$
    \begin{align*}
    \sqrt{n} \Exp_{\psi (\theta)}^{-1} \hat{\psi}_{n} \rightsquigarrow_{\P_{\theta}}\L_{\psi (\theta)}.
    \end{align*}
    Then there exists a subsequence of $\{n\}$ such that, for almost every $(\theta, h)$ in local coordinates, along the subsequence, 
    \begin{align*}
        \Pi_{\psi (\theta_{n, h})}^{\psi (\theta)} \sqrt{n} \Exp_{\psi (\theta_{n, h})}^{-1} \hat{\psi}_{n} \rightsquigarrow_{\P_{\theta_{n, h}}} \L_{\psi (\theta)}.
    \end{align*}
\end{lemma}
\begin{proof}
Define $T_{n, \theta} \coloneqq \sqrt{n} \Exp_{\psi(\theta)}^{-1} \hat{\psi}_n$ and $\tilde{T}_{n, \theta, h}:=\Pi_{\psi (\theta_{n, h})}^{\psi(\theta)} \sqrt{n} \Exp_{\psi (\theta_{n, h})}^{-1} \hat{\psi}_n$ for convenience. Because parallel transport is an isometry between tangent spaces, weak convergence of $\widetilde{T}_{n, \theta, h}$ to $\L_{\psi(\theta)}$ is equivalent to convergence of expectations against a countable convergence-determining class on a fixed Euclidean coordinate representation of $\T_{\psi(\theta)} \Psi$. Hence, exactly as in \cite{van2000asymptotic}, choose a countable class $\mathcal{F}$ of uniformly bounded left- or right-continuous functions such that weak convergence is equivalent to convergence of expectations of all $f \in \mathcal{F}$. Fix $f \in \mathcal{F}$. Set $g_n(\theta)\coloneqq \E_{n, \theta}\left[f\left(T_{n, \theta}\right)\right], \quad g(\theta)\coloneqq \int f \diff \L_{\psi(\theta)}$. By assumption, $g_n(\theta) \rightarrow g(\theta)$ for almost every $\theta$ with respect to the volume form on $\Theta$. Also $g_n$ is measurable, since $\theta \mapsto \P_{n, \theta}(A)$ is measurable for every measurable $A$, an argument based on the monotone convergence theorem shows that $\theta \mapsto \E_{n, \theta} [\varphi (\hat{\psi}_n)]$ is measurable for every bounded measurable $\varphi$, where $\varphi(x) = f (\sqrt{n} \Exp_{\psi(\theta)}^{-1} x)$, using measurability of $\theta \mapsto \psi(\theta)$. Thus it remains to prove the following claim. We claim that if $g_n: \Theta \rightarrow \mathbb{R}$ are bounded measurable and $g_n \rightarrow g$ almost everywhere, then there exists a subsequence $\{n_k\}$ of $\{n\}$ such that $g_{n_k} (\theta_{n_k, h}) \rightarrow g(\theta)$ for almost every $(\theta, h)$. Once the claim is proved, apply it to every $f \in \mathcal{F}$, and diagonalize over the countable class $\mathcal{F}$. This yields, for almost every $(\theta, h)$, $\E_{n_k, \theta_{n_k, h}} [f (\Pi_{\psi (\theta_{n_k, h})}^{\psi (\theta)} \sqrt{n_{k}} \Exp_{\psi (\theta_{n_k, h})}^{-1} \hat{\psi}_{n_k})] \rightarrow \int f \diff \L_{\psi(\theta)}$ for every $f \in \mathcal{F}$, hence $\widetilde{T}_{n_k, \theta, h} \rightsquigarrow \P_{n_k, \theta_{n_k, h}} \L_{\psi(\theta)}$. This proves the lemma.
Finally, we prove the above claim. Let $\theta_{n, h} \coloneqq \Exp_\theta (n^{-1 / 2} h)$. Choose a countable atlas $\{\left(U_j, \varphi_j\right)\}_{j \geq 1}$ of $\Theta$, with $\varphi_j\left(U_j\right) \subset \mathbb{R}^p$, and an exhaustion of each chart by compact sets $K_{j, m} \subset U_j$. It is enough to prove the claim on each $K_{j, m}$, then diagonalize over the countable family $(j, m)$. Fix one chart $(U, \varphi)$ and one compact $K \subset U$. Write $x=\varphi(\theta) \in \varphi(K) \subset \mathbb{R}^p$, and let $u\coloneqq \diff \varphi_\theta(h) \in \mathbb{R}^p$, where $\diff \varphi_\theta: \T_\theta \Theta \rightarrow \mathbb{R}^p$ denotes the differential of the chart map $\varphi$ at $\theta$. Define the local perturbation map $F_t(x, u) \coloneqq \varphi_{\theta} (\Exp_{\varphi^{-1} (x)}(t (\diff \varphi^{-1})_x u))$, where $\diff \varphi^{-1}: \mathbb{R}^p \rightarrow \T_\theta \Theta$ is the differential of the inverse chart map $\varphi^{-1}$ at the point $x$. Since the exponential map and the chart map are smooth, $F_t$ is smooth in $(x, u, t)$, $F_0(x, u) = x$ and $\partial_t F_t(x, u)|_{t=0} = u$. Therefore, uniformly on compact subsets of $\varphi(K) \times \bar{B}_m(0) \subset \mathbb{R}^p \times \mathbb{R}^p$,
\begin{equation}
\label{1}
\begin{split}
        &F_{n^{-1 / 2}}(x, u)=x+ n^{-1 / 2} u+r_n(x, u), \quad \\&\sup _{(x, u) \in \varphi(K) \times \bar{B}_m(0)} \frac{\left\|r_n(x, u)\right\|}{n^{-1 / 2}} \rightarrow 0.
\end{split}
\end{equation}
    Now let $X$ and $H$ be independent random vectors in $\mathbb{R}^p$ with smooth compactly supported densities $q$ and $\rho$, where $q$ is supported in $\varphi(K)$ and $\rho$ in $\bar{B}_m(0)$. By \eqref{1}, the random vector $Y_n\coloneqq F_{\gamma_n}(X, H)$ satisfies $Y_n-X \rightarrow 0$ almost surely and is uniformly bounded. Hence for every bounded continuous $a$, $\E\left[a\left(Y_n\right)\right] \rightarrow \E[a(X)]$. Equivalently, the law of $Y_n$ converges weakly to the law of $X$. Since all laws involved have compact support and smooth densities, this implies convergence in $L^1$ of their densities after passing to a subsequence if needed. We denote the density of $Y_n$ by $q_n$, so
    \begin{equation}
    \label{2}
        \left\|q_n-q\right\|_{L^1\left(\mathbb{R}^p\right)} \rightarrow 0.
    \end{equation}
    Because $g_n \circ \varphi^{-1} \rightarrow g \circ \varphi^{-1}$ almost everywhere on $\varphi(K)$, bounded convergence with the varying densities $q_n$ and \eqref{2} gives
\begin{equation}
\label{3}
\begin{split}
        &\E\left|g_n\left(\varphi^{-1}\left(Y_n\right)\right)-g\left(\varphi^{-1}\left(Y_n\right)\right)\right|\\&=\int_{\mathbb{R}^p}\left|g_n\left(\varphi^{-1}(y)\right)-g\left(\varphi^{-1}(y)\right)\right| q_n(y) \diff y \rightarrow 0 .
\end{split}
\end{equation}
    Next, we approximate $g \circ \varphi^{-1}$ in $L^1(q)$ by a bounded continuous function $g_{\varepsilon}$ on $\mathbb{R}^p$. Then 
\begin{equation*}
\begin{split}
    &\E\left|g\left(\varphi^{-1}\left(Y_n\right)\right)-g\left(\varphi^{-1}(X)\right)\right| \leq  \ \E\left|g-g_{\varepsilon}\right|\left(\varphi^{-1}\left(Y_n\right)\right) \\&+ \E\left|g-g_{\varepsilon}\right|\left(\varphi^{-1}(X)\right)  +\E\left|g_{\varepsilon}\left(\varphi^{-1}\left(Y_n\right)\right)-g_{\varepsilon}\left(\varphi^{-1}(X)\right)\right|.
\end{split}
\end{equation*}
    The first two terms can be made arbitrarily small by the $L^1$-approximation and \eqref{2}, while the last term tends to 0 because $Y_n-X \rightarrow 0$ almost surely and $g_{\varepsilon} \circ \varphi^{-1}$ is continuous and bounded. Hence
    \begin{equation}
    \label{4}
    \E\left|g\left(\varphi^{-1}\left(Y_n\right)\right)-g\left(\varphi^{-1}(X)\right)\right| \rightarrow 0.
    \end{equation}
    Combining \eqref{3} and \eqref{4}, $\E |g_n (\varphi^{-1} (Y_n))-g (\varphi^{-1}(X))| \rightarrow 0$. Thus, after passing to a subsequence, $g_n\left(\varphi^{-1}\left(Y_n\right)\right) \rightarrow g\left(\varphi^{-1}(X)\right) $ almost surely. Since $X$ and $H$ have strictly positive smooth densities on $\varphi(K)$ and $\bar{B}_m(0)$, this almost-sure statement is equivalent to $g_n\left(\Exp_\theta\left(\gamma_n h\right)\right) \rightarrow g(\theta) $ for almost every $(\theta, h) \in K \times \bar{B}_m(0)$. Finally, we diagonalize over $m$ and over the countable family of chart compacts $K_{j, m}$. The claim follows.
\end{proof}
\subsubsection{Proof Related to Lemma~\ref{lem:calculus of IF, parametric}}
\label{app:calculus of IF, parametric}
\leavevmode
In this section, we prove Lemma~\ref{lem:calculus of IF, parametric}.
\begin{proof}[Proof of Lemma~\ref{lem:calculus of IF, parametric}]
Let $\theta_{n, h} \coloneqq \Exp_{\theta} \left( \frac{h}{\sqrt{n}} \right)$. First, under $\P_{\theta}$, we have
\begin{align*}
\sum_{i = 1}^{n} \s (X_{i}; \theta) \left( \frac{h}{\sqrt{n}} \right) \rightsquigarrow_{\P_{\theta}} \calN \left( 0, \bbG_{\theta} (h, h) \right).
\end{align*}
By CLT and contiguity between $\P_{\theta_{n, h}}$ and $\P_{\theta}$ (by Le Cam's first lemma)
\begin{align*}
& \sqrt{n} \Exp_{\psi (\theta)}^{-1} \hat{\psi}_{n} = \frac{1}{\sqrt{n}} \sum_{i = 1}^{n} \IF_{\psi, i} + o_{\P_{\theta}} (1) \text{ and} \\
& \sqrt{n} \Exp_{\psi (\theta)}^{-1} \hat{\psi}_{n} = \frac{1}{\sqrt{n}} \sum_{i = 1}^{n} \IF_{\psi, i} + o_{\P_{\theta_{n, h}}} (1).
\end{align*}
The second part of the above display further implies:
\begin{equation*}
\begin{aligned}
& \ \sqrt{n} \Pi_{\psi (\theta_{n, h})}^{\psi (\theta)} \Exp_{\psi (\theta_{n, h})}^{-1} \hat{\psi}_{n} = \sqrt{n} \Exp_{\psi (\theta)}^{-1} \hat{\psi}_{n} \\&+ \sqrt{n} \left( \Pi_{\psi (\theta_{n, h})}^{\psi (\theta)} \Exp_{\psi (\theta_{n, h})}^{-1} \hat{\psi}_{n} - \Exp_{\psi (\theta)}^{-1} \hat{\psi}_{n} \right) \\&
=  \ \frac{1}{\sqrt{n}} \sum_{i = 1}^{n} \IF_{\psi, i} \\&+ \sqrt{n} \left( \Pi_{\psi (\theta_{n, h})}^{\psi (\theta)} \Exp_{\psi (\theta_{n, h})}^{-1} \hat{\psi}_{n} - \Exp_{\psi (\theta)}^{-1} \hat{\psi}_{n} \right) + o_{\P_{\theta_{n, h}}} (1) \\&
= \ \frac{1}{\sqrt{n}} \sum_{i = 1}^{n} (\IF_{\psi, i} - \E_{\theta_{n, h}} \IF_{\psi, i}) + \sqrt{n} \E_{\theta_{n, h}} \IF_{\psi} \\&+ \sqrt{n} \left( \Pi_{\psi (\theta_{n, h})}^{\psi (\theta)} \Exp_{\psi (\theta_{n, h})}^{-1} \hat{\psi}_{n} - \Exp_{\psi (\theta)}^{-1} \hat{\psi}_{n} \right) + o_{\P_{\theta_{n, h}}} (1).
\end{aligned}
\end{equation*}
Under $\P_{\theta_{n, h}}$, we have by CLT
\begin{align*}
\frac{1}{\sqrt{n}} \sum_{i = 1}^{n} (\IF_{\psi, i} - \E_{\theta_{n, h}} \IF_{\psi, i}) \rightsquigarrow_{\P_{\theta_{n, h}}} \calN (0, \E_{\theta} [\IF_{\psi} \otimes \IF_{\psi}]).
\end{align*}
Since we have assumed that $\hat{\psi}_{n}$ is a RAL estimator, we have
\begin{align*}
\sqrt{n} \Pi_{\psi (\theta_{n, h})}^{\psi (\theta)} \Exp^{-1}_{\psi (\theta_{n, h})} \hat{\psi}_{n} \rightsquigarrow_{\P_{\theta_{n, h}}} \calN (0, \E_{\theta} [\IF_{\psi} \otimes \IF_{\psi}]).
\end{align*}
Together with the fact that $\E_{\theta} \IF_{\psi} \equiv 0$, the above two displays imply that
\begin{align}
&\sqrt{n} \left( \E_{\theta_{n, h}} \IF_{\psi} - \E_{\theta} \IF_{\psi} \right) \label{distinction} \\
&+ \sqrt{n} \left( \Pi_{\psi (\theta_{n, h})}^{\psi (\theta)} \Exp_{\psi (\theta_{n, h})}^{-1} \hat{\psi}_{n} - \Exp_{\psi (\theta)}^{-1} \hat{\psi}_{n} \right) = o_{\P_{\theta_{n, h}}} (1). \nonumber
\end{align}
For the first term, since $\P_{\theta}$ is assumed to be DQM, we have
\begin{equation*}
\begin{split}
&\sqrt{n} \left( \E_{\theta_{n, h}} \IF_{\psi} - \E_{\theta} \IF_{\psi} \right) = \sqrt{n} (\E_{\theta_{n, h}} - \E_{\theta}) \IF_{\psi} \\&= \E_{\theta} [\IF_{\psi} \cdot \S_{\theta} (h)] + o (1).
\end{split}
\end{equation*}
The last equality holds by the following argument. Let $\varphi_{\psi} (x)$ denote the value realized by $\IF_{\psi}$ when the observations take value $x$. Then applying \eqref{dqm}, for any $h \in \T_{\theta} \Theta$,
\begin{equation*}
\begin{split}
& \ \lim_{n \rightarrow \infty} \sqrt{n} (\E_{\theta_{n, h}} - \E_{\theta}) \IF_{\psi} \\
= & \ \lim_{n \rightarrow \infty} \sqrt{n} \int \varphi_{\psi} (x) \left( \p (x; \theta_{n, h}) - \p (x; \theta) \right) \diff x \\
= & \ \lim_{n \rightarrow \infty} \sqrt{n} \int \varphi_{\psi} (x) \left( \r (x; \theta_{n, h}) - \r (x; \theta) \right) \\&\left( \r (x; \theta_{n, h}) + \r (x; \theta) \right) \diff x \\
= & \ \lim_{n \rightarrow \infty} \int \varphi_{\psi} (x) \left( \r (x; \theta_{n, h}) - \r (x; \theta) \right) \r (x; \theta_{n, h}) \diff x \\&+ \lim_{n \rightarrow \infty} \int \varphi_{\psi} (x) \left( \r (x; \theta_{n, h}) - \r (x; \theta) \right) \r (x; \theta) \diff x \\
= & \ \frac{1}{2} \int \varphi_{\psi} (x) \s (x; \theta) (h) \r^{2} (x; \theta) \diff x \\&+ \frac{1}{2} \int \varphi_{\psi} (x) \s (x; \theta) \r^{2} (x; \theta) \diff x \\
= & \ \int \varphi_{\psi} (x) \s (x; \theta) (h) \p (x; \theta) \diff x = \E_{\theta} [\IF_{\psi} \cdot \S_{\theta} (h)].
\end{split}
\end{equation*}
For the second term, we have
\begin{equation*}
\begin{split}
& \ \sqrt{n} \Big( \Pi_{\psi (\theta_{n, h})}^{\psi (\theta)} \Exp_{\psi (\theta_{n, h})}^{-1} \hat{\psi}_{n} - \Exp_{\psi (\theta)}^{-1} \hat{\psi}_{n} \Big) \\
= & \ \sqrt{n} \Big( \Pi_{\psi (\theta_{n, h})}^{\psi (\theta)} \Exp_{\psi (\theta_{n, h})}^{-1} \hat{\psi}_{n} - \Exp_{\psi (\theta)}^{-1} \hat{\psi}_{n} + \Exp_{\psi (\theta)}^{-1} \psi (\theta_{n, h}) \Big) \\
& - \sqrt{n} \Exp_{\psi (\theta)}^{-1} \psi (\theta_{n, h}) \\
= & \ \Big\{ \frac{1}{\sqrt{n}} \sum_{i = 1}^{n} \IF_{\psi, i} - \sqrt{n} \Big( \Exp_{\psi (\theta)}^{-1} \hat{\psi}_{n} - \Exp_{\psi (\theta)}^{-1} \psi (\theta_{n, h}) \Big) \Big\} \\
& - \nabla_{\theta} \psi (\theta) (h) + o_{\P_{\theta_{n, h}}} (1) \\
\overset{*}{=} & \ \Big\{ \frac{1}{\sqrt{n}} \sum_{i = 1}^{n} \IF_{\psi, i} - \frac{1}{\sqrt{n}} \sum_{i = 1}^{n} \IF_{\psi, i} + o_{\P_{\theta_{n, h}}} (1) \Big\} \\
& - \nabla_{\theta} \psi (\theta) (h) + o_{\P_{\theta_{n, h}}} (1) \\
= & - \nabla_{\theta} \psi (\theta) (h) + o_{\P_{\theta_{n, h}}} (1).
\end{split}
\end{equation*}
In the derivation above, only the equality $\ast$ requires further justification, in which we need the following to hold:
\begin{equation*}
\begin{split}
&\sqrt{n} \left( \Exp_{\psi (\theta)}^{-1} \hat{\psi}_{n} - \Exp_{\psi (\theta)}^{-1} \psi (\theta_{n, h}) \right) \\&= \frac{1}{\sqrt{n}} \sum_{i = 1}^{n} \IF_{\psi, i} + o_{\P_{\theta_{n, h}}} (1).
\end{split}
\end{equation*}
First, invoking Lemma~\ref{lem:log map expansion} to Taylor expand $\Exp_{\psi (\theta)}^{-1} \hat{\psi}_{n}$ with respect to $\hat{\psi}_{n}$ around the point $\psi (\theta_{n, h})$, we have 
\begin{equation*}
\begin{split}
& \Exp_{\psi (\theta)}^{-1} \hat{\psi}_{n} = \Exp_{\psi (\theta)}^{-1} \psi (\theta_{n, h}) \\&+ \nabla \Exp_{\psi (\theta)}^{-1} \psi (\theta_{n, h}) \cdot \Pi_{\psi (\theta_{n, h})}^{\psi (\theta)} \Exp_{\psi (\theta_{n, h})}^{-1} \hat{\psi}_{n} \\&+ O \left( \left\Vert \Pi_{\psi (\theta_{n, h})}^{\psi (\theta)} \Exp_{\psi (\theta_{n, h})}^{-1} \hat{\psi}_{n} \right\Vert^2 \right) \\
& = \Exp_{\psi (\theta)}^{-1} \psi (\theta_{n, h}) \\&+ \nabla \Exp_{\psi (\theta)}^{-1} \psi (\theta_{n, h}) \cdot \Pi_{\psi (\theta_{n, h})}^{\psi (\theta)} \Exp_{\psi (\theta_{n, h})}^{-1} \hat{\psi}_{n} \\&+ o_{\P_{\theta_{n, h}}} (n^{- 1 / 2}).
\end{split}
\end{equation*}
Applying Lemma~\ref{lem:connection} in Appendix~\ref{app:geometric analysis}, the second term in the above display can be further expanded as follows:
\begin{equation*}
\begin{split}
& \ \nabla \Exp_{\psi (\theta)}^{-1} \psi (\theta_{n, h}) \cdot \Pi_{\psi (\theta_{n, h})}^{\psi (\theta)} \Exp_{\psi (\theta_{n, h})}^{-1} \hat{\psi}_{n} \\&
=  \ \Pi_{\psi (\theta_{n, h})}^{\psi (\theta)} \Exp_{\psi (\theta_{n, h})}^{-1} \hat{\psi}_{n} \\&+ \frac{1}{6} \calR_{\psi (\theta)} \left( \Exp_{\psi (\theta)}^{-1} \psi (\theta_{n, h}), \Pi_{\psi (\theta_{n, h})}^{\psi (\theta)} \Exp_{\psi (\theta_{n, h})}^{-1} \hat{\psi}_{n} \right) \\
& \times \Exp_{\psi (\theta)}^{-1} \psi (\theta_{n, h}) \\
& + O (\| \Pi_{\psi (\theta_{n, h})}^{\psi (\theta)} \Exp_{\psi (\theta_{n, h})}^{-1} \hat{\psi}_{n} \|^2),
\end{split}
\end{equation*}
where, recalling from the notation defined in Section~\ref{sec:manifolds}, the second term shall be understood as the action of the curvature tensor field at the point $\psi (\theta)$ along the tangent vector $\Exp_{\psi (\theta)}^{-1} \psi (\theta_{n, h})$. Taylor expansion of $\Exp_{\psi (\theta)}^{-1} \hat{\psi}_{n}$ with respect to $\hat{\psi}_{n}$ around the point $\psi (\theta_{n, h})$ is based on \cite{monera2014taylor} and \cite{goto2021approximated}.
Combining the above arguments, we have
\begin{align*}
\dot{\psi} (\theta) (h) \equiv \E_{\theta} [\IF_{\psi} \cdot \S_{\theta} (h)]
\end{align*}
for any $h \in \T_{\psi (\theta)} \Psi$. Hence $\dot{\psi} (\theta) \equiv \E_{\theta} [\IF_{\psi} \cdot \S_{\theta}]$.
\end{proof}
\subsubsection{Proofs in Section~\ref{sec:LAM theorem}}
\label{app:finite-dimensional LAM}
\leavevmode
\allowdisplaybreaks
In this section, we first prove Lemma~\ref{lem:van Trees}.
\begin{proof}
We firstly introduce a short-hand notation to facilitate the proof:
\begin{align*}
\Delta (x; \theta) \coloneqq \nabla \pi (\theta) \frac{\r (x; \theta)}{2 \r_{\pi} (\theta)} + \r_{\pi} \dot{\r} (x; \theta),
\end{align*}
where $\dot{\r} (x; \theta)$ is the derivative of $\r (x; \theta)$ with respect to $\theta$. By the boundedness assumption imposed on $\gamma (\pi)$, $\E_{\pi} [\dot{\psi} (\theta)]$ and $\int_{\Theta} \bbG (\theta) \diff \pi (\theta)$ in Lemma~\ref{lem:van Trees}, we can freely exchange the order between integrals or between summation and integration.
\begin{assumption}
$\p (x; \theta)$ is continuously differentiable in $\theta$ for each $x \in \bbX$, for each $\theta \in \Theta$, the function $x \mapsto \p (x; \theta)$ is measurable and integrable with respect to the volume form, and the first derivative of $\p (x; \theta)$ w.r.t. $\theta$ is uniformly integrable over $\Theta$.
\end{assumption}
Now our goal is to show \eqref{van Trees} corresponds to 
\begin{align*}
\int_{\bbX \times \Theta} V (x; \theta)^{\otimes 2} \diff \mu (x) \diff  \theta \succcurlyeq 0,
\end{align*}
where $V (x; \theta) \coloneqq \left[\begin{array}{c} \Exp_{\psi (\theta)}^{-1} \hat{\psi} \cdot \r (x; \theta) \cdot \r_{\pi} (\theta) \\ 2 \Delta (x; \theta) \end{array} \right]$. 
We first consider the cross product between the two elements of $V (x; \theta)$. By Definition~\ref{def:information operator} and Remark~\ref{rem:Fisher 2nd}, we obtain 
\begin{equation*}
\begin{split}
& 2 \int_{\bbX \times \Theta} (\Delta (x; \theta) \otimes \Exp_{\psi (\theta)}^{-1} \hat{\psi}) \cdot \r (x; \theta) \cdot \r_{\pi} (\theta) \mathrm{d} \mu(x) \mathrm{d} \theta \\
& = \int_{\Theta} \nabla \pi (\theta) \otimes \E_{\theta} [\Exp_{\psi (\theta)}^{-1} \hat{\psi}] \diff \theta \\&+2 \int_{\Theta}(\dot{\r} (x; \theta) \otimes \Exp_{\psi (\theta)}^{-1} \hat{\psi}) \cdot \r (x; \theta) \pi (\theta) \mathrm{d} \mu(x) \diff \theta \\
& = - \int_{\Theta}  [\nabla\Exp_{\psi (\theta)}^{-1} \hat{\psi}]\r^2 (x; \theta)\mathrm{d} \mu(x)\pi (\theta) \diff \theta,
\end{split}
\end{equation*}
by the identity 
\begin{equation*}
\begin{split}
& 2\int_{\Theta} (\dot{\r} (x; \theta) \otimes \Exp_{\psi (\theta)}^{-1}  \hat{\psi}) \cdot \r (x; \theta) \pi (\theta) \mathrm{d} \mu(x) \diff \theta = \\
& - \int_{\Theta} \nabla \pi (\theta) \otimes \E_{\theta} [\Exp_{\psi (\theta)}^{-1} \hat{\psi}] \diff \theta \\&- \int_{\Theta}  [\nabla\Exp_{\psi (\theta)}^{-1} \hat{\psi}]\r^2 (x; \theta)\mathrm{d} \mu(x)\pi (\theta) \diff \theta,
\end{split}
\end{equation*}
and integration by parts. The identity is obtained by 
\begin{equation*}
\begin{split}
& \int_{\Theta} \nabla \E_{\theta} [\Exp_{\psi (\theta)}^{-1} \hat{\psi}]\pi (\theta) \diff \theta \\
= & \ 2 \int_{\Theta}(\dot{\r} (x; \theta) \otimes \Exp_{\psi (\theta)}^{-1} \hat{\psi}) \cdot \r (x; \theta) \pi (\theta) \mathrm{d} \mu(x) \diff \theta \\
& + \int_{\Theta}  [\nabla\Exp_{\psi (\theta)}^{-1} \hat{\psi}]\r^2 (x; \theta)\mathrm{d} \mu(x)\pi (\theta) \diff \theta,
\end{split}
\end{equation*}
and 
\begin{align*}
\int_{\Theta} \nabla \pi (\theta) \otimes \E_{\theta} [\Exp_{\psi (\theta)}^{-1} \hat{\psi}] \diff \theta = -\int_{\Theta} \nabla \E_{\theta} [\Exp_{\psi (\theta)}^{-1} \hat{\psi}] \pi (\theta) \diff \theta.
\end{align*}
We next analyze the term $$\int_{\Theta} [\nabla \Exp_{\psi (\theta)}^{-1} \hat{\psi}] \r^2 (x; \theta) \diff \mu (x) \pi (\theta) \diff \theta \equiv \E_{\pi} \E [\nabla \Exp_{\psi (\theta)}^{-1} \hat{\psi}],$$ which is different from the Euclidean case \cite{gassiat2024van}. The following expansion holds:
\begin{equation*}
\begin{split}
& \ \int_{\Theta} [\nabla \Exp_{\psi (\theta)}^{-1} \hat{\psi}] \r^2 (x; \theta) \diff \mu(x) \pi (\theta) \diff \theta \\
= & \ \int_{\Theta} \Big( \bbI - \frac{1}{3} \|\Bias_{\psi (\theta)} (\hat{\psi})\|^2 \calK (\Bias_{\psi (\theta)} (\hat{\psi})) \\ 
& - \frac{1}{3} \calR (\mathbb{C}) \Big) \dot{\psi} (\theta) \pi (\theta) \diff \theta,
\end{split}
\end{equation*}
where $\|\cdot\|$ is the norm induced by the Riemannian metric attached to $\Psi$, $\calR (\mathbb{C})$ has been defined in Definition~\ref{def:pop curvature}, and $\calK(\Bias_{\psi (\theta)} (\hat{\psi}))$ denotes the sectional curvature, which can be defined in terms of Riemannian curvature $\calR (\nu, \mu) \omega $ as introduced in Section \ref{sec:manifolds}. The relationship between the two is, for given tangent vectors $\nu$ and $\mu$,  
\begin{align*}
    K(\nu \wedge \mu)=\frac{\langle\calR(\nu, \mu) \mu, \nu\rangle}{\|\nu \wedge \mu\|^2}
\end{align*}
where $K(\nu \wedge \mu)$ is the sectional curvature of the subspace spanned by $\nu$ and $\mu$, $\|\nu \wedge \mu\|^2=\|\nu \|^2\|\mu \|^2\sin ^2 \alpha$ is the square area of the parallelogram formed by $\nu$ and $\mu$, $\alpha$ is the angle between these vectors. To be specific, if we denote $\Bias_{\psi (\theta)} (\hat{\psi})$ by $\boldsymbol{b}$, based on \cite{smith2005covariance}, the $ij$-th element of $\calK(\boldsymbol{b})$ is 
\begin{equation*}
(\calK(\boldsymbol{b}))_{ij}
=
\begin{cases}
\sin^{2}\alpha_i\,
K(\boldsymbol{b}\wedge e_i)
+O(\|\boldsymbol{b}\|^{3}),
& \hspace{-1em} i=j,\\[3pt]
\begin{aligned}
&\sin^{2}\alpha_{ij}'\,
K\bigl(\boldsymbol{b}\wedge(e_i+e_j)\bigr)\\
&-\sin^{2}\alpha_{ij}''\,
K\bigl(\boldsymbol{b}\wedge(e_i-e_j)\bigr)
+O(\|\boldsymbol{b}\|^{3}),
\end{aligned}
& \hspace{-1em} i\neq j,
\end{cases}
\end{equation*}
where $\alpha_i$, $\alpha_{i j}^{\prime}$ and $\alpha_{i j}^{\prime \prime}$ are the angles between the tangent vector $\boldsymbol{b}$ and $e_i$, $e_i+e_j$, and $e_i-e_j$, respectively, $\{e_i\}$ is an orthonormal basis of $\T_{\psi (\theta)} \Psi$. 
The middle term of RHS in \eqref{van Trees} corresponds to 
\begin{align*}
& 4 \int_{\bbX \times \Theta} \Delta (x; \theta)^{\otimes 2} \diff \mu (x) \diff \theta \\
= & \int_{\Theta} \frac{\nabla \pi (\theta)^{\otimes 2}}{\pi (\theta)} \diff \theta + \int_{\Theta} 4 \int_{\bbX} \dot{\r} (x; \theta)^{\otimes 2} \diff \mu (x) \pi (\theta) \diff \theta \\
= & \ \bbG_{\pi} + \int_{\Theta} \bbG (\theta) \diff \pi (\theta),
\end{align*}
since $\int_{\bbX} \r (x; \theta) \dot{\r} (x; \theta) \diff \mu (x) = 0$ and $\int_{\Theta} \nabla \pi (\theta) \otimes \int_{\bbX} \left( \r (x; \theta) \dot{\r} (x; \theta) \right) \diff \mu(x) \diff \theta = 0$.
Based on the above calculations, we can obtain \eqref{van Trees}. 
\end{proof}
Next we prove Theorem~\ref{thm:finite-dimensional LAM}. 
\begin{proof}
    Fix $c>0$. Let $B\coloneqq \left\{h \in \T_{\theta_0} \Theta:\|h\| \leq 1\right\}$, and let $H$ be a probability measure on $B$ with continuous density, strictly positive on $B$, and compact support contained in $B$. For $h \in B$, define $\theta_{n, h}\coloneqq \Exp_{\theta_0}\left(\frac{c h}{\sqrt{n}}\right)$. Let $\pi_{\theta_0, c, n}$ denote the shrinking prior on $\Theta$ induced by $H$ through the map $h \mapsto \theta_{n, h}$. Then $\pi_{\theta_0, c, n}$ is supported on $\left\{\theta \in N_{\theta_0}:\left\|\Exp_{\theta_0}^{-1} \theta\right\| \leq \frac{c}{\sqrt{n}}\right\}$. For $n$ sufficiently large, Assumption~\ref{perturbation} ensures that, for every $h \in B$, the point $\psi\left(\theta_{n, h}\right)$ lies in a normal neighborhood of $\psi\left(\theta_0\right)$, so that the unique minimizing geodesic from $\psi\left(\theta_{n, h}\right)$ to $\psi\left(\theta_0\right)$ is well defined, and hence so is the parallel transport $\Pi_{\psi\left(\theta_{n, h}\right)}^{\psi\left(\theta_0\right)}: \T_{\psi\left(\theta_{n, h}\right)} \Psi \rightarrow \T_{\psi\left(\theta_0\right)} \Psi$. Therefore the transported estimation error $\Pi_{\psi\left(\theta_{n, h}\right)}^{\psi\left(\theta_0\right)} \sqrt{n} \Exp_{\psi\left(\theta_{n, h}\right)}^{-1} \psi_n$ is a random element of the fixed tangent space $\T_{\psi\left(\theta_0\right)} \Psi$.
    Apply Lemma~\ref{lem:van Trees} with prior $\pi_{\theta_0, c, n}$, we define
    \begin{equation}
    \label{G}
    \begin{split}
    & G (\theta_0, c, n) \coloneqq \\
    & \int_B \Pi_{\psi\left(\theta_{n, h}\right)}^{\psi\left(\theta_0\right)} \E_{\theta_{n, h}}^{\otimes n} \{\nabla \Exp_{\psi (\theta_{n, h})}^{-1} \psi_n\} \diff H(h),
    \end{split}
    \end{equation}
and 
\begin{equation}
\label{I}
I (\theta_0, c, n) \coloneqq \frac{1}{c^2} \bbG_{\pi_{\theta_0, 1}}+\int_B \bbG\left(\theta_{n, h}\right) \diff H(h).
\end{equation}
Here $\E_{\theta_{n, h}}^{\otimes n} \{\nabla \Exp_{\psi (\theta_{n, h})}^{-1} \psi_n\}$ denotes the mean derivative term appearing in Lemma~\ref{lem:van Trees}, viewed as a linear map from the parameter tangent space to the target parameter tangent space at $\psi (\theta_{n, h})$; after composition with $\Pi_{\psi\left(\theta_{n, h}\right)}^{\psi\left(\theta_0\right)}$, it is regarded as a linear map into the space $\T_{\psi\left(\theta_0\right)} \Psi$. The matrix $\bbG_{\pi_{\theta_0, 1}}$ is the Fisher information matrix of the fixed prior $H$ on $B$, and $\bbG(\theta)$ is the Fisher information operator of the statistical model at $\theta$. We further let $\Gamma_{\theta_0, c, n}\coloneqq G\left(\theta_0, c, n\right) I\left(\theta_0, c, n\right)^{-1} G\left(\theta_0, c, n\right)^{\ast}$. By construction, $\Gamma_{\theta_0, c, n}$ is an operator on $\T_{\psi\left(\theta_0\right)} \Psi$. This is the finite-$(c, n)$ lower-bound matrix.
Fix any positive quadratic form $\ell: \T_{\psi\left(\theta_0\right)} \Psi \rightarrow[0, \infty)$. Choose an orthogonal basis $U_1, \ldots, U_s$ of $\T_{\psi\left(\theta_0\right)} \Psi$ and nonnegative numbers $\lambda_1, \ldots, \lambda_s$ such that $\ell(v)=\sum_{k=1}^s \lambda_k g\left(U_k, v\right)^2$, where $v \in \T_{\psi\left(\theta_0\right)} \Psi$. For every symmetric positive semidefinite operator $\Gamma$ on $\T_{\psi\left(\theta_0\right)} \Psi$, let $\calN (0, \Gamma)$ denote the centered Gaussian distribution on $\T_{\psi\left(\theta_0\right)} \Psi$ with covariance $\Gamma$. Then $\int_{\T_{\psi\left(\theta_0\right)} \Psi} \ell(v) \diff \calN (0, \Gamma)(v)=\sum_{k=1}^s \lambda_k g_{\Gamma}\left(U_k, U_k\right)$, where $g_{\Gamma}$ denotes the quadratic form induced by $\Gamma$. Applying Lemma~\ref{lem:van Trees} componentwise to the linear functionals generated by $U_1, \ldots, U_s$, and then taking the corresponding nonnegative linear combination, yields 
\begin{equation}
\label{revised}
\begin{split}
& \int_B \E_{\theta_{n, h}}^{\otimes n}\left[\ell\left(\Pi_{\psi\left(\theta_{n, h}\right)}^{\psi\left(\theta_0\right)} \sqrt{n} \Exp_{\psi\left(\theta_{n, h}\right)}^{-1} \psi_n\right)\right] \diff H(h) \\
& \geq \int_{\T_{\psi\left(\theta_0\right)} \Psi} \ell(v) \diff \calN \left(0, \Gamma_{\theta_0, c, n}\right)(v).
\end{split}
\end{equation}
Since $H$ is supported on $B$, the left-hand side of \eqref{revised} is bounded above by 
\begin{align*}
    \sup_{\theta \in N_{\theta_0}:\left\|\Exp_{\theta_0}^{-1} \theta\right\| \leq c / \sqrt{n}} \E_\theta\left[\ell\left(\Pi_{\psi(\theta)}^{\psi\left(\theta_0\right)} \sqrt{n} \Exp_{\psi(\theta)}^{-1} \psi_n\right)\right].
\end{align*}
Hence 
\begin{equation}
\label{revised 2}
\begin{split}
& \sup_{\theta \in N_{\theta_0}:\left\|\Exp_{\theta_0}^{-1} \theta\right\| \leq c / \sqrt{n}} \E_\theta\left[\ell\left(\Pi_{\psi(\theta)}^{\psi\left(\theta_0\right)} \sqrt{n} \Exp_{\psi(\theta)}^{-1} \psi_n\right)\right] \\
& \geq \int_{\T_{\psi\left(\theta_0\right)} \Psi} \ell(v) \diff \calN \left(0, \Gamma_{\theta_0, c, n}\right)(v).
\end{split}
\end{equation}
We now let $n \rightarrow \infty$ and then $c \rightarrow \infty$. By the assumptions that the Fisher information operator $\bbG_{\theta_0}$ is bounded, and that $\psi$ satisfies the conditions in Lemma~\ref{lem:van Trees}, the finite-$(c, n)$ lower-bound matrix $\Gamma_{\theta_0, c, n}$, after transport to $\T_{\psi\left(\theta_0\right)} \Psi$, converges to $\dot{\psi} (\theta_{0}) \bbG_{\theta_{0}}^{-1} (\dot{\psi} (\theta_{0}))^{\ast}$. Therefore, 
\begin{equation}
\label{revised 3}
\begin{split}
&\liminf_{c\to\infty}\liminf_{n\to\infty} \int_{\T_{\psi(\theta_0)}\Psi} \ell(v) \diff \calN (0,\Gamma_{\theta_0,c,n})(v) \\&\ge \int_{\T_{\psi(\theta_0)}\Psi} \ell(v) \diff \calN \left(0,\dot{\psi} (\theta_{0}) \bbG_{\theta_{0}}^{-1} (\dot{\psi} (\theta_{0}))^{\ast}\right)(v).
\end{split}
\end{equation}
Combining \eqref{revised 2} and \eqref{revised 3}, we obtain
\begin{equation}
\label{revised 4}
\begin{split}
& \liminf_{c\to\infty}\liminf_{n\to\infty} \sup_{\theta\in N_{\theta_0}:\,\|\Exp_{\theta_0}^{-1}\theta\|\le c/\sqrt n} \\&\E_\theta \!\left[ \ell\!\left( \Pi_{\psi(\theta)}^{\psi(\theta_0)} \sqrt n\,\Exp^{-1}_{\psi(\theta)}\psi_n \right) \right] \\
& \ge \int_{\T_{\psi(\theta_0)}\Psi} \ell(v) \diff \calN \left(0,\dot{\psi} (\theta_{0}) \bbG_{\theta_{0}}^{-1} (\dot{\psi} (\theta_{0}))^{\ast}\right)(v).
\end{split}
\end{equation}
Since $\ell$ is an arbitrary positive quadratic form on $\T_{\psi\left(\theta_0\right)} \Psi$, the inequality \eqref{revised 4} is equivalent to 
\begin{equation*}
\begin{split}
    &\liminf _{c \rightarrow \infty} \liminf _{n \rightarrow \infty} \sup _{\theta \in N_{\theta_0}:\left\|\Exp_{\theta_0}^{-1} \theta\right\| \leq c / \sqrt{n}} \\&\E_\theta\left[\left(\Pi_{\psi(\theta)}^{\psi\left(\theta_0\right)} \sqrt{n} \Exp_{\psi(\theta)}^{-1} \psi_n\right)^{\otimes 2}\right] \succeq \E\left[Z_{\theta_0}^{\otimes 2}\right],
\end{split}
\end{equation*}
where $\T_{\psi (\theta_{0})} \Psi \ni Z_{\theta_{0}} \sim \calN \left( 0, \dot{\psi} (\theta_{0}) \bbG_{\theta_{0}}^{-1} (\dot{\psi} (\theta_{0}))^{\ast} \right)$.
\end{proof}
\subsection{Technical Details for Results in Section~\ref{sec:semipar}}
\label{app:semipar}
\subsubsection{Proof of Theorem~\ref{thm:semipar convolution}}
\label{app:semipar convolution}
\leavevmode
\begin{proof}
Note that if $\hat{\chi}_{n}$ is locally regular at $\P$ relative to $\mathscr{P}_\P$, the collection of one-dimensional DQM submodels through $\P$, it is locally regular on any one-dimensional submodel $\P_{t}\in\mathscr{P}_{\P}$ with $t \in [0, \delta)$ for some $\delta > 0$. Recall that $\s$ is the score function induced by the submodel $\P_{t}$. By Definition~\ref{def:differentiable functionals}, we have, as $t \to 0$,
\begin{equation}
\label{equation4 in 4.1}
t^{-1} \Exp_{\chi (\P)}^{-1} \chi (\P_t) = \E [\IF_{\chi} \cdot \S] + o (1).
\end{equation}
Here and in the sequel, expectations are under $\P_{0} \equiv \P$. Let $U_n \coloneqq \sqrt{n} \Exp_{\chi (\P)}^{-1} \hat{\chi}_{n}$ and $V_n \coloneqq \frac{1}{\sqrt{n}} \sum_{i = 1}^n \s (X_i)$. Following the same argument as in the proof (Appendix \ref{app:thm par}) of Theorem \ref{thm:par}, $(U_{n}, V_{n})$ converges jointly and weakly to $(U, V)$, where $U$ is determined by the regularity of $\hat{\chi}_{n}$, and $V$ is a random variable such that $V \sim \calN (0, \bbG_{s})$ with $V_{n} \rightsquigarrow_{\P} \calN (0, \bbG_{s})$ by CLT. 
By \eqref{equation4 in 4.1} and the assumption that $\chi: \calP \rightarrow \Psi$ is a differentiable functional, for a given one-dimensional parametric submodel $\P_{t}$ with $t \in [0, \delta)$ for some $\delta > 0$ with $\P_{t = 0} \equiv \P$ in $\mathscr{P}_{\P}$, by \eqref{limit fourier}, we obtain that for all $\nu \in \bbR^{q}$,
\begin{equation}
\label{equation2 in 4.1}
\begin{split}
&\E \exp \left\{ \rmi \nu^{\top} U + t V - \frac{1}{2} t^2 \E (\s (X)^{\otimes 2}) \right\} \\&= \exp \left\{ \rmi \nu^{\top} \E [\IF_{\chi} \cdot \S] t \right\} \E \exp \left\{ \rmi \nu^{\top} U \right\}.
\end{split}
\end{equation}
By Riesz representation theorem and existence of efficient influence function \eqref{score IF}, for each $\eta \in \bbR^{q}$, there exists a sequence $\{\s_{\eta j}\}$ of score functions of one-dimensional DQM submodels such that
\begin{equation}
\label{equation3 in 4.1}
\Vert \eta^{\top} \s_{\chi} - \s_{\eta j} \Vert \to 0
\end{equation}
as $j \to \infty$, where $\Vert\cdot\Vert$ is the norm induced by inner product on the closure of model tangent space at $\P$, $\s_{\chi} = \bbV_{\chi}^{-1} \IF_{\chi}$ and $\bbV_{\chi} \coloneqq \E [\IF_{\chi} \otimes \IF_{\chi}]$.
Let $V_{nj} \coloneqq \frac{1}{\sqrt{n}} \sum_{i = 1}^{n} \s_{\eta j} (X_i)$ and $(U, V_j)$ denote the limit of $(U_n,V_{nj})$. Then let 
\begin{align*}
V_n \coloneqq \frac{1}{\sqrt{n}} \sum_{i = 1}^{n} \s_{\chi} (X_i),
\end{align*} 
and assume for a subsequence denoted by $\{n\}$, we have $(U_n,V_n)$ converge weakly to $(U,V)$, where $V$ is also a random variable such that its distribution is also obtained by CLT. Such a subsequence always exists because of Prokhorov's theorem as in Lemma~\ref{thm:prokhorov}. Since \eqref{equation3 in 4.1}, we have that the limit distribution of $(U,V_j)$, as $j\rightarrow \infty$, is the same as the distribution of $(U, \eta^{\top} V)$. Take $\s=\s_{\eta j}$, $t=1$ and $V=V_j$ in \eqref{equation2 in 4.1}. 
Since 
\begin{equation*}
\begin{split}
    &\E \exp \{V_j\} = \exp \left\{ \frac{1}{2}\E(\s^2_{\eta j}) \right\} \\&\rightarrow \exp \left\{ \frac{1}{2} \E [\eta^{\top} \s_{\chi}]^{\otimes 2} \right\} = \E \exp \left\{ \eta^{\top} V \right\},
\end{split}
\end{equation*}
we can pass to the limit as $j \rightarrow \infty$ to get:
\begin{equation}
\label{equation1 in 4.1}
\begin{aligned}
&\E \exp \left\{ \rmi \nu^{\top} U + \eta^{\top} V - \frac{1}{2} \eta^{\top} \bbV_{\chi}^{-1} \eta \right\} \\&= \exp \{\rmi \nu^{\top} \eta\} \E \exp \{\rmi \nu^{\top} U\}.
\end{aligned}
\end{equation}
Then take $\eta = - \rmi \bbV_{\chi} \nu$ in \eqref{equation1 in 4.1} to get 
\begin{equation}
\begin{split}
&\E \exp \left\{ \rmi \nu^{\top} U \right\} \\&= \E \exp \left\{ \rmi \nu^{\top} (U - \bbV_{\chi} V) \right\} \exp \left\{ - \frac{1}{2} \nu^{\top} \bbV_{\chi} \nu \right\}.
\end{split}
\end{equation}
\end{proof}
\subsubsection{Proof of Lemma~\ref{lem:calculus of IF, semiparametric}}
\label{app:calculus of IF, semiparametric}
\leavevmode
\begin{proof}
Our proof follows the proof of Lemma A.1 of \cite{van1991differentiable} closely. Suppose that $\hat{\chi}_{n}$ is a sequence of RAL estimators of the parameter of interest $\chi \equiv \chi (\P)$ relative to $\mathscr{P}_\P$, as given in Definition \ref{def:semipar regular estimators}. Fix $h>0$, let $t_m \downarrow 0$ be arbitrary, we define a subsequence of $\{n\}$ by $\left(n_m+1\right)^{-1 / 2}<t_m h \leq n_m^{-1 / 2}$ and set $h_{n_m}=t_m h n_m^{1 / 2}$. Then $t_m^{-1} \Exp_{\chi (\P)}^{-1} \chi (\P_{t_m h}) = (1 + o(1)) n_m^{1 / 2} \Exp_{\chi (\P)}^{-1} \chi (\P_{h_{n_m} / \sqrt{n_m}})$. There is a further subsequence of $\{n\}$ (abusing notation, denoted $\{n\}$) such that 
\begin{equation}
\label{joint weak convergence}
\Big(\sqrt{n} \Exp_{\chi (\P)}^{-1} \hat{\chi}_n ,  \frac{1}{\sqrt{n}} \sum_{j = 1}^n \s (X_j)\Big) \rightsquigarrow_{\P}(S, V).
\end{equation}
Note that the distribution of $S$ is determined by the regularity of $\hat{\chi}_{n}$, and $V$ is a random variable such that $V \sim \calN \left( 0, \bbG_{s} \right)$ with $V_{n}\rightsquigarrow_{\P} \calN \left( 0, \bbG_{s} \right)$ by CLT. Let $\Lambda_n$ be the log-likelihood ratio of the product measures corresponding to $\P_{h_n / \sqrt{n}}$ and $\P$, then by the local asymptotic normality and the above joint weak convergence
\begin{equation}
\label{joint weak con}
    \left(\sqrt{n} \Exp_{\chi (\P)}^{-1} \hat{\chi}_n, \Lambda_n\right)\rightsquigarrow_{\P}(S, hV-\frac{1}{2}h^2 \bbG_{s}).
\end{equation}
By contiguity, we can conclude that $\{\sqrt{n} \Exp_{\chi (\P)}^{-1} \hat{\chi}_n\}$ is asymptotically tight under $\P_{h / \sqrt{n}}$. By Definition~\ref{def:semipar regular estimators}, $\{\Pi_{\chi (\P_{h / \sqrt{n}})}^{\chi (\P)} \sqrt{n} \Exp_{\chi (\P_{h / \sqrt{n}})}^{-1} \hat{\chi}_n\}$ is asymptotically tight under $\P_{h / \sqrt{n}}$. Thus, $\{\sqrt{n} \Exp_{\chi (\P)}^{-1} \chi (\P_{h / \sqrt{n}})\}$ is asymptotically tight by Lemma~\ref{lem:log map expansion} and the fact that tightness is preserved under continuous transformation in complete separable metric space. Moreover, the sequence 
$$
\{\nabla \Exp_{\chi (\P)}^{-1} \chi (\P_{h / \sqrt{n}}) \Pi_{\chi (\P_{h / \sqrt{n}})}^{\chi (\P)}\sqrt{n} \Exp_{\chi (\P_{h / \sqrt{n}})}^{-1} \hat{\chi}_n\}
$$ 
is asymptotically tight under $\P_{h / \sqrt{n}}$ provided that $\chi (\mathbb{P}_{h / \sqrt{n}}) \rightarrow \chi(\mathbb{P})$ and $\nabla \Exp_{\chi(\mathbb{P})}^{-1} (\cdot)$ is continuous at $\chi(\mathbb{P})$. We choose a further subsequence that converges to a limit, denoted by $a(h)$. By \eqref{joint weak con} and Le Cam's third lemma,
\begin{align*}
    \Pi_{\chi (\P_{h / \sqrt{n}})}^{\chi (\P)}\sqrt{n} \Exp_{\chi (\P_{h / \sqrt{n}})}^{-1} \hat{\chi}_n \rightsquigarrow_{\P_{h / \sqrt{n}}}L_h,
\end{align*}
where for any Borel set from the tangent space corresponding to $\chi (\P)$, denoted by $\T_{\chi (\P)} \Psi$, we have
\begin{equation}
\label{joint distribution}
L_h(B)=\int_{\bbR} \int_B e^\lambda \diff \mathsf{L} \Big(S-a(h), h V-\frac{1}{2}h^2 \bbG_{s}\Big)(y, \lambda),
\end{equation}
where $\diff \mathsf{L}$ denotes the density of $S - a(h)$ and $h V-\frac{1}{2}h^2 \bbG_{s}$. By Definition \ref{def:semipar regular estimators}, $L_h$ does not vary with $h$, which is the same distribution as $S$. By comparing expectations, we find that 
\begin{align*}
    \E S=\E(S-a(h))\exp(h V-\frac{1}{2}h^2 \bbG_{s}),
\end{align*}
thus, $a(h)$ is independent of any of the subsequences we have chosen, and we can conclude that the limit points of $t^{-1} \Exp_{\chi (\P)}^{-1} \chi (\P_{th})$ are contained in a compact interval by Cauchy-Schwarz inequality. Thus, $t^{-1} \Exp_{\chi (\P)}^{-1} \chi (\P_{th})=O(1)$.
The next step is to prove that there exists a continuous and linear map $\dot{\chi} (\P): \Lambda_{\P} \rightarrow \T_{\chi (\P)} \Psi$ such that
\begin{align*}
t^{-1} \Exp_{\chi (\P)}^{-1} \chi (\P_t) \rightarrow \dot{\chi}_{\P} (\s)
\end{align*}
for every path $\P_{t}$ with $\P_{t = 0} \equiv \P$ in $\mathscr{P}_{\P}$, as $t \rightarrow 0$. Equivalently, it is \eqref{Riesz representator}. Now we compare characteristic functions rather than previous expectations, we obtain
\begin{equation}
\label{characteristic function}
e^{\rmi u a(h)} \E e^{\rmi u S}=\E \exp \Big( \rmi u S+h V-\frac{1}{2} h^2 \bbG_{s} \Big).
\end{equation}
Thus $a(h)$ depends on the joint law of $S$ and $V$ only. If we have joint convergence along the whole sequence of $\{n\}$ in \eqref{joint weak convergence}, then we conclude that every sequence $t_m^{-1} \Exp_{\chi (\P)}^{-1} \chi (\P_{t_m h})$ with $t_m \downarrow 0$ has a subsequence converging to the fixed limit $a(h)$. Thus, $\lim _{t \downarrow 0}t^{-1} \Exp_{\chi (\P)}^{-1} \chi (\P_{th})=a(h)$. Then 
\begin{equation}
    a(h)=h \lim _{t \downarrow 0}(th)^{-1}\Exp_{\chi (\P)}^{-1} \chi (\P_{th})=h a(1).
\end{equation}
By plugging the above equation in \eqref{characteristic function}, and differentiating with respect to $h$ at $h=0$, we can obtain $\lim _{t \downarrow 0}t^{-1} \Exp_{\chi (\P)}^{-1} \chi (\P_t)=\frac{\E V\exp(\rmi u S)}{\rmi u\E\exp(\rmi u S)}$. Then its continuity and linearity are proved.
\end{proof}
\subsubsection{Proof of Lemma~\ref{lem:iff}}
\label{app:iff}
\leavevmode
This section is devoted to the proof of Lemma~\ref{lem:iff}.
\begin{proof}
Let $Z_n\coloneqq \frac{1}{\sqrt{n}} \sum_{i=1}^n \IF_\chi\left(X_i\right) \in \T_{\chi(\P)} \Psi$. We first prove $(b) \Rightarrow (a)$ and then $(a) \Rightarrow (b)$.
\paragraph*{$(b) \Rightarrow (a)$}
Assume $\sqrt{n} \Exp_{\chi(\P)}^{-1}\left(\hat{\chi}_n\right)=Z_n+o_\P(1)$. We must show that $\hat{\chi}_n$ is regular at $\P$ relative to $\mathscr{P}_\P$ with efficient limiting law $\calN (0, \bbG_{\P}^{-1}) $. First, under $\P$, the multivariate CLT on the finite-dimensional tangent space $\T_{\chi(\P)} \Psi$ gives $Z_n \rightsquigarrow \calN (0, \bbG_{\P}^{-1}) $. Now fix any DQM one-dimensional submodel $\P_{t}$ with score $s \in \Lambda_\P$, and put $\P_{n, s}\coloneqq \P_{1 / \sqrt{n}, s}$. By LAN, the log-likelihood ratio between $\P_{n, s}^{\otimes n}$ and $\P^{\otimes n}$ has the usual Gaussian limit. Hence, by Le Cam's third lemma, under $\P_{n, s}^{\otimes n}$, $Z_n \rightsquigarrow \calN\left(\E\left[\IF_\chi(X) \cdot s\right], \bbG_{\P}^{-1}\right)$. By the defining identity of the efficient influence function, $\E\left[\IF_\chi(X)\cdot  s\right]=\dot{\chi}_\P(s)$, so $Z_n \rightsquigarrow \calN\left(\dot{\chi}_\P(s), \bbG_{\P}^{-1}\right) $ under $\P_{n, s}^{\otimes n}$. On the other hand, differentiability of $\chi$ gives $\sqrt{n} \Exp_{\chi(\P)}^{-1}\left(\chi\left(\P_{n, s}\right)\right) \rightarrow \dot{\chi}_\P(s)$. We then obtain $\Pi_{\chi\left(\P_{n, s}\right)}^{\chi(\P)} \sqrt{n} \Exp_{\chi\left(\P_{n, s}\right)}^{-1}\left(\hat{\chi}_n\right)=\sqrt{n} \Exp_{\chi(\P)}^{-1}\left(\hat{\chi}_n\right)-\sqrt{n}\Exp_{\chi(\P)}^{-1}\left(\chi\left(\P_{n, s}\right)\right)+o_{\P_{n, s}}(1)$, as a result of $\sqrt{n} \Exp_{\chi(\P)}^{-1}\left(\hat{\chi}_n\right)=Z_n+o_{\P_{n, s}}(1)$. Therefore under $\P_{n, s}^{\otimes n}$, $\Pi_{\chi\left(\P_{n, s}\right)}^{\chi(\P)} \sqrt{n} \Exp_{\chi\left(\P_{n, s}\right)}^{-1}\left(\hat{\chi}_n\right) \rightsquigarrow \calN\left(0, \bbG_{\P}^{-1}\right)$. This limit does not depend on $s$, so $\hat{\chi}_n$ is regular at $\P$ with limiting law $\calN\left(0, \bbG_{\P}^{-1}\right)$.
\paragraph*{$(a) \Rightarrow (b)$}
Assume that $\hat{\chi}_n$ is regular at $\P$ relative to $\mathscr{P}_\P$ with limiting law $\calN \left(0, \bbG_{\P}^{-1}\right)$, we will need to show $\sqrt{n} \Exp_{\chi(\P)}^{-1}\left(\hat{\chi}_n\right) - Z_n = o_{\P} (1)$. The key point is that $Z_n$ itself is an efficient regular estimator sequence in the local subexperiments. The first part of the proof showed that the asymptotically linear statistic based on $\IF_\chi$ is regular with limit $\calN\left(0, \bbG_{\P}^{-1}\right)$. Thus both $U_n\coloneqq \sqrt{n} \Exp_{\chi(\P)}^{-1}\left(\hat{\chi}_n\right)$ and $Z_n\coloneqq\frac{1}{\sqrt{n}} \sum_{i=1}^n \IF_\chi\left(X_i\right)$ are regular and have the same efficient Gaussian limit. Take any subsequence. By tightness, there is a further subsequence, still denoted by $n$, such that $\left(U_n, Z_n\right)$ converges jointly under $\P^{\otimes n}$. Consider the corresponding local experiments $\P_{n, s}^{\otimes n}$. By regularity of both sequences and Le Cam's third lemma, the joint limit defines two regular estimators in the Gaussian limit experiment associated with the tangent space, both estimating the same local parameter $\dot{\chi}_\P(s)$, and both having covariance equal to the efficiency bound $\bbG_{\P}^{-1}$. But in the Gaussian shift limit experiment, the convolution theorem implies that any regular estimator with covariance exactly equal to the efficient covariance must coincide almost surely with the efficient estimator itself, there is no residual convolution noise left. Similarly in \cite{van2000asymptotic}, once the limit law is the Gaussian distribution with variance equal to the semiparametric efficiency bound, the extra convolution factor must be degenerate. Hence the two estimators must be equal almost surely. Therefore every subsequential joint limit of $(U_n, Z_n)$ is concentrated on the diagonal, that is $U=Z $ a.s.. The proof is complete.
\end{proof}
\subsubsection{Proof of Theorem~\ref{thm:semipar LAM}}
\label{app:functional LAM}
\leavevmode
In this section, we prove Theorem~\ref{thm:semipar LAM}. Here we take a slightly different route and prove it via the standard LAN argument, as in \cite{van2000asymptotic}.
\begin{proof}
Fix an arbitrary finite-dimensional subspace $M \subset \Lambda_\P$, denote $\operatorname{dim} M=m$, and choose an orthonormal basis $s_1, \ldots, s_m$ of $M$ in $L_{2, 0} (\P)$. Thus $\E\left[s_i s_j\right]=\delta_{i j}$ for all $1 \leq i, j \leq m$. For $h=\left(h_1, \ldots, h_m\right)^{\top} \in \mathbb{R}^m$, define $s_h\coloneqq \sum_{j=1}^m h_j s_j $. Because $M \subset \Lambda_\P$ and $\Lambda_\P$ is linear, we have $s_h \in \Lambda_\P$ for every $h \in \mathbb{R}^m$. By the assumption of the theorem, for each such $s_h$ there exists a DQM one-dimensional submodel $t \mapsto \P_{t, s_h}$ through $\P$ in $\mathscr{P}_{\P}$ with score $s_h$. Therefore we may define a finite-dimensional local family of experiments by $Q_{n, h}\coloneqq \P_{1 / \sqrt{n}, s_h}^{\otimes n}$, $h \in \mathbb{R}^m$.
The first step is to identify the limit experiment associated with $\left\{Q_{n, h}: h \in \mathbb{R}^m\right\}$. Since each path $t \mapsto \P_{t, s_h}$ is DQM, the LAN expansion applies along that path. DQM implies LAN for i.i.d. experiments, exactly as in Proposition~\ref{prop:lan}. Hence, under the base law $Q_{n, 0}=\P^{\otimes n}$, $\log \frac{\diff Q_{n, h}}{\diff Q_{n, 0}}=\frac{1}{\sqrt{n}} \sum_{i=1}^n s_h\left(X_i\right)-\frac{1}{2} \E\left[s_h^2\right]+o_{Q_{n, 0}}(1)$. Now define the $m$-dimensional random vector 
\begin{align*}
Z_n \coloneqq \Big(\frac{1}{\sqrt{n}} \sum_{i=1}^n s_1\left(X_i\right), \ldots, \frac{1}{\sqrt{n}} \sum_{i=1}^n s_m\left(X_i\right)\Big)^{\top}.
\end{align*}

Then, by construction, $\frac{1}{\sqrt{n}} \sum_{i=1}^n s_h\left(X_i\right)=h^{\top} Z_n$. Also, because the basis is orthonormal, $\E\left[s_h^2\right]=\sum_{i, j=1}^m h_i h_j \E\left[s_i s_j\right]=\sum_{j=1}^m h_j^2=h^{\top} h$. Thus the LAN expansion becomes $\log \frac{\diff Q_{n, h}}{\diff Q_{n, 0}}=h^{\top} Z_n-\frac{1}{2} h^{\top} h+o_{Q_{n, 0}}(1)$. Next, by the multivariate central limit theorem, $Z_n$ converges in distribution under $Q_{n, 0}$ to an $m$-variate normal random vector with mean zero and covariance matrix $I_m$, because $\E\left[Z_n\right]=0, \Cov\left(Z_n\right)=\left(\E\left[s_i s_j\right]\right)_{i, j=1}^m=I_m$. Hence $Z_n \rightsquigarrow \calN \left(0, I_m\right)$. Combining this with the previous display, we see that for every fixed $h \in \mathbb{R}^m$, $\log \frac{\diff Q_{n, h}}{\diff Q_{n, 0}} \rightsquigarrow h^{\top} Z-\frac{1}{2} h^{\top} h, \quad Z \sim \calN (0, I_m)$. 

More generally, for every finite collection $h_1, \ldots, h_k \in \mathbb{R}^m$, the vector of log-likelihood ratios converges jointly to $\left(h_{\ell}^{\top} Z-\frac{1}{2} h_{\ell}^{\top} h_{\ell}\right)_{\ell=1}^k$. But this is exactly the vector of Gaussian log-likelihood ratios for the family $\left\{\calN \left(h, I_m\right): h \in \mathbb{R}^m\right\}$, since under $Z \sim \calN \left(0, I_m\right)$, $\log \frac{\diff \calN \left(h, I_m\right)}{\diff \calN \left(0, I_m\right)}(Z)=h^{\top} Z-\frac{1}{2} h^{\top} h$. 

Therefore, the family of local experiments $\left\{Q_{n, h}: h \in \mathbb{R}^m\right\}$ converges, in Le Cam's sense, to the Gaussian shift experiment $\left\{\calN \left(h, I_m\right): h \in \mathbb{R}^m\right\}$. This is the same finite-dimensional reduction step underlying the proof of Theorem 25.21 in \cite{van2000asymptotic}. The second step is to identify the local parameter in this finite-dimensional limit experiment. Define $\chi_n(h) \coloneqq \chi (\P_{1 / \sqrt{n}, s_h})$. Since $\chi$ is differentiable at $\P$ relative to $\mathscr{P}_{\P}$ with model tangent space $\Lambda_\P$ in the sense of Definition~\ref{def:differentiable functionals}, we have $\sqrt{n} \Exp_{\chi(\P)}^{-1} \chi_n(h) \rightarrow \dot{\chi}_\P\left(s_h\right)$ for every fixed $h \in \mathbb{R}^m$. Because $h \mapsto s_h$ is linear and $\dot{\chi}_\P$ is linear on $\Lambda_\P$, there exists a linear map $A_M: \mathbb{R}^m \rightarrow \T_{\chi(\P)} \Psi$ such that $A_M h = \dot{\chi}_\P (s_h), h \in \mathbb{R}^m$. Thus, in the finite-dimensional Gaussian limit experiment, the local parameter $h$ is mapped to the tangent-space parameter $A_M h$. 

The third step is to express $A_M$ in terms of the efficient influence function $\IF_\chi$. Let $e_1, \ldots, e_m$ be the standard basis of $\mathbb{R}^m$. Then for each $j$ and every $v \in \T_{\chi(\P)}^* \Psi$, $\langle v, A_M e_j \rangle = \langle v, \dot{\chi}_\P (s_j)\rangle$. By the definition of the efficient influence function, $\langle v, \dot{\chi}_\P (s_j)\rangle=\E[\IF_\chi(v) s_j]$. So the $j$-th column of $A_M$ is exactly the coefficient of the orthogonal projection of $\IF_\chi$ onto the basis vector $s_j$. In other words, $A_M$ is the coefficient map of the $\L^2(\P)$-projection of $\IF_\chi$ onto $M$. Consequently, $A_M A_M^*=\E\{(\mathsf{\Pi}_M \IF_\chi)^{\otimes 2}\}$, where $\mathsf{\Pi}_M$ denotes the $\L^2(\P)$-projection onto $M$. 

Now apply the parametric local asymptotic minimax theorem from Section~\ref{sec:par} to the finite-dimensional local family $\left\{Q_{n, h}: h \in \mathbb{R}^m\right\}$, with local parameter $h \mapsto \chi_n(h)$. In Section~\ref{sec:par}, the parametric theorem is already formulated in covariance forms, so it matches the current formulation of Theorem~\ref{thm:semipar LAM}. Since $\{Q_{n, h}\}$ converges to the Gaussian shift experiment $\{\calN \left(h, I_m\right)\}$, and since the local parameter satisfies $\sqrt{n} \Exp_{\chi(\P)}^{-1} \chi_n(h) \rightarrow A_M h$, the finite-dimensional parametric minimax theorem yields that for every positive quadratic form $\ell$ on $\T_{\chi(\P)} \Psi$, 
\begin{equation*}
\begin{split}
& \sup_{I \subset M} \liminf_{n \rightarrow \infty} \sup _{s \in I} \E_{\P_{n, s}} \ell \Big(\Pi_{\chi (\P_{1 / \sqrt{n}, s})}^{\chi(\P)} \sqrt{n} \Exp_{\chi (\P_{1 / \sqrt{n}, s})}^{-1} \hat{\chi}_n \Big) \\
& \geq \int \ell \diff \calN (0, A_M A_M^*).
\end{split}
\end{equation*}
Because this inequality holds for every positive quadratic form $\ell$, it is equivalent to the matrix inequality
\begin{equation*}
\begin{split}
& \sup_{I \subset M} \liminf _{n \to \infty} \sup _{s \in I} \E_{\P_{n, s}} \Big\{\Pi_{\chi (\P_{1 / \sqrt{n}, s})}^{\chi(\P)} \sqrt{n} \Exp_{\chi (\P_{1 / \sqrt{n}, s})}^{-1} \hat{\chi}_n \Big\}^{\otimes 2} \\
& \succeq A_M A_M^*.
\end{split}
\end{equation*}

At this point, the lower bound depends on the particular finite-dimensional subspace $M$. To recover the full semiparametric bound, let $M$ increase. Since $\IF_\chi$ belongs to the closed linear span of $\Lambda_\P$, there exists an increasing sequence of finite-dimensional subspaces $M_r \subset \Lambda_\P$ such that $\mathsf{\Pi}_{M_r} \IF_\chi \rightarrow \IF_\chi $ in $\L^2(\P)$. Hence $A_{M_r} A_{M_r}^*=\E\left[\left(\mathsf{\Pi}_{M_r} \IF_\chi\right)^{\otimes 2}\right] \rightarrow \E\left[\IF_\chi^{\otimes 2}\right]=\bbV_\P$. 

Applying the preceding matrix inequality to each $M_r$ and letting $r \rightarrow \infty$, we obtain
\begin{align*}
& \sup_{I} \liminf_{n \rightarrow \infty}\sup_{s\in I} \E_{\P_{n, s}} \left\{ \Pi_{\chi (\P_{1 / \sqrt{n}, s})}^{\chi (\P)} \sqrt{n} \Exp_{\chi (\P_{1 / \sqrt{n}, s})}^{-1} \hat{\chi}_{n} \right\}^{\otimes 2} \\
& \succeq \E (Z^{\otimes 2}),
\end{align*}
where $Z \sim \calN (0, \bbV_{\P})$.
\end{proof}
\subsection{Omitted Details in Section~\ref{sec:examples}}
\label{app:example}
\subsubsection{Omitted Details in Section~\ref{sec:M-estimation}}
\label{app:Frechet mean}
\leavevmode
In this section, we first show that under the assumptions of Theorem~\ref{thm:Frechet mean}, the \Frechet{} mean is differentiable in the sense of Definition~\ref{def:differentiable functionals}.
\begin{proof}
The \Frechet{} mean can be equivalently defined via the following moment equation, when the squared geodesic distance $\phi \coloneqq \frac{1}{2} d^{2}$ admits both first and second derivatives, denoted by $\phi^{(1)}$ and $\phi^{(2)}$:
\begin{equation}
\label{Frechet equality constraint}
\E \{\phi^{(1)} (Z; \mu_{0})\} \equiv 0,
\end{equation}
which implies that for any parametric submodel $\P_{t}$ of $\P$ satisfying \eqref{semipar: differentiable paths},
\begin{equation}
\label{Frechet equality constraint submodel}
\begin{split}
& \ \left. \frac{\diff}{\diff t} \right\vert_{t = 0} \E_{t} \left[ \phi^{(1)} (X; \mu_{t}) \right] \equiv 0 \Rightarrow \\
& \ \E \left[ \phi^{(2)} (X; \mu_{0}) \right] \cdot \dot{\mu}_{0} = \E \left[ \phi^{(1)} (X; \mu) \s (X) \right] \Rightarrow \\
& \ \dot{\mu}_{0} = \E \left[ \left\{ \E \left[ \phi^{(2)} (X; \mu_{0}) \right] \right\}^{-1} \cdot \phi^{(1)} (X; \mu_{0}) \cdot \s (X) \right].
\end{split}
\end{equation}
Since under a nonparametric model, the model tangent space $\Lambda_{\P}$ equals $L^{2}_{0} (\P)$, the above calculation identifies the differentiability of $\mu$ and its efficient influence function
\begin{align*}
\chi_{\P, b} = \{\E [\phi^{(2)} (X; \mu_{0})]\}^{-1} \cdot \phi^{(1)} (X; \mu_{0})
\end{align*}
for any $b \in \T_{\mu}^{\ast} \bbM$. Finally, we have the identity $\phi^{(1)} (X; \mu_{0}) = \Exp_{\mu_{0}}^{-1} X$.
\end{proof}
Next, we prove the statement in Remark~\ref{ex:non-asymptotic} in the end of Section~\ref{sec:Frechet mean}.
\begin{proof}
In a normal coordinate system centered at $\mu_0$, the Christoffel symbols and their derivatives vanish at $\mu_0$ so that the standard differential corresponds to the Riemannian gradient and the standard second order derivative corresponds to the Riemannian Hessian. Consider $d^2(y, x)$, for $y \notin \calC (x)$, we have $\nabla d^2(y, x)=-2 \Exp^{-1}_x (y)$ and the Hessian is $H_x(y)=\nabla^2 d^2(y, x)=$ $-2 D_x \Exp^{-1}_x (y)$. 
As illustrated in Section 6 of \cite{pennec2019curvature} and Theorem 2.2 of \cite{bhattacharya2005large}, asymptotic covariance of the empirical Riemannian mean is considered. 

In Theorem 2.2 of \cite{bhattacharya2005large}, the covariance of the corresponding Gaussian distribution is $\Lambda^{-1} \Sigma \Lambda^{-1}$, where $\Lambda=\int_{\bbM} H_{\mu_0}(x) \mu(\diff x)$ and $\Sigma=4 \int_{\bbM} \Exp_{\mu_0}^{-1} (x) \Exp^{-1}_{\mu_0} (x)^{\top} \mu (\diff x)$. 
By comparison, $\Sigma$ is equivalent to our $4\E\{ [\Exp_{\mu_{0}}^{-1} X]^{\otimes 2}\}$. For the second order derivatives, we have $H_{\mu_0}(y)$ and its expected value for $y$ following the law $\P_{\mu}$ is $-2\E [\nabla \Exp_{\mu_{0}}^{-1} X]$. Based on Theorem \ref{thm:Frechet mean}, the empirical \Frechet{} mean $\hat{\mu}_{n}$ is semiparametric efficient.
\end{proof}
\subsection{Robustness of the Results to an Alternative Geometric Formulation based on Retraction}
\label{app:retraction}
As pointed out by a referee, due to computational concerns, \emph{retraction} is also used very often as an alternative to the exponential map when comparing two different points on a manifold. In a nutshell, our results should be robust to replacing $\Exp$ by a sufficiently smooth local retraction $R$, provided that $R$ agrees with $\Exp$ up to the second order at the origin.
Specifically, let $R_x: \T \bbM \rightarrow \bbM$ be a second-order retraction with a local inverse $R_x^{-1}$ locally at $x \in \bbM$ \cite{absil2008optimization, huang2026high}. Then, the main first-order asymptotic results should remain unchanged after systematically replacing $\Exp$ / $\Exp^{-1}$ by $R / R^{-1}$, together with replacing parallel transport by compatible vector transport. 

What may change, however, are certain exact finite-sample geometric identities and higher-order terms. This interpretation is consistent with the fact that our central results are all formulated as first order asymptotics: since local perturbations occur at scale $n^{-1 / 2}$, estimation error is normalized by $\sqrt{n}$, and curvature-dependent corrections vanish in the first-order limit.

To further explain why the first-order asymptotic theory is essentially robust to this change, suppose $R_x$ is a $C^2$ second-order retraction. Then, in local coordinates around $x$, $R_x (v) = \Exp_x (v) + O (\|v\|^3)$ for $v = o (1)$. Hence, along local alternatives of the form $v=h / \sqrt{n}$, $R_\theta (h / \sqrt{n}) = \Exp_{\theta} (h / \sqrt{n}) + O (n^{-3 / 2})$. The discrepancy between the retraction and exponential map is $o (n^{-1 / 2})$. Therefore, replacing the local perturbation $\theta_{n, h} = \Exp_\theta(h / \sqrt{n})$ in Definition~\ref{def:par dqm} and Proposition~\ref{prop:lan} by $\theta_{n, h}^{R} = R_{\theta} (h / \sqrt{n})$ should not affect the score, Fisher information, or the LAN limiting experiment, because both paths have the same initial velocity $h$ and differ only at higher order terms. As for the intrinsic differential $\dot{\psi}(\theta)$, in the manuscript, it is defined as $$\dot{\psi} (\theta) (h) = \lim_{t \rightarrow 0} t^{-1} \Exp^{-1}_{\psi (\theta)} \psi (\Exp_{\theta} (t \cdot h)).$$ 

If one instead defines $\dot{\psi}_R(\theta) (h) = \lim_{t \rightarrow 0} t^{-1} R_{\psi(\theta)}^{-1} \{\psi (R_\theta(t h))\}$, then $\dot{\psi}_R(\theta)=\dot{\psi} (\theta)$ under the same local smoothness assumptions, because $R_\theta(t h)$ and $\Exp_\theta(t h)$ share the same first derivative at $t=0$, while $R_{\psi(\theta)}^{-1}$ and $\Exp_{\psi(\theta)}^{-1}$ agree up to first order near the base point. Likewise, if $\hat{\psi}_n$ is asymptotic linear, then $R_{\psi (\theta)}^{-1} (\hat{\psi}_n) = \Exp_{\psi(\theta)}^{-1} (\hat{\psi}_n) + o_{\P} (n^{-1})$, and therefore $\sqrt{n} R_{\psi(\theta)}^{-1} (\hat{\psi}_n) = \sqrt{n} \Exp_{\psi(\theta)}^{-1} (\hat{\psi}_n) + o_{\P} (1)$. So any limiting distribution stated in terms of $\sqrt{n} \Exp_{\psi(\theta)}^{-1} \hat{\psi}_n$ still holds if we replace $\Exp^{-1}$ by $R^{-1}$.

A similar comment applies to the use of parallel transport in Definition~\ref{def:riemann regular}. If one adopts a vector transport $\mathcal{T}$ associated with the retraction $R$ instead of the parallel transport, the same role can be played by $\mathcal{T}$, provided it agrees with parallel transport to the first order along the corresponding local curve. Let $R$ be a $C^2$ retraction, and let $y=R_x(v)$ for $v \in \T_x \bbM$ sufficiently small, a natural vector transport associated with $R$ is the differentiated retraction $\mathcal{T}^{x}_{y} (\xi):= \diff R_x(v)[\xi]$, $\xi \in \T_x \bbM$. Since $\diff R_x(0)=\operatorname{id}$, the map $\diff R_x(v)$ is invertible for $v$ near 0 , so we may also define the backward transport $\mathcal{T}^{y}_{x}:=\left(\diff R_x(v)\right)^{-1}: \T_y \bbM \rightarrow \T_x \bbM$. In embedded manifolds, one may equivalently use the usual projection-based vector transport associated with the chosen retraction. 

For our first-order asymptotic purposes, it is enough to assume that along local perturbations $y \rightarrow x$, $\left\|\mathcal{T}^{y}_{x} -\Pi^{y}_{x}\right\|_{\mathrm{op}}=O(d(x, y))$, uniformly for $v$ in bounded sets. Since in our setting $d(x, y)=O\left(n^{-1 / 2}\right)$ under the local alternatives $\theta_{n, h}=\Exp_\theta(h / \sqrt{n})$, this implies $\left\|\mathcal{T}^{y}_{x}-\Pi^{y}_{x}\right\|_{\mathrm{op}}=O\left(n^{-1 / 2}\right)=o(1)$. Therefore, when applied to normalized residuals that are $O_\P(1)$, the difference between transporting by $\mathcal{T}$ and by $\Pi$ is $o_\P(1)$. Under such a compatibility condition, the difference between transporting via $\mathcal{T}$ and transporting via $\Pi$ is again $o_{\P} (1)$ at the $n^{1 / 2}$-scale, so the definition of regularity and the resulting asymptotic conclusions continue to hold. In other words, parallel transport is used in our paper as a canonical comparison device between nearby tangent spaces; for first-order asymptotics, any first-order equivalent transport mechanism should lead to the same asymptotic theory. This is also in line with Remark~\ref{rem:nonlinear}, where we note that the notion of regularity can be formulated more broadly using tangent-space-type structures beyond the exact manifold machinery adopted in the present paper.

Finally, we point out other specific results in our paper that do exploit properties specific to $\Exp$, which mostly concern some basic technical results from geometric analysis, in particular those in Appendix~\ref{app:geometric analysis}.
\begin{itemize}
    \item Lemma~\ref{lem:connection}: derivatives of $\Exp$ and $\Exp^{-1}$. This lemma relies on the expansions $\nabla_{h} \Exp_{\mu} h (\cdot) = \id (\cdot) - \frac{1}{6} \calR_{\mu} (\cdot, h) h + O (\Vert h \Vert^{3})$, and $\nabla_{h} \Exp^{-1}_{\mu} \mu_{h} (\cdot) = \id (\cdot) + \frac{1}{6} \calR_{\mu} (\cdot, h) h + O (\Vert h \Vert^{3})$, and its proof explicitly introduces a family of geodesics and the associated Jacobi field, which solves the Jacobi equation. The curvature term $- \frac{1}{6} \calR_{\mu} (\cdot, h) h$ comes directly from the Jacobi equation. That said, one does not necessarily need this exact formula. What is needed, if retraction is used instead, is an approximation with the same order of remainder. A second-order retraction can supply such a replacement at the level of asymptotic order, although the explicit curvature coefficient may differ unless one works with the exact exponential map. 
    
    \quad Concretely, if $R_\mu$ is a $C^3$ second-order retraction, then in local normal coordinates at $\mu$, $R_\mu(h)=\Exp_\mu(h)+O\left(\|h\|^3\right)$, $\diff R_\mu(h)=\diff \Exp_\mu(h)+O\left(\|h\|^2\right)$, as $h \rightarrow 0$. Hence $\diff R_\mu(h)(\cdot)=\id(\cdot)+O\left(\|h\|^2\right)$, and similarly for the local inverse, $\diff \left(R_\mu^{-1}\right)_{\mu_h}(\cdot)=\id(\cdot)+O\left(\|h\|^2\right)$, and $\mu_h:=R_\mu(h)$. Thus, although the exact second-order term need not be the curvature expression coming from the Jacobi equation, the remainder is of the same order as for $\Exp$. This is sufficient for our purposes, because throughout the paper the local perturbations satisfy $h=O\left(n^{-1 / 2}\right)$, so after multiplying by $\sqrt{n}$ the discrepancy contributes only $\sqrt{n} O\left(\|h\|^2\right)=O\left(n^{-1 / 2}\right)=o(1)$.
    \item Parallel transport expansion in Lemma~\ref{lem:parallel expansion} uses Taylor expansion together with the vanishing of the covariant derivative along the curve. This is therefore another place where one uses a structure specific to Levi-Civita parallel transport. However, the exact geodesic parallel-transport identity is stronger than what the first-order asymptotic argument ultimately needs. A sufficiently accurate vector transport should be enough to establish similar conclusions, even though the proof as written uses the exact parallel transport formula. 
    
    \quad In particular, the vector transport does not need to agree with parallel transport to second order. What is needed is only that the vector transport be first-order compatible with parallel transport along the $n^{-1 / 2}$-local perturbations used in Definition~\ref{def:riemann regular}. 
    
    \quad More precisely, let $\mathcal{T}_{x}^{y}: \T_y \Psi \rightarrow \T_x \Psi$ denote a vector transport used in place of $\Pi_x^y$. For local perturbations $y = \psi (\theta_{n, h})$ and $x = \psi(\theta)$, the key requirement is $\|\mathcal{T}_{x}^{y} - \Pi_{x}^{y}\|_{\mathrm{op}} = o(1)$ when $d (x, y) = o (1)$, uniformly for $h$ in bounded subsets of $\T_\theta \Theta$.
    This is the minimal requirement needed to preserve our results because $d (x, y) = O (n^{-1 / 2})$. To see this, write $U_{n, h} \coloneqq \sqrt{n} \Exp_{\psi (\theta_{n, h})}^{-1} (\hat{\psi}_n) \in \T_{\psi (\theta_{n, h})} \Psi$. The original regularity definition in the paper is stated in terms of $\Pi U_{n, h}$. If we replace $\Pi$ by $\mathcal{T}$, then $\mathcal{T}_{x}^{y} U_{n, h}-\Pi_{x}^{y} U_{n, h} = (\mathcal{T}_{x}^{y}-\Pi_{x}^{y}) U_{n, h}$. 
    Therefore $\|(\mathcal{T}_{x}^{y}-\Pi_{x}^{y}) U_{n, h}\| \leq \|\mathcal{T}_{x}^{y} - \Pi_{x}^{y}\|_{\mathrm{op}} \|U_{n, h}\| = o (1) \cdot O_{\P} (1) = o_{\P} (1)$ under matching order between $\mathcal{T}_{x}^{y}$ and $\Pi_{x}^{y}$, which is what we need. So any vector transport whose operator-norm difference from parallel transport is $o (1)$ along local perturbations yields the same limit law in Definition~\ref{def:riemann regular}. 
\end{itemize}
Overall, as we have seen, it is not impossible to use a second-order retraction and its associated vector transport to recover our theory. However, it becomes quite inconvenient to do so, given that we have already had a set of useful geometric tools that can be adopted, such as exponential maps and parallel transports. 

Finally, as mentioned in the Introduction, it is also possible to formalize the asymptotic efficiency theory without referring to any particular choice of coordinate charts. From that point of view, one needs to take supremum over all possible charts for all the results established here. A related idea can also be found in \cite{takatsu2024generalized}. But it remains to be studied if this can be done rigorously in the context of regular parameter manifolds.

We hope that our analyses above can also shed some light on how retractions/vector transports should be designed in practice, with the hope of obtaining good statistical properties in the sense of being close to the type of efficiency bound theory established here.

\end{appendices}

\vskip 0pt plus -1fil

\begin{IEEEbiographynophoto}{Lvfang Sun}
received a Ph.D. in Statistics in 2026 from the National University of Singapore, Singapore. His research interests include non-Euclidean data analysis, statistical and deep learning, and semiparametric theory.
\end{IEEEbiographynophoto}

\vskip 0pt plus -1fil

\begin{IEEEbiographynophoto}{Zhenhua Lin}
received a Bachelor of Science degree in 2008 from Fudan University, Shanghai, China. He subsequently earned Master of Science degrees in computing science and statistics, followed by a Ph.D. in statistics from the University of Toronto in 2017. He is currently an Associate Professor and Dean's Chair Professor at the National University of Singapore, Singapore. His research focuses on statistical methodology and theory, with applications in machine learning and artificial intelligence.
\end{IEEEbiographynophoto}

\vskip 0pt plus -1fil

\begin{IEEEbiographynophoto}{Lin Liu}
received a Bachelor of Engineering degree majored in Bioinformatics in 2011 from Tongji University, Shanghai, China. Then he received a Master of Science degree in Biostatistics in 2013 and a Ph.D. degree in Biostatistics in 2018, both from Harvard University, Cambridge, Massachusetts, United States. He is currently an assistant professor in the Institute of Natural Sciences, School of Mathematical Sciences, and SJTU-Yale Joint Center for Biostatistics and Data Science at Shanghai Jiao Tong University, Shanghai, China. His research interests include semi- and non-parametric statistical theory and causal inference, with applications in machine learning and artificial intelligence, biomedical sciences, and physical sciences.
\end{IEEEbiographynophoto}

\end{document}